\documentclass[]{amsart}

\usepackage{amsmath, amssymb,amsthm}
\usepackage{mathtools}

\usepackage{hyperref}
\usepackage{pgf,tikz}
\usepackage{amsaddr}

\usepackage{caption}

\newtheorem{lemma}{Lemma}[subsection]
\newtheorem{theorem}[lemma]{Theorem}
\newtheorem{proposition}[lemma]{Proposition}

\newtheorem*{thm*}{Theorem} 

\theoremstyle{definition}
\newtheorem{definition}[lemma]{Definition}
\newtheorem{remark}[lemma]{Remark}
\newtheorem{example}[lemma]{Example}
\newtheorem{assumption}[lemma]{Assumption}

\newtheorem{construction}[lemma]{Construction}

\DeclareMathOperator{\End}{End}
\DeclareMathOperator{\Aut}{Aut}
\DeclareMathOperator{\ad}{ad}
\DeclareMathOperator{\InDer}{InDer}

\DeclareMathOperator{\Lie}{Lie}
\DeclareMathOperator{\Hom}{Hom}

\title{Cubic norm pairs and hermitian cubic norm structures}

\author{Michiel Smet}
\address{Department of Electronics and Information Systems\\ Faculty of Engineering and
	Architecture, Ghent University}
\email{Michiel.Smet@UGent.be}



\begin{document}

	\maketitle
	
	\begin{abstract}
		
		We generalize cubic norm structures to cubic norm pairs and extend hermitian cubic norm structures to arbitrary commutative unital rings. For the associated ``skew dimension one structurable algebra" of these pairs, we construct a corresponding Lie algebra and a group of automorphisms of the Lie algebra.
	\end{abstract}

	\section{Introduction}
	
	Structurable algebras, as introduced by Allison \cite{allison1978class}, play an important role in the construction of (exceptional) simple Lie algebras, see e.g. \cite{Allison1979ModelsOI}. However, these algebras are only defined over fields of characteristic different from $2$ and $3$. We remark that a generalization of structurable algebras to arbitrary rings has already been considered in \cite{Allison1993NonassociativeCA}, but in the context of a different construction of exceptional Lie algebras than the one we are interested in. The main thing we do, is placing the subclass of skew dimension one structurable algebras in a broader class of algebras coming from hermitian cubic norm structures over arbitrary (commutative, unital) rings. 
	This is a natural thing to do, since over fields of characteristic different from $2$ and $3$, each simple skew dimension one structurable algebra can be obtained from a hermitian cubic norm structure. This was proven in \cite{Tom2019} for the main class of skew dimension one structurable algebras, i.e., forms of matrix structurable algebras with nontrivial norms \cite[Example 1.9.(a)]{All90}.
	Our notion of a hermitian cubic norm structure differs slightly from the one of \cite{Tom2019}, even when restricted to fields, in order to also allow forms of matrix structurable algebras with trivial norms \cite[Example 1.9.(b)]{All90}.
	
	To be precise, given a hermitian cubic norm structure over a commutative unital ring $R$ we introduce a ``structurable" algebra $\mathcal{B}$, using the construction from \cite{Tom2019}. Next, we use the construction of the Lie algebra $\mathcal{L}(\mathcal{B})$ from \cite{Allison1979ModelsOI} for the structurable algebra $\mathcal{B}$. To simplify the verification that a Lie algebra is obtained, we use a slightly different framing of the construction. This Lie algebra carries a $\mathbb{Z}$-grading of a specific form. Specifically, one can decompose $\mathcal{L}(\mathcal{B}) = L_{-2} \oplus L_{-1} \oplus L_0 \oplus L_1 \oplus L_2$ and set $L_k = 0$ for all $k \notin \{-2,-1,0,1,2\}$. That the algebra is graded, corresponds to the condition $[L_i,L_j] \subset L_{i + j}$ for all $i,j \in \mathbb{Z}$. 
	
	This Lie algebra on its own is relatively easy to construct. The challenging part, which will be the focus of our efforts, is constructing a certain group of automorphisms of this Lie algebra.
	The automorphism groups we construct are compatible with scalar extensions.

\begin{thm*}[A]
		Suppose that $J$ is a well behaved hermitian cubic norm structure with associated structurable algebra $\mathcal{B}$ and Lie algebra $\mathcal{L}(\mathcal{B})$.
		There exist functors $G^\pm$ mapping associative, commutative, unital $R$-algebras $K$ to \[G^\pm(K) \subset \Aut_K(\mathcal{L}(\mathcal{B}) \otimes K)\] such that
		\[ \mathcal{L}(\mathcal{B}) \otimes K \cong \Lie(G^-)(K) \oplus L_0 \otimes K \oplus \Lie(G^+)(K)\]
		with $L_0$ the $0$-graded component of $\mathcal{L}(\mathcal{B})$.
		
		Moreover, from the subgroup $G(K)$ generated by $G^\pm(K)$ of $\Aut(\mathcal{L}(\mathcal{B}) \otimes K)$ we can recover $G^\sigma(K)$, for $\sigma = \pm$, using
		\[ G^{\sigma}(K) = \{ 1 + g_1 + g_2 + g_3 + g_4 \in G(K) : g_i \text{ acts as } \sigma i \text{ on the grading}\}.\]
		Furthermore, $[\Lie(G)(R), \Lie(G)(R)] \cong \mathcal{L}(\mathcal{B})$. 
		\begin{proof}
			This is Theorem \ref{thm: hcns lie alg}.
		\end{proof}
	\end{thm*}

This group $G$ generated by $G^+$ and $G^-$, or more precisely its closure in $\Aut(\mathcal{L}(\mathcal{B}))$ endowed with Zariski topology, is a smooth connected algebraic group whenever $R$ is a field and $\mathcal{L}(\mathcal{B})$ is finite dimensional.
Moreover, when this $G$ is reductive, the groups $G^+$ and $G^-$ are the unipotent radicals of a pair of opposite parabolic groups.

	Since directly constructing the group $G$ for a general hermitian cubic norm structure is difficult, or even infeasible, we first look at a specific subclass for which we can perform this construction with relative ease. Namely, first we look at split hermitian cubic norm systems, or more precisely, cubic norm pairs. For these, the Lie algebra is $\mathbb{Z}^2$-graded instead of simple $\mathbb{Z}$-graded, which makes the construction of $G$ and proving properties of $G$ a lot easier. The nonzero components in this $\mathbb{Z}^2$-grading form a $G_2$ root system. The main theorem for cubic norm pairs is Theorem \ref{thm: cnp main}.
	
	\subsection*{Cubic norm pairs}
		Cubic norm structures were originally introduced by McCrimmon \cite{McC69} over arbitrary (commutative, unital) rings. They generalized certain algebraic structures used by Freudenthal to construct exceptional Lie algebras of type $E_7$ and $E_8$ \cite{Freu54}. They also extend the class of cubic forms that lead to Jordan algebras over fields of characteristic different from $2$ and $3$, as characterized by Springer \cite{Spri62}. These cubic norm structures are quite well studied; for example, \cite{Petersson86,Petersson86_2} leads to a classification of them over fields. A nice summary of the classification is \cite[Theorem 39.6, Theorem 46.8]{Skip2024}.
	
	We introduce cubic norm pairs, a generalization of cubic norm structures. These pairs already appeared under a weaker axiomatization in \cite{FLK01}. They can be thought of as cubic norm structures without a unit. A second way to see cubic norm pairs, is as split hermitian cubic norm structures, i.e., the ones for which the construction of the Lie algebra and associated group are relatively easy and done explicitly in this article.
	A first reason to use the generalization is their characterizing property, namely having a certain divided power representation (as defined by Faulkner for Jordan pairs \cite{faulkner2000jordan}) involving the cubic norms. These divided power representations are precisely what enable us to define the generating automorphisms of $G$ with relative ease.
	This is also the main reason the pairs are indispensable, as the whole class of cubic norm pairs plays a role in proving the main theorem that cannot be fulfilled by the class of cubic norm structures.
 
	We remark that the group $G$ is quite well studied, whenever $G$ is simple and the cubic norm pair corresponds to a cubic norm structure over a field. Namely, in that case, one has a Tits hexagon, as studied in \cite{muhlherr2022}.
	Moreover, for certain simple Lie algebras over arbitrary fields that correspond to this setting, the group $G$ was also reconstructed in \cite{TOM24}.

	\subsection*{Hermitian cubic norm structures}
	Hermitian cubic norm structures were already introduced and studied over fields of characteristic different from $2$ and $3$ by De Medts \cite{Tom2019} and Allison, Faulkner and Yoshii \cite{All08}. 
	
	We expand the definition of hermitian cubic norm structures so that it works over arbitrary commutative, unital base rings. We also make the one-to-one correspondence between split hermitian cubic norm structures and cubic norm pairs precise, and show that each hermitian cubic norm pair splits after a faithfully flat scalar extension.
	Originally \cite{Tom2019}, split hermitian cubic norm structures correspond to a pairing of certain isotopes of cubic norm structures.
	 Our notion of a hermitian cubic norm structure is, thus, slightly more general than the one of De Medts \cite{Tom2019}, since we allow, for example, trivial norms.
	
	\subsection*{Operator Kantor pairs}
	Operator Kantor pairs \cite{OKP} were introduced to get some grip on pairs of group functors $G^+$ and $G^-$ of the form we obtain here. 
	Proving the main theorem required for us to prove a Lemma \ref{lem: oij}, which basically amounts to proving that we have an operator Kantor pair, cfr., Theorem \ref{thm: cns defines okp}.
	As such, we were also interested in operator Kantor pairs related to hermitian cubic norm structures and cubic norm pairs.
	
	On the one hand, we characterize operator Kantor pairs that have a $G_2$-grading of various sorts.
	This allows us to characterize operator Kantor pairs over $R$ coming from cubic norm pairs over $R$.
	We also examine what happens if an operator Kantor pair is $G_2$-graded in the sense of \cite{Torben}, and see that it implies, as long as there is no $3$-torsion, that we have an operator Kantor pair arising from a cubic norm structure over an $R$-algebra $K$.
	
	We also determine the operator Kantor pair operators in terms of the hermitian cubic norm structure operators. Finally, we will show that all forms of operator Kantor pairs that can be constructed from hermitian cubic norm structures come, in fact, from hermitian cubic norm structures (provided they contain $1$).

	\subsection*{Outline}
	In the second section, we define the notion of a cubic norm pair, the corresponding Jordan pair, and a characterizing divided power representation.
	
	In the third section, we construct a structurable algebra from a cubic norm pair, along with the Lie algebra associated to this structurable algebra, and certain automorphisms of this Lie algebra.
	We also investigate how these automorphisms interact and obtain a version of Theorem (A) for cubic norm pairs.
	
	In the fourth section, we introduce the concept of a hermitian cubic norm structure. We relate these structures to cubic norm pairs and use this relation to work towards the main theorem for hermitian cubic norm structures.
	
	In the fifth section, we construct an operator Kantor pair corresponding to the Lie algebra and group functor. We characterize the operator Kantor pairs one can obtain from cubic norm pairs. We also see that operator Kantor pairs that can be defined from a hermitian cubic norm structure over a faithfully flat scalar extension, are defined from hermitian cubic norm structures over the base ring.
	These constitute two very important recognition theorems for operator Kantor pairs.
	We also compute the main operator Kantor pair operators in terms of the hermitian cubic norm structure and compare them with the operators proposed in \cite{OKP}.
	
	We finish with a brief appendix relating some simpler definitions we use in the article to the more general versions employed in \cite{OKP}.

	\subsection*{Notation} We work over a base ring $R$ which is assumed to be unital and commutative. The category $R\textbf{-alg}$ is the category of unital, commutative, associative $R$-algebras.
	When we write $A \otimes B$, we mean $A \otimes_R B$.

	\section{Cubic norm pairs}
	
	\label{sec: cnp}
	
	\subsection{Cubic norm pairs}
	
	In this subsection, we first define what we mean by homogeneous maps, and in particular cubic maps. 
	Secondly, we give the definition of a cubic norm pair and derive some basic properties that allow us to characterize the cubic norm pairs coming from cubic norm structures.
	
	\begin{definition}
		Consider $R$-modules $J_1$ and $J_2$. We consider the functors $R\textbf{-alg} \longrightarrow \textbf{Set}$ given by $K \mapsto J_i \otimes_R K$. Following \cite[12.7]{Skip2024}, we call a natural transformation \[f : J_1 \otimes \cdot \longrightarrow J_2 \otimes \cdot\] \textit{homogeneous of degree} $i$ if
		\[ f^K(\lambda j) = \lambda^i f^K(j)\]
		for all $\lambda \in K$, $j \in J_1 \otimes K$, and all $K \in R\textbf{-alg}$.
		
	\end{definition}

\begin{remark}
	Homogeneous maps are very often simply homogeneous polynomial maps. Suppose that $J_1 \cong R^n$ and $J_2 \cong M$. For a homogeneous map $f : R^n \longrightarrow M$, define $g(x_1,\dots,x_n) = f^{R[x_1,\dots,x_n]}(x_1 e_1 + \dots x_n e_n)$. By applying an evaluation map $R[x_1,\dots,x_n] \longrightarrow K$, we see that $f^K(x_1 e_1 + \dots + x_n e_n) = g(x_1,\dots,x_n)$ for arbitrary $x_i \in K$.
\end{remark}

	\begin{definition}
		Suppose that $f$ is homogeneous of degree $i$.
		There exist natural transformations
		\[ f^{(k,l)} : (J_1 \otimes \cdot \times J_1 \otimes \cdot  ) \longrightarrow J_2 \otimes \cdot, \]
		called the \textit{linearisations} of $f$, such that $f^{(k,l)}$
		is homogeneous of degree $k$ in the first component, of degree $l$ in the second component,
		and
		\[ f^{K}(\lambda u + \mu v) = \sum_{k + l = i} \lambda^i \mu^j {f^{(k,l)}}^{K}(u,v).\]
		These linearisations can be determined over $K[\lambda,\mu]$. See for example \cite[12.10.(c)]{Skip2024}. Setting $\mu = 0$ shows that ${f^{(i,0)}}^K(u,v) = f^K(u)$.
	\end{definition}

	\begin{remark}		
		One can linearize $f^{(k,l)}$ to obtain $f^{((a,b),l)}$ and $f^{(k,(c,d))}$ by linearising the first and second argument further. We note that \[{f^{((a,b),c)}}^K(u,v,w)= {f^{(a,(b,c))}}^K(u,v,w)\] as both are the terms with coefficient $\lambda^a \mu^b \nu^c$ in $f^K(\lambda u + \mu v + \nu w)$. Hence, we can write $f^{(a,b,c)}$ for $f^{((a,b),c)}$.
		Going further, one can obtain arbitrary linearisations $f^{(a_1,\dots, a_n)}$.
		The same argument also shows that $f^{(a_1,\dots,a_n)}$ is symmetric in its arguments.
		
		We call a homogeneous map of degree $2$ \textit{quadratic} and a homogeneous map of degree $3$ \textit{cubic}.
		Note that if $q$ is quadratic, \[(u,v) \mapsto q(u + v) - q(u) - q(v) = q^{(1,1)}(u,v)\] is bilinear. This linearisation will sometimes be called the \textit{polarization} of the quadratic form.
	\end{remark}

	\begin{definition}
		\label{def: cubic norm pair}
		Consider a commutative unital ring $R$, a $R$-module $J$, and a $R$-module $J'$.
		Suppose that we have cubic maps $N, N' : J, J' \longrightarrow R$, quadratic maps $\sharp : J \longrightarrow J',$ $\sharp' : J'\longrightarrow J$, and a bilinear map $T : J \times J' \longrightarrow R$ and write $T'(a,b) = T(b,a)$ for $T' : J'\times J \longrightarrow R$.
		We use $(u,v) \mapsto u \times v$ to denote the polarization of $\sharp$ and use $u \times' v$ for the polarization of $\sharp'$.
		We call $(J,J',N,N',\sharp, \sharp', T, T')$ a \textit{cubic norm pair} over $R$ if
		\begin{enumerate}
			\item $T(a,b^\sharp) = N^{(1,2)}(a,b)$
			\item $(a^\sharp)^{\sharp'} = N(a)a$,
			\item $(a^\sharp \times' c) \times a = N(a)c + T(a,c) a^\sharp$,
			\item $N(Q_a c) = N(a)^2 N'(c)$ for $Q_a c = T(a,c)a - a^\sharp \times' c$
		\end{enumerate}
	hold for all $a, b \in J \otimes K$ and $c \in J' \otimes K$ and if the same equations with the roles of $J$ and $J'$ reversed, hold as well.
	We call a cubic norm pair \textit{unital} if there exists a $1 \in J$ such that $N(1) = 1$.
	\end{definition}

	\begin{remark}
		\begin{itemize}
			\item 
			Given the symmetry between $J$ and $J'$ we will often drop the $'$. To be precise, we will just use $N(j)$ to denote $N'(j)$ for $j \in J'$, $j^\sharp$ for $j^{\sharp'}$, and $T(k,j)$ for $T'(k,j)$.
			\item 
			We will often abbreviate the cubic norm pair as $(J,J')$ if the operators are clear from context.
			\item Later, the operator $Q$ shall reappear and we will show that $(J,J',Q)$ forms a quadratic Jordan pair.
		\end{itemize}
	\end{remark}

	\begin{lemma}
		\label{lem: basic properties cnp}
		Suppose that $(J,J')$ is a cubic norm pair. For $a, b, c \in J$ and $t \in R$, the following properties hold:
		\begin{enumerate}
			\item $(t a)^\sharp = t^2 a ^\sharp$,
			\item $N(ta) = t^3 N(a)$,
			\item $T(a \times b, c) = T(a,b \times c)$,
			\item $(a + b)^\sharp = a^\sharp + a \times b + b^\sharp$,
			\item $N(a + b) = N(a) + T(a^\sharp,b) + T(a,b^\sharp) + N(b)$,
			\item $T(a,a^\sharp) = 3 N(a)$,
			\item $a^{\sharp\sharp} = N(a)a$,
			\item $a^\sharp \times (a \times b) = N(a)b + T(a^\sharp,b)a$,
			\item $a^\sharp \times b^\sharp + (a \times b)^\sharp = T(a^\sharp,b)b + T(a,b^\sharp)a$.
		\end{enumerate}
		\begin{proof}
			The first two properties follow from the fact that $\sharp$ is quadratic and $N$ is cubic.
			The fourth, fifth, and sixth property follow from the definition of $\times$ and Definition \ref{def: cubic norm pair}.1.
			The third property follows from $T(a,b \times c)= N^{(1,(1,1))}(a,b,c) = N^{((1,1),1)}(a,b,c) = T(a \times b, c)$ since all four terms are the term with coefficient $\lambda \mu \nu$ in $N(\lambda a + \mu b + \nu c)$. 
			The seventh, eight and ninth property all follow from Definition \ref{def: cubic norm pair}.2 and its linearisations.
			For example, if we take the $(3,1)$-linearisation of $N(a)a = (a^\sharp)^\sharp$, we get
			\[ N(a)b + N^{(2,1)}(a,b) a = a^\sharp \times (a \times b).\]
			Taking the $(2,2)$-linearisation yields the ninth equation.
		\end{proof}
	\end{lemma}

\begin{definition}
	Following McCrimmon \cite{McC69}, we call $(J,N,\sharp,T, 1)$ with $N : J \longrightarrow R$ cubic, $\sharp : J \longrightarrow J$ quadratic with polarization $\times$, symmetric $T : J \times J \longrightarrow J$, and $1 \in J$ a \textit{cubic norm structure} if
	\begin{itemize}
		\item $x^{\sharp \sharp} = N(x)x$,		
		\item $T(x,y^\sharp) = N^{(1,2)}(x,y)$,
		\item $N(1) = 1$, $ 1^\sharp = 1$, $T(1,y)1 - 1 \times y = y$
	\end{itemize}
	hold for all $x$ and $y$ over all scalar extensions. We call $1$ the unit.
	Nowadays, \cite[33.1, 33.4]{Skip2024} one typically assumes $1$ to be unimodular, i.e., there exists $f \in J^*$ such that $f(1) = 1$.
	This assumption is not useful in our setting. We remark that two typical identities one proves using unimodularity such as \cite[33.8, Equations (19), (21)]{Skip2024} are axioms (3) and (4).
\end{definition}

\begin{example}
	For examples, one can look at \cite[Examples 2.3, 2.4, 2.5]{Petersson86}. These are examples of cubic norm structures $(J,N,\sharp,T,1)$.
	For these examples the unit is unimodular, hence axioms (3) and (4) follow as well.
\end{example}

\begin{example}
	The class of cubic norm pairs is broader than the cubic norm structures. For example, we can immediately think of degenerate cases with $\sharp = 0$ and $N = 0$, equipped with an arbitrary bilinear $T$.
	
	A more interesting example is $(I,I)$ with $I$ an ideal of $R$, under $N(i) = i^3,$ $i^\sharp = i^2$ and $T(i,j) = 3ij$.
\end{example}

	\begin{lemma}
	\label{lem: sharp is non linear endo}
	For a cubic norm pair $(J,J')$ we have $(Q_j k)^\sharp = Q_{j^\sharp} k^\sharp$ for $j \in J$, $k \in J'$.
	\begin{proof}
		We compute
		\begin{align*}
			(Q_jk)^\sharp & = (j^\sharp \times k - T(j,k) j)^\sharp \\
			& = (j^\sharp \times k)^\sharp - T(j,k) (j \times (j^\sharp \times k) - T(j,k) j^\sharp) \\
			& = T(j,k) N(j) k + T(j^\sharp,k^\sharp) j^\sharp - j^{\sharp \sharp} \times k^\sharp - T(j,k) (j \times (j^\sharp \times k) - T(j,k) j^\sharp) \\
			& = T(j^\sharp,k^\sharp) j^\sharp - j^{\sharp \sharp} \times k^\sharp \\
			& = Q_{j^\sharp} k^\sharp
		\end{align*}
		where we rewrote $(j^\sharp \times k)^\sharp$ using properties (9) and (7) from Lemma \ref{lem: basic properties cnp} and used Definition \ref{def: cubic norm pair}.3 to let the terms with $T(j,k)$ disappear. 
	\end{proof}
\end{lemma}

	\begin{proposition}
		\label{prop: unit cns}
		Consider a cubic norm pair $(J,J')$ containing an element $a$ for which $N(a)$ is invertible, then the map $Q_a : J' \longrightarrow J$ is a bijection. Moreover, for such $a$, we can define a cubic norm structure $(J_a,N',\sharp', T', a)$ on $J \cong J_a$ with operators
		\begin{itemize}
			\item $N'(x) = N(x)/N(a)$,
			\item $x^{\sharp'} = Q_a x^\sharp / N(a)$,
			\item  $T'(x,y) = T(x, Q_{a^\sharp} y) /N(a)^2$.
		\end{itemize} Moreover $Q_a : J'_{a^\sharp} \longrightarrow J_a$ is an isotopy of cubic norm structures.
		\begin{proof}
			Recall that $Q_a b = T(a,b) a - a^\sharp \times b$.
			First, we will show that $Q_{a^\sharp} Q_a b = N(a)^2 b$ and use this to show that $Q_a$ is a bijection.
			Second, we check that all the axioms for cubic norm structures hold.
			
			We start by computing
			\begin{align*}
				Q_{a^\sharp} Q_a b & = Q_{a^\sharp} (T(a,b) a - a^\sharp \times b) \\
				& = T(a,b) T(a,a^\sharp )a^\sharp - T(a^\sharp, a^\sharp \times b)a^\sharp - N(a)a \times T(a,b)a + N(a)a \times (a^\sharp \times b) \\
				& = 3 N(a) T(a,b)a^\sharp - 2 N(a) T(a,b) a^\sharp - 2 N(a) T(a,b) a^\sharp + N(a)^2 b + N(a) T(a,b) a^\sharp \\
				& = N(a)^2 b,
			\end{align*} 
			using properties (3), (6), and (7) of Lemma \ref{lem: basic properties cnp}, that $(a \times a) = 2a^\sharp$, and Definition \ref{def: cubic norm pair}.3.
			Hence, whenever $N(a)$ is invertible, $Q_a$ has a left inverse. 
			Definition \ref{def: cubic norm pair}.4 implies that $N(a^\sharp)$ is invertible if $N(a)$ is invertible, since 
			\[ N(a)^4 = N(N(a)a) = N(Q_a a^\sharp) = N(a)^2N(a^\sharp).\]
			Therefore, $Q_a$ has a right inverse, namely $Q_{a^\sharp} /N(a)^4$, as well.
			This proves that $Q_a$ is a bijection.
			
			The newly defined $T'$ is symmetric since \[T(x, Q_{a^\sharp} y) = T(x,a^\sharp) T(y,a^\sharp) - T(a, x \times y) N(a).\]
			We note that
			\[ T'(x,y^{\sharp'}) = T(x,y^\sharp)/N(a) = N'^{(1,2)}(x,y)\]
			and
			\[ (x^{\sharp'})^{\sharp'} = Q_a (Q_a x^\sharp)^\sharp / N(a)^3 = Q_a Q_{a^\sharp} N(x)x /N(a)^3 = N'(x) x.\]
			In this cubic norm structure, $a$ fulfills the role of a unit, since \[a^{\sharp'} = Q_{a} a^\sharp/N(a) = T(a,a^\sharp)/N(a) a - a^\sharp \times a^\sharp /N(a) = a\] and since
			\begin{align*}
				N(a)^2 (T'(a,b) a - a \times' b) = & T(a, Q_{a^\sharp} b) a - N(a) Q_a (a \times b) \\
				= & T(a,a^\sharp) T(a^\sharp, b) a - N(a) T(a, a \times b) a \\ & - N(a) T(a,a \times b) a + a^\sharp \times (N(a) a \times b) \\
				= & - T(a^\sharp, b) (a^\sharp)^\sharp  + a^\sharp \times ((a^\sharp)^\sharp) \times b) \\ = & N(a)^2 b.			\end{align*}
			
			Recall that $N(a)^2 = N(a^\sharp)$.
			We verify that $Q_a : J'_{a^\sharp} \longrightarrow J_a$ is an isotopy.
			We use the operators $(N',\sharp', T')$ on $J_a$ and the operators $(N'', \sharp'', T'')$ on $J'_{a^\sharp}$. We compute
			\begin{itemize}
				\item $N'(Q_a x) = N(a)N(x) = N(a)^3 N''(x)$ using 	Definition \ref{def: cubic norm pair}.4,
				\item $(Q_a x)^{\sharp' } = N(a)^3 x^\sharp = N(a) Q_a x^{\sharp ''}$ using Lemma \ref{lem: sharp is non linear endo}, $Q_a Q_{a^\sharp} = N(a)^4 \text{Id}$, and $Q_a x^{\sharp ''} = N(a)^2 x^\sharp$,
				\item  and $T'(Q_a x, Q_a y) = N(a) T(Q_a x, y) = T''(x,y)$.
			\end{itemize} 
			This shows that $J_a \cong J'^{(N(a)a^\sharp)}_{a^\sharp}$, using the definition \cite[33.11]{Skip2024} (or \cite[Theorem 2]{McC69}) of the isotope $J'^{(N(a)a^\sharp)}_{a^\sharp}$.
		\end{proof}
	\end{proposition}

	\subsection{Jordan pairs and divided power representations}

	In this subsection, we prove that cubic norm pairs correspond precisely to Jordan pairs $(J,J',Q)$ that have a specific divided power representation on the module $R \oplus J' \oplus J \oplus R$. 
	Before we can achieve this goal, however, we first need to define Jordan pairs and divided power representations.
	
	Consider a pair of $R$-modules $(J,J')$ endowed with quadratic maps
	\[ Q : J \longrightarrow \text{Hom}_R(J',J)\]
	and
	\[ Q': J'\longrightarrow \text{Hom}_R(J,J').\]
	We will often write $Q_x y $ for $Q(x)(y)$. Moreover, we will also write $Q$ for $Q'$.
	Define \[D_{x,y} z = Q^{(1,1)}_{x,z} y = Q_{x + z} y - Q_x y - Q_z y\] using the polarization of $Q$.
	
	\begin{definition}[\cite{faulkner2000jordan} axioms JP1-JP3]
		A quadruple $(J,J',Q,Q')$ is called  a \textit{quadratic Jordan pair} if
		\begin{enumerate}
			\item[JP1] $Q_a D'_{b,a} = D_{a,b} Q_a$
			\item[JP2] $D_{Q_a b,b} = D_{a,Q'_b a}$
			\item[JP3] $Q_a Q'_b Q_a = Q_{Q_a b}$
		\end{enumerate}
		hold over all scalar extensions, for all $a \in J$ and $b \in J'$, and if these equations hold with the roles of $(J,Q)$ and $(J',Q')$ reversed.
		If there exists a $1 \in J$ such that $Q_1 : J'\longrightarrow J$ is an isomorphism, we call $J$ (slightly abusively) a \textit{quadratic Jordan algebra}.
	\end{definition}
	
	\begin{remark}
		For a cubic norm pair, both $(R,R)$ with $Q_x y = x^2 y$ and $(J,J')$ with $Q_j k = - j^\sharp \times k + T(j,k) j$ are quadratic Jordan pairs. For cubic norm pairs, this technically has not been proven before (a partial result is \cite[Proposition 1]{FLK01}). In Lemma \ref{lem: dprep is dprep} we establish that $(J,J')$ is a quadratic Jordan pair. One checks that the unit of $R$ is a unit for the corresponding algebra and that $J$ is an algebra if and only if it is unital.
	\end{remark}
	
	Now, we give a preparatory definition that will allow us to define divided power representations.
	
	\begin{definition}[\cite{faulkner2000jordan} Corollary 7]
		Consider an $R$-module $M$ and a unital, associative $R$-algebra $A$.
		A sequence of homogeneous maps \[(\rho_i : M \otimes \cdot \longrightarrow A \otimes \cdot)_{i \in \mathbb{N}}\] with $\rho_i$ homogeneous of degree $i$, is called a \textit{binomial divided power representation} if
		\[ \left(1 + \sum_{i = 1}^\infty t^i \rho^K_i(j)\right)\left(1 + \sum_{i = 1}^\infty t^i \rho^K_i(k)\right) = \left(1 + \sum_{i = 1}^\infty t^i \rho^K_i(j + k)\right)\]
		holds in $(A \otimes K)[[t]]$.
		We also define $\rho_0(j) = 1$. We can also identify a divided power representation with the natural transformation $\rho_{[t]}^K : (M \otimes K) \longrightarrow (A \otimes K)[[t]]$ with $\rho_{[t]}(j) = \sum_{i = 0}^\infty t^i \rho_i(j)$.
		
		In what follows we will mostly suppress the $K$ in $\rho^K_i$.
	\end{definition}

	\begin{remark}
		\label{rem: homogeneous faulkner}
		The definition of a homogeneous map in \cite{faulkner2000jordan} does not resemble our definition. The two definitions are equivalent, however.
		For a homogeneous map in our sense, it is easy to verify \cite[A1-A6]{faulkner2000jordan}.
		Conversely, with some effort, one can show \cite[Theorem 35]{faulkner2000jordan} that any homogeneous map in Faulkners sense is defined over arbitrary scalar extensions $K$. The naturality in $K$ follows from the uniqueness of the scalar extension of the homogeneous map. 
	\end{remark}

	\begin{remark}
		The binomial part of the name, hints at the fact that
		\[ \rho_i(u) \rho_j(u) = \rho_{i+j}(u) \binom{i+j}{i}\]
		for all $i,j \in \mathbb{N}$. The divided power part is also explained in \cite{faulkner2000jordan}.
	\end{remark}
	
	\begin{example}
		\label{ex: dprepr}
		For the cubic norm pair $(J,J')$, we can consider two representations in \[A = \text{End}_{R}\begin{pmatrix}
			R \\ J' \\ J \\ R
		\end{pmatrix} \cong \begin{pmatrix}
		R & (J')^* & J^* & R \\ J' & \End(J') & \text{Hom}(J,J') & J' \\ J & \text{Hom}(J',J) & \End(J) & J \\ R & (J')^* & J^* & R
	\end{pmatrix},\]
writing $X^*$ for the dual of $X$ and $Y$ for $\text{Hom}_R(R,Y)$.
We will construct a natural transformation $\rho^+_{[t]} : J \otimes \cdot \longrightarrow  (A \otimes \cdot)[[t]]$ and a similar natural transformation $\rho^-_{[t]}$ from the images of the homogeneous maps $\text{Id}, \sharp, \sharp'$, and $N$ under certain linear maps.

Thus, we define
		\[(\rho^+_{[t]})^K(j) = \begin{pmatrix}
			1 & t T(j,\cdot) & t^2 T(j^\sharp, \cdot ) & t^3 N(j) \\
			& 1 & t j \times \cdot & t^2 j^\sharp \\
			& & 1 & t j \\
			& & & 1
		\end{pmatrix}\]
		and 
		\[(\rho^-_{[t]})^K(j) = \begin{pmatrix}
			1 \\
			- tj & 1 \\
			t^2 j^\sharp & - t j \times \cdot & 1 \\
			-t^3 N(j) & t^2 T(j^\sharp,\cdot) & - t T(j,\cdot) & 1
		\end{pmatrix}.\]
	Note that if $j = \sum j_i \otimes k_i$, the element $j \times \cdot$ represents the element $\sum (j_i \times \cdot ) \otimes k_i$. We also note that $\rho^\pm_i = 0$ for all $i > 3$.
	
	These are binomial divided power representations if and only if \[(\rho^+)^K_{[t]}(j)(\rho^+)^K_{[t]}(k) = (\rho^+)^K_{[t]}(j + k),\] which we may simply verify using the matrix product.
	That these are binomial divided power representation follows from $N^{(1,2)}(j,k) = T(j,k^\sharp) = T(k^\sharp,j)$, its implication $T(j, k \times \cdot ) = T(j \times k, \cdot),$ the linearisation $\sharp^{(1,1)}(a,b) = a \times b$, and the linearity of the other components.
	\end{example}

Now, we give the definition of what kind of pairing the above pair of representation is, list some properties of such pairings, and then prove that this example satisfies the definition in Lemma \ref{lem: dprep is dprep}. 

\begin{definition}[\cite{faulkner2000jordan} \label{def: dprep2} Equation (10)\footnote{In this definition $Q$ and $Q'$ are supposed to define a Jordan pair, which is an immediate consequence of the definition we give. There is also a minor notational difference, since we did not introduce the adjoint representation $(\ad)_{[t]}$ associated to $\rho_{[t]}$.}]
	Consider a pair of $R$-modules $(J,J')$ equipped with quadratic maps
	\[ Q : J \longrightarrow \text{Hom}(J',J)\]
	and
	\[ Q' : J' \longrightarrow \text{Hom}(J,J'),\]
	 an associative unital algebra $A$, and a pair of binomial divided power representations $(\rho^+_i : J \longrightarrow A)_i$ and $(\rho^- : J'\longrightarrow A)_i$. Such a pair is called a \textit{divided power representation with respect to $Q$ and $Q'$} if 
	\begin{equation}
		\label{def: dprep}
		\sum_{a + b = k} \rho^+_a(u)\rho^-_l(v)\rho^+_b(u) (-1)^b = \begin{cases}
			0 & k > 2l \\
			\rho^+_{l}(Q_u v)  & k = 2l\\
			\text{No assumption} & k < 2l
		\end{cases}
	\end{equation}
	holds over all scalar extensions, for all $(k,l) \in \mathbb{N}^2_{>0}$, and if the same holds whenever we reverse the roles of $+$ and $-$.
	
\end{definition}

\begin{lemma}
	\label{lem: dprep is JP}
	Suppose that $(J,J')$ is a pair of modules and that $Q : J \longrightarrow \text{Hom}(J',J)$ and $Q' : J'\longrightarrow \text{Hom}(J,J')$ are quadratic maps. The pair $(J,J ')$ is a quadratic Jordan pair with operators $Q$ and $Q'$ if and only if there exists a divided power representation with respect to $Q$ and $Q'$ such that $\rho^\pm_1$ is injective over $R$.
	\begin{proof}
		For a quadratic Jordan pair, one has the TKK representation as defined in \cite[Example 3]{faulkner2000jordan}, using $\rho^\pm_{[t]} = 1 + t \ad x + t^2 Q_x$ for the binomial divided power representations. The linear map $x \mapsto \ad x$ and the quadratic map $x \mapsto Q_x$ are easily made into natural transformations. All other $\rho^\pm_i$ are $0$ by definition.
		
		For the converse, we can look at the adjoint representation and adapt the arguments of \cite[Theorem 5]{faulkner2000jordan} to this setting.
		We first prove the quadratic Jordan pair axioms over $R$, thereafter over scalar extensions.
		
		By the adjoint representation we mean $\ad^+_i(x)(a) = \sum_{u+v = i} \rho^+_u(x) a \rho^+_v(x) (-1)^v$ for $a \in A$, which corresponds to a pair of binomial divided power representations that can be checked to form a divided power representation (follow the argument of \cite[Lemma 10, 11, 12, 13, 14]{faulkner2000jordan}). We also consider the adjoint of the adjoint, defined as $\text{Ad}^+_i(x) e = \sum_{a + b = i} (-1)^b \ad^+_a(x) e \ad^+_b(x)$, whenever $e$ is an endomorphism of $A$.
		
		The arguments of \cite[Theorem 5, Page 162]{faulkner2000jordan} can be adapted using these maps. Namely, elements $x^{(i)}$ over there correspond to $\ad^\pm_i(x)$ and elements $A_{x^{(i)}}$ correspond to $\text{Ad}^\pm_i(x)$. The equality $\rho^+_1(D_{x,y} z) = - [[\rho^+_1(x),\rho^-_1(y)], \rho^+_1(z)]$ is a linearisation of Equation \ref{def: dprep}.  First, the proof of (JP1) becomes \[- \rho_1^+(Q_x D_{y,x} z) + \rho^+_1(D_{x,y} Q_x z)  = \text{Ad}^+_3(x)(\ad^-_1(y))(\rho^-_1(z)) = 0(\rho^-_1(z)) = 0,\] using $\ad_3^+(x) \rho^-_1(z) = 0$ by Equation \ref{def: dprep} and $\ad^-_1(y) \rho^-_1(z) = 0$ for any binomial divided power representation.  The proof of (JP3) translates similarly.
		The only difficult part is $\ad^-_2(y) \ad^+_1(x) \rho^-_1(z) = 0$, which follows from
		\begin{align*}
			ad^-_2(y) [\rho^+_1(x),\rho^-_1(z)] & = \sum_{k + l = 2} [\ad^-_k(y) x , \ad^-_l(y) \rho^-_1(z)] \\ & = [\ad^-_2(y) x, \rho^-_1(z)] = \ad^-_1(Q_y x) \rho^-_1(z) = 0.
		\end{align*}
		The proof of the (JP2) carries over as well, once again using \[\rho^+_1(D_{x,y} z) = - [[\rho^+_1(x),\rho^-_1(y)], \rho^+_1(z)]\] so that we can identify $D_{x,y}$ with an element of the algebra in which we have a representation.
		
		Now, if $\rho^\pm_1$ is injective over $R$, it is injective over all flat $R$-algebras $K$. Hence, the quadratic Jordan pair axioms hold for all flat $K$, and thus for all $R$-algebras.
	\end{proof}
\end{lemma}

\begin{definition}
	For a Jordan pair $(J,J')$ with operators $Q$ and $Q'$, we call a divided power representation with respect to $Q$ and $Q'$ simply a \textit{divided power representation}.
\end{definition}

\begin{lemma}
	\label{lem: dprep is dprep}
	Each cubic norm pair $(J,J')$ forms a quadratic Jordan pair $(J,J',Q)$.
	Moreover, the pair of binomial power representations $\rho^\pm_{[t]}$ of Example \ref{ex: dprepr} form a divided power representation of the quadratic Jordan pair $(J,J')$.
	\begin{proof}	
		By Lemma \ref{lem: dprep is JP} it is sufficient to establish the moreover-part.
		
		We only show that Equation (\ref{def: dprep}) holds in the way it stands there and not with the roles of $+$ and $-$ reversed, as the arguments we would make, would be almost identically the same. The proof we give works for scalar extensions $K$ such that $\End_R(X) \otimes K$ embeds into $\End_K(X \otimes K)$, such as polynomial rings, since we will just multiply the matrices together. For arbitrary $R$-algebras, note that homogeneous maps that agree on all polynomial rings will agree on all $R$-algebras.
		
		 The equation must hold for certain values for the parameters $k$ and $l$. Given that $k - l$  measures the distance to the diagonal, we know that $|k - l| \le 3$ or that the equation holds trivially.
		Hence, we must only check that the equation holds for $(k,l) \in \{(2,1), (3,1),(4,1),(4,2), (6,3)\}.$
		Moreover, for a fixed $(k,l)$ we should only look at what stands at positions $(i,j)$ in the matrix, with $j - i = k - l$. When we do such a check, we will say that we check requirement $(k,l,i,j)$. Since computing what Equation (\ref{def: dprep}) says for $(k,l,i,j)$ is easy, we immediately give the result of said computation and then say why it holds.
		
		We we verify Equation (\ref{def: dprep}) for each case $(k,l,i,j)$:
		\begin{itemize}
			\item $(2,1,1,2)$ holds if $T(j,k) T(j,\cdot)- T(j^\sharp, k \times \cdot ) = T(T(j,k)j,\cdot ) - T(j^\sharp \times k, \cdot)$, which holds since $T$ is symmetric and Lemma \ref{lem: basic properties cnp}.3.
			\item $(2,1,2,3)$ holds if $j \times (k \times (j \times l)) - k T(j^\sharp,l) = T(j,k)j \times l - (j^\sharp \times k) \times l$ holds for all $l$, which is a linearisation of Definition \ref{def: cubic norm pair}.3.
			\item $(2,1,3,4)$ holds if $jT(k,j) - k \times j^\sharp = jT(k,j) - k \times j^\sharp$, which holds trivially.
			\item $(3,1,1,3)$ holds if $ - N(j)T(k,\cdot) + T(j^\sharp, k \times (j \times \cdot)) - T(j,k)T(j^\sharp, \cdot ) = 0$, which is the case since $ -N(j)l + j^\sharp \times (j \times l) - j T(j^\sharp, l) = 0$ for all $l$, as it is a linearisation of Definition \ref{def: cubic norm pair}.2.
			\item $(3,1,2,4)$ holds if $- j \times (k \times j^\sharp) + j^\sharp T(k,j) + kN(j) = 0$, which is precisely Definition \ref{def: cubic norm pair}.3.
			\item $(4,1,1,4)$ holds if $2 N(j)T(k,j) - T(j^\sharp,k \times j^\sharp) = 0$, which is the case since $N(j)j = (j^\sharp)^\sharp$,
			\item $(4,2,1,3)$ holds if $- N(j)T(k^\sharp, j \times \cdot )+ T(j^\sharp,k^\sharp) T(j^\sharp, \cdot ) = T((Q_jk)^\sharp,\cdot)$, which is the case since $(Q_jk)^\sharp = Q_{j^\sharp} k^\sharp$ by Lemma \ref{lem: sharp is non linear endo}.
			\item $(4,2,2,4)$ holds for the same reason.
			\item $(6,3,1,4)$ holds if $N(j)^2N(k) = N(Q_j k)$ which is precisely Definition \ref{def: cubic norm pair}.4.
		\end{itemize}
		This is everything we had to check.
	\end{proof}
\end{lemma}

\begin{remark}
	\label{rem: necessity def}
	This lemma explains why we defined cubic norm pairs in the way we did.
	In this work, we will restrict ourselves to Jordan pairs that admit \textit{a divided power representation as in Example} \ref{ex: dprepr}. Only the axiom $j^{\sharp\sharp} = N(j)j$ does not immediately follow from having such a representation.
	In Lemma \ref{lem: necessity def} we make having a divided power representation as in Example \ref{ex: dprepr} precise.

	The reason that this divided power representation is so important, is that this representation will be the one that allows us to define an action as automorphisms on the TKK Lie algebra of a skew-dimension one structurable algebra $\mathcal{A} \cong R \oplus J' \oplus J \oplus R$.
\end{remark}

\begin{lemma}
	\label{lem: necessity def}
	Let $(J,J')$ be a quadratic Jordan algebra with a divided power representation $\rho^\pm$ in $\End(R \oplus M\oplus N \oplus R)$ with $\rho^+_{[t]}(x)$ upper triangular for all $x \in J$, $\rho^-_{[t]}(y)$ lower triangular for all $y \in J'$, and $t$ tracking the distance to the diagonal in both cases.
	Suppose, additionally, that $j \mapsto \rho^+_1(j) \cdot (0,0,0,1)$ and $j' \mapsto \rho^-_1(j') \cdot (1,0,0,0)$ define bijections $J \longrightarrow N$ and $J' \longrightarrow M$. Assume that $\rho^-_1(k) \cdot j =T(j,k) = \rho^+_1(j) \cdot k$ under these bijections.
	Then $(J,J')$ is a Jordan pair with representation as in Example \ref{ex: dprepr}. This $(J,J')$ is a cubic norm pair if and only if $N(j)j = j^{\sharp\sharp}$ for all $j \in J \cup J'$.
	\begin{proof}
		We compose all $(\rho^\pm_i)^K$ with the map \[\End(R \oplus J'\oplus J \oplus R) \otimes K \longrightarrow \End_K( (R \oplus J'\oplus J \oplus R) \otimes K).\]
		
		We define \[N(j) = \rho^\pm_3(j) \cdot (0,0,0,1)\] for $j \in J \otimes K$ and \[N(j') = \rho^\pm_3(j') \cdot (1,0,0,0)\] for $j' \in J' \otimes K.$
		We define $j^\sharp$ similarly by evaluating $\rho^\pm_2(j)$ on a $1$ contained in a suitable copy of $R$. 
		We also define $\times$ as the polarization of $\sharp$.
		
		Axiom (1) for cubic norm pairs is the linearisation $N^{(1,2)}(j,k) = \rho^\pm_{(1,2)}(j,k) = T(j,k^\sharp)$.
		Moreover, having a divided power representation guarantees that we have binomial divided power representations as in Example \ref{ex: dprepr}. Hence, it remains only verify that the other cubic norm pair axioms hold.
		
		The analogous equalities to the equalities (3,1,2,4), and (6,3,1,4) of Lemma \ref{lem: dprep is dprep}, are axioms (3), and (4).
	\end{proof}
\end{lemma}
 
\begin{remark}
	We also consider a divided power representation for the base ring $R$, considered as a quadratic Jordan algebra with $Q_x y = x^2 y$.
	Namely, we can consider \[\rho^+_{[t]}(r) = \begin{pmatrix}
		1 &tr \\
		0& 1 
	\end{pmatrix} \text{ and }\rho^-_{[t]}(r) = \begin{pmatrix}
		1  & 0\\
		-tr & 1\\
	\end{pmatrix},\] which are natural transformations $\rho^\pm_{[t]} : R \otimes K \longrightarrow (\text{Mat}_2(R) \otimes K) [[t]]$. This has natural actions on $J^2$ and $R^2$, induced by the action of $M_2(R)$.
\end{remark}

So, in what follows we will be considering two Jordan pairs. Namely, $(J,J')$ and $(R,R)$ both of which have a ``natural" representation. We remark that this ``natural" representation is context dependent, since some cubic norm pairs $(J,J')$ can be seen as coming from an associative commutative algebra.

\section{Structurable algebras, associated Lie algebras, and associated inner automorphisms}

\label{sec: stru alg}

In this section, we construct a structurable algebra from a cubic norm pair. This structurable algebra has an associated Lie algebra, which we will construct as well. We also use the ``natural" divided power representation of a cubic norm pair to construct automorphisms of this Lie algebra. 
We also determine commutators of certain automorphisms.
With these tools, we prove the main theorem advertised in the introduction for cubic norm pairs.

\subsection{Kantor triple systems}

In this subsection, we restate parts of \cite[section 3]{All99} in such a way that avoids making any assumption on the base ring $R$. 
This will be useful in the construction of the Lie algebra. 
In the next section, we will apply the Lie algebra construction we introduce here to a class of structurable algebras.

\begin{definition}
	Consider a module $M$ endowed with a trilinear map \[\{\cdot,\cdot,\cdot\} : M^3 \longrightarrow M\] and define $V_{x,y} z = \{x,y,z\}$ and $K_{x,z} y = V_{x,y} z - V_{z,y} x$. We call $(M,V)$ a \textit{Kantor triple system} if
	\begin{enumerate}
		\item $V_{x,y} V_{u,v} - V_{u,v} V_{x,y} = V_{V_{x,y} u,v} - V_{u,V_{y,x} v}$,
		\item $K_{x,y} V_{v,u} + V_{u,v} K_{x,y} = K_{K_{x,y} v, u}$,
	\end{enumerate}
	for all $x,y,u,v \in M$.
\end{definition} 

\begin{definition}
	Consider a module $M$ with a trilinear map $[\cdot,\cdot,\cdot]$. We define $W : M \times M \longrightarrow \End(M)$ as $(x,y) \mapsto (z \mapsto [x,y,z])$.
	We call $M$ a \textit{Lie triple system} if
	\begin{enumerate}
		\item $W_{x,x} = 0$,
		\item $W_{x,y} z + W_{y,z} x + W_{z,x} y = 0$,
		\item $[W_{x,y}, W_{u,v}] = W_{W_{x,y} u,v} + W_{u,W_{x,y} v}$,
	\end{enumerate}
	hold for all $x,y,u,v$ and $z$ in $M$.
\end{definition}

We remark that for any Lie triple system, there exists an inner derivation Lie algebra $\InDer(M)$ that coincides with the linear span $\langle W_{x,y} | x,y \in M \rangle \subset \End(M).$ From the third axiom, it follows that this linear span is closed under $(f,g) \mapsto [f,g] =  fg - gf$.

\begin{construction}
	Any Lie triple system $M$ has a corresponding Lie algebra $L = \InDer(M) \oplus M$. Any products involving elements of $\InDer(M)$ are the obvious ones, for $M \times M \longrightarrow (\InDer M \oplus M)$ we use $(x,y) \mapsto W_{x,y}$. An easy check shows that $L$ definitely is a Lie algebra.
\end{construction}

Now, we associate a Lie triple system to a Kantor triple system $M$, using $N = M_+ \oplus M_-$ as underlying module (using signs to distinguish distinct copies of the same module).
We represent an element of $\End(N)$ as a matrix.
So, we define 
\begin{align*}
	W_{x_+,y_-} = \begin{pmatrix}
		V_{x,y} & 0 \\ 0 & - V_{y,x}
	\end{pmatrix}, \\
W_{x_-,y_+} = \begin{pmatrix}
	- V_{y,x} & 0 \\ 0 & V_{x,y}
\end{pmatrix}, \\
	W_{x_+,y_+} = \begin{pmatrix}
		0 & K_{x,y} \\ 0 & 0
	\end{pmatrix}, \\
	W_{x_-,y_-} = \begin{pmatrix}
	0	& 0 \\ K_{x,y}  & 0
	\end{pmatrix}
\end{align*}
and extend using bilinearity.

We remark that the first and second axioms for Lie triple systems follow immediately from the definition of $W$ and the definition of $K$.

\begin{lemma}
	If $(M,V)$ is a Kantor triple system, then $(N,W)$ is a Lie triple system.
	\begin{proof}
		This follows from \cite[Theorem 7]{All99}. The assumption $1/2 \in R$ made in the section containing \cite[Theorem 7]{All99} plays no role in the proof.
	\end{proof}
\end{lemma}

\begin{remark}
	When we speak of the Lie algebra associated to a Kantor triple system, we mean the Lie algebra associated to the Lie triple system associated to it.
	We remark that this Lie algebra has a $\mathbb{Z}$-grading by construction, with $M_+$ being $1$-graded and $M_-$ being $-1$-graded.
\end{remark}

\subsection{Construction of the structurable algebra and associated Lie algebra}

In this subsection, we take the structurable algebra constructed in \cite[Example 6.4]{Allison1993NonassociativeCA} from maps $T$ and $\sharp$ and check whether it defines a Kantor triple system for cubic norm pairs, to construct a Lie algebra. 
In \cite{Allison1993NonassociativeCA}, one also considers Lie algebras defined from these structurable algebras, but these Lie algebras do not always match the ones we need.

\begin{construction}
	\label{con: A}
	We can associate a structurable algebra, in the sense of \cite{Allison1993NonassociativeCA}, to an arbitrary cubic norm pair.
	Namely, set 
	\[ \mathcal{A} = \begin{pmatrix}
		R & J \\
		J' & R \\
	\end{pmatrix},\]
	with multiplication
	\[ \begin{pmatrix}
		a & b \\
		c & d
	\end{pmatrix} \begin{pmatrix}
		e & f \\
		g & h
	\end{pmatrix} = \begin{pmatrix}
		ae + T(b,g) & a f + b h + c \times g \\
		c e + d g + b \times f & d h + T(c,f)\\
	\end{pmatrix}\]
	and involution
	\[ \begin{pmatrix}
		a & b \\ c & d
	\end{pmatrix} \mapsto \begin{pmatrix}
		d & b \\ c & a
	\end{pmatrix}\]
	which is denoted by $x \mapsto \bar{x}$.
	This coincides with the construction using only $T,\times$ and $\times'$ described in \cite[Example 6.4]{Allison1993NonassociativeCA}.
	The necessary and sufficient conditions formulated over there for $\mathcal{A}$ to be structurable are
	\begin{itemize}
		\item $N^{(1,1,1)}(a,b,c) = N^{(1,1,1)}(b,c,a)$,
		\item the $((1,1,1),1)$ linearisation of Definition \ref{def: cubic norm pair}.2,
		\item the $((1,1,1),1)$ linearisation of  Definition \ref{def: cubic norm pair}.3.
	\end{itemize}
	Define the space $\mathcal{S}$ as \[\mathcal{S} = \left\{ \begin{pmatrix}
		r & 0 \\ 0 & -r
	\end{pmatrix} : r \in R \right\}.\]
We remark that $a - \bar{a} \in \mathcal{S}$ for all $a \in \mathcal{A}$. This space is usually called the space of skew elements (hence the name skew dimension $1$ structurable algebra one typically uses for these algebras and their forms, whenever one works over fields of characteristic different from $2$ and $3$).
\end{construction}

Now, we construct the Lie algebra.
\begin{construction}
	
	\label{con: lie}
	We define $V_{x,y} z = (x\bar{y})z + (z \bar{y})x - (z \bar{x}) y$ and \[\mathfrak{instr}(\mathcal{A}) =  \text{span}\{ (V_{x,y}, - V_{y,x} ) \in \End(\mathcal{A}) \times \End(\mathcal{A})\}.\]
	Since $\mathcal{A}$ is structurable in the sense of \cite{Allison1993NonassociativeCA}, $\mathfrak{instr}(\mathcal{A})$ is a Lie subalgebra of $\End(\mathcal{A})^2$ and the first axiom for Kantor triple systems holds immediately (\cite[Lemma 5.4]{Allison1993NonassociativeCA}). There is a second condition in \cite{Allison1993NonassociativeCA} to be structurable called (sk), which we will not need, since we can control multiplications involving $\mathcal{S}$ very well.
	If, $1/6 \in \Phi$ it is well known that the above guarantees that there exists a $5$-graded Lie algebra
	\[ \mathcal{L}(\mathcal{A}) = \mathcal{S}_{-2}\oplus \mathcal{A}_{-1}\oplus \mathfrak{instr}(A)_{0} \oplus \mathcal{A}_{1} \oplus \mathcal{S}_{2}\]
	where we used $X_i$ for a space $X$ to, first, distinguish between distinct copies of a space and, second, to denote the $\mathbb{Z}$-grading. The even graded subspace space acts as
	\[ s_{-2} \oplus (V_{x,y}, -V_{y,x}) \oplus s_{2} \mapsto \begin{pmatrix}
		V_{x,y} & L_{s_{2}} \\ L_{s_{-2}} & - V_{y,x} \\
	\end{pmatrix},\]
	using $L_s$ to denote the left multiplication action of $s$, on 
	\[ \begin{pmatrix}
		\mathcal{A}_1 \\ \mathcal{A}_{-1}
	\end{pmatrix}.\]
	Furthermore, the Lie bracket on 
	\[ \begin{pmatrix}
		\mathcal{A}_1 \\ \mathcal{A}_{-1}
	\end{pmatrix}\]
	is given by 
	\[ \left[ \begin{pmatrix}
		a_1 \\ a_{-1}
	\end{pmatrix}, \begin{pmatrix}
		b_1 \\ b_{-1}
	\end{pmatrix}\right] = \begin{pmatrix}
		- V_{a_1,b_{-1}} + V_{b_{1},a_{-1}} & L_{- a_1 \bar{b}_1 + b_1 \bar{a}_1} \\
		L_{- a_{-1} \bar{b}_{-1} + b_{-1} \bar{a}_{-1}} & - V_{a_{-1},b_{1}} + V_{b_{-1},a_{1}}
	\end{pmatrix}.\]

To be complete, we now prove that this construction of a Lie algebra $\mathcal{L}(\mathcal{A})$ works over arbitrary rings and not solely over rings containing $1/6$. As a preparation, we first compute some $V$ operators, these computations will help to show that the $\mathbb{Z}$-grading refines to a $\mathbb{Z}^2$-grading.
For notational brevity, we introduce \[t = \begin{pmatrix}
	1 & 0 \\ 0 & 0
\end{pmatrix} \text{ and } \bar{t} = \begin{pmatrix}
0 & 0 \\ 0 & 1
\end{pmatrix}.\] We also write $R_a$ and $L_a$ for the right and left multiplication with $a$.

\begin{lemma}
	\label{lem: Vs}
	Consider the elements $t, \bar{t} \in R[t]/(t^2 - t) \subset \mathcal{A}$, $v, v' \in J \subset \mathcal{A}$, $w, w' \in J' \subset \mathcal{A}$, and $x \in J \oplus J' \subset \mathcal{A}$. We have
	\begin{enumerate}
		\item $R_x R_t = R_{\bar{t}} R_x$,
		\item $R_x R_{\bar{t}} = R_t R_x$,
		\item $V_{t,x} = V_{x,t} = L_{x\bar{t}}$,
		\item $V_{\bar{t},x} = V_{x,\bar{t}} = L_{xt}$,
		\item $0 = V_{t,t} = V_{t,w} = V_{w,t} = V_{\bar{t},v} = V_{v ,\bar{t}}  = V_{\bar{t}, \bar{t}}$,
		\item $V_{t,\bar{t}} = L_t + R_{t - \bar{t}}$,
		\item $V_{\bar{t},t} = L_{\bar{t}} + R_{\bar{t} - t}$,
		\item $L_v = V_{t,v} = V_{v,t}$,
		\item $L_{w \times w'} = V_{w,w'}$,
		\item $L_w = V_{\bar{t},w} = V_{w,\bar{t}},$
		\item $L_{v \times v'} = V_{v,v'}$,
		\item 
		$V_{v,w} \begin{pmatrix}
			a & b \\ c & d
		\end{pmatrix} = \begin{pmatrix}
		0 & D_{v,w} b \\ T(v,w)c - D_{w,v} c & T(v,w) d
	\end{pmatrix},$
		\item $V_{w,v} \begin{pmatrix}
			a & b \\ c & d 
		\end{pmatrix} = \begin{pmatrix}
		T(v,w) a & T(v,w) b - D_{v,w} b \\ D_{w,v}c & 0
	\end{pmatrix}.$
	\end{enumerate}
	\begin{proof}
		For the first two equations, assume that $x = v + w$, compute
		\[  \begin{pmatrix}
			a & b \\ c & d
		\end{pmatrix} \begin{pmatrix}
		 & v \\ w &
	\end{pmatrix} = \begin{pmatrix}
	T(b,w) & c \times w + av \\ b \times v + dw & T(c,v)
\end{pmatrix},\]
and note that the variables $a$ and $c$ occur only in the second column, while the variables $b$ and $d$ occur only in the first column. The operators $R_t$ and $R_{\bar{t}}$ are projection operators on the columns. Hence, the first two equations follow.

For the next equations, we will use
\[ V_{a,b} = L_{a\bar{b}} + R_aR_{\bar{b}} - R_{b} R_{\bar{a}}.\]
Equations (3) and (4) follow from this since $x$ is off-diagonal, i.e., $tx = x\bar{t}$ and $\bar{t}x = xt$, using equations (1) and (2).
The list of equations (5) follows similarly, as these are all $V_{a,b}$ with $L_{a\bar{b}} = 0$ and $R_aR_{\bar{b}} - R_{b} R_{\bar{a}}$ = 0, using equations (1) and (2) or $R_t R_{\bar{t}} = 0 = R_{\bar{t}} R_t$.
Equations (6) and (7) also follow immediately since $R_t^2 = R_t$ and $R_{\bar{t}}^2 = R_{\bar{t}}$.
Equations (8) and (10) are special cases of (3) and (4). 

Note that the expression
\[ R_w R_{w'}  \begin{pmatrix}
	a & b \\ c & d
\end{pmatrix} = R_w \begin{pmatrix}
T(b,w') & c \times w' \\ dw' & 0
\end{pmatrix} = \begin{pmatrix}
T(w \times w', c) & d (w' \times w) \\ 0 & 0
\end{pmatrix}\]
is symmetric in $w$ and $w'$. This implies that equation (9) holds, since $w w' = w \times w'$.
Equation (11) follows similarly.

Now, we compute
\[ R_v R_w \begin{pmatrix}
	a & b \\ c & d
\end{pmatrix} = R_v \begin{pmatrix}
T(b,w) &c \times w \\ dw & 0
\end{pmatrix} = \begin{pmatrix}
 0 &  T(b,w) v\\ (c \times w) \times v & d T(v,w).
\end{pmatrix}\]
Analogously, one obtains
\[ R_w R_v  \begin{pmatrix}
	a & b \\ c & d
\end{pmatrix} = \begin{pmatrix}
a T(v,w) & (b \times v) \times w \\ T(c,v) w & 0
\end{pmatrix}.\]
We conclude that
\[ V_{v,w} = L_{T(v,w) t} + R_v R_w - R_w R_v\]
acts as
\[ \begin{pmatrix}
a & b \\ c & d
\end{pmatrix} \mapsto \begin{pmatrix}
0 & T(v,w) b + T(b,w) v - (b \times v) \times w \\ (c \times w) \times v - T(c,v) w & T(v,w) d 
\end{pmatrix} ,\]
which is what we need to prove for equation (12). Equation (13) follows similarly.
	\end{proof}
\end{lemma}
	
	\begin{lemma}
		\label{lem: lie cnp}
		The algebra $\mathcal{L}(\mathcal{A})$ is a Lie algebra for all cubic norm pairs.
		\begin{proof}
			We define
			\[ t = \begin{pmatrix}
				1 \\ & 0
			\end{pmatrix}, \quad \bar{t} = \begin{pmatrix}
				0 \\ & 1
			\end{pmatrix},\]
			and write $v$ and $w^*$ for matrices of the form
			\[ v = \begin{pmatrix}
				& v \\&
			\end{pmatrix}, \quad w^* = \begin{pmatrix}
				& \\
				w & 
			\end{pmatrix}.\]
			Proving that $\mathcal{L}(\mathcal{A})$ is a Lie algebra is equivalent to proving that $(\mathcal{A},V)$ is a Kantor triple system and that $K_{x,z}= L_{x\bar{z} - z \bar{x}}$, since $\mathcal{L}(\mathcal{A})$ is precisely the Lie algebra associated to such a triple system (slightly modified, as our construction passes through the Lie triple system $(A,-W)$ where $W$ is the operation of the usual Lie triple system). That $A$ is structurable in the sense of \cite{Allison1993NonassociativeCA} implies that the first axiom for Kantor triple system holds. Now, we remark that 
			\[ V_{x,y} z - V_{z,y} x = (x\bar{z} - z\bar{x}) y,\]
			so that $K_{x,z} = L_{(x\bar{z} - z \bar{x})}$.
			We use $[x,z]$ to denote $x\bar{z} - z \bar{x}$ (this is the Lie bracket with respect to $W$).
			Hence, to check the second axiom, it is sufficient to check whether
			\[ L_{t - \bar{t}} V_{x,y} + V_{y,x} L_{t - \bar{t}} = L_{[(t - \bar{t})x,y]}.\] 
			This equation is linear in $x$ and $y$, hence it is sufficient to check it for all pairs $(x,y)$ with $x$ and $y$ part of $\{t,\bar{t}\} \cup  J \cup J'$.
			Direct computations show that
			\begin{itemize}
				\item $0 = V_{t,t} = V_{t,v^*} = V_{v^*,t} = V_{\bar{t},\bar{t}} = V_{\bar{t},v} = V_{v,\bar{t}}$ by Lemma \ref{lem: Vs} while $[t,t] = [t,v^*] = [\bar{t},\bar{t}] = [\bar{t},v^*] = 0$,
				so that the axiom holds for these pairings;
				\item $L_v = V_{t,v} = V_{v,t}$, $L_{v^*} = V_{v^*,\bar{t}} = V_{\bar{t},v^*}$, $L_{(v \times w)^*} = V_{v,w} = V_{w,v}$, $L_{v \times w} = V_{v^*,w^*} = V_{w^*,v^*}$ by Lemma \ref{lem: Vs}, $L_{t - \bar{t}} L_v + L_v L_{t - \bar{t}} = 0$ using that $a \mapsto \bar{a}$ is an isomorphism to $\mathcal{A}^\text{op}$ and Equations (1) and (2) proved in Lemma \ref{lem: Vs}, and $[a,b] = 0$ for all $(a,b)$ for which we consider $V_{a,b}$;
				\item $L_{t - \bar{t}}V_{t,\bar{t}} + V_{\bar{t},t} L_{t - \bar{t}} = L_{t - \bar{t}}$ while $[L_{t - \bar{t}} t, \bar{t}] = L_{[t, \bar{t}]}$ is readily verified. The analogous equation for $V_{\bar{t},t}$ follows similarly.
				\item For the final two cases, we first use that
				\[ L_{t - \bar{t}} V_{v,w^*} + V_{w^*,v} L_{t - \bar{t}} = T(v,w) L_{t - \bar{t}}\]
				and $L_{[v,w^*]} = T(v,w) L_{t - \bar{t}}$. The final case
				\[ L_{t - \bar{t}} V_{w^*,v} + V_{v, w^*} L_{t - \bar{t}} = T(v,w) L_{t - \bar{t}}\]
				follows similarly.
			\end{itemize}
			This proves that the second axiom for Kantor triple systems holds.
		\end{proof}
	\end{lemma}

Note that $[v,w] = - L_{v\bar{w} - w \bar{v}}$ in $\mathcal{L}(\mathcal{A})$ whenever $v, w \in \mathcal{A}_1$.

\begin{remark}
	\label{rmk: grading}
	The computations of Lemma \ref{lem: Vs} more ore less establish the existence of another grading.
	First, introduce $X$ as the span of all elements $( V_{v,w}, - V_{w,v}), (V_{w,v}, - V_{v,w})$, $(V_{t, \bar{t}}, - V_{\bar{t},t})$, and $( - V_{\bar{t},t}, V_{t, \bar{t}})$.
	Note that $\mathfrak{instr}(\mathcal{A}) \cong L_{J'} + X + L_J$ by Lemma \ref{lem: Vs}.
	Now, the usual $\mathbb{Z}$-grading can be refined into a $\mathbb{Z}^2$-grading of the following form:	
	\begin{center}
		\begin{tikzpicture}[scale=0.4]
			\node[] at (-12,0) {$=$};
			
			\node[] at (-16,6.96) {$\mathcal{S}_2$};
			\node[] at (-16,3.48) {$\mathcal{A}_1$};
			\node[] at (-16,0) {$\mathfrak{instr}(A)$};
			\node[] at (-16,-6.96) {$\mathcal{S}_{-2}$};
			\node[] at (-16,-3.48) {$\mathcal{A}_{-1}$};
			
			\node[] at (0,0) {$X_0$};
			\node[] at (4,0) {$J_{\alpha}$};
			\node[] at (-4,0) {$J'_{- \alpha}$};
			\node[] at (2,3.48) {$J'_{\beta + 2 \alpha}$};
			\node[] at (-2,-3.48) {$J_{-\beta - 2 \alpha}$};
			\node[] at (2,-3.48) {$J'_{-\beta - \alpha}$};
			\node[] at (-2,3.48) {$J_{\beta + \alpha}$};
			
			\node[] at (6,3.48) {$R_{\beta + 3 \alpha}$};
			\node[] at (-6,3.48) {$R_\beta$};
			\node[] at (6,-3.48) {$R_{-\beta}$};
			\node[] at (-6,-3.48) {$R_{- \beta - 3 \alpha}$};
			\node[] at (0,6.96) {$\mathcal{S}_{2 \beta + 3 \alpha}$};
			\node[] at (0,-6.96) {$\mathcal{S}_{- 2 \beta - 3 \alpha}$};
			
			\draw[-,dotted] (0,6.96) -- (6,-3.48);
			\draw[-,dotted] (0,-6.96) -- (-6,3.48);	
		\end{tikzpicture}
	\captionof{figure}{Different gradings on $\mathcal{L}(\mathcal{A})$}
\end{center}
We will formalize this observation in Lemma \ref{lem: dotted grading} by proving that the dotted grading, i.e., the spaces with roots $\lambda \beta + \mu \alpha$ for which $\mu - \lambda$ is constant, form a $\mathbb{Z}$-grading of the Lie algebra.
The nonzero graded components form a root system $G_2$, which we used to label the roots (for what we consider the $G_2$-root system to be, see Definition \ref{def: G2 grading}).

For uniformity of notation, we will write $R_{\pm (2 \beta + 3 \alpha)}$ for the corresponding $\mathcal{S}_{\pm (2 \beta + 3 \alpha)}$.
We use $r_{2 \beta + 3 \alpha } = r (t - \bar{t})_{2}$ and and $r_{-(2 \beta + 3 \alpha )} = -r (t - \bar{t})_{-2}$, using 
\[ t = \begin{pmatrix}
	1 & 0\\ 0  &0 
\end{pmatrix}.\] 
This sign convention is chosen to guarantee that \[2 Q_{r (t - \bar{t})_2} s (t - \bar{t})_{-2} = - r^2 s[(t - \bar{t})_2, [(t - \bar{t})_2, (t - \bar{t})_{-2}]] = 2 r^2 s (t - \bar{t})_2,\] which will guarantee that $(R_{2\beta + 3 \alpha}, R_{- (2 \beta + 3 \alpha)})$ defines a divided power representation for the Jordan pair $(R,R)$.
\end{remark}

\begin{lemma}
	\label{lem: dotted grading}
	The map $x \alpha + y \beta \mapsto x - y$, when applied to the subspaces identified in Remark \ref{rmk: grading}, induces a $\mathbb{Z}$-grading on the Lie algebra $L(\mathcal{A})$.
	\begin{proof}
		We remark that $t$ is $2$-graded, $v$ is $0$-graded, $w$ is $1$-graded and $\bar{t}$ is $-1$-graded for $t, \bar{t}, v, w$ as used in Lemma \ref{lem: lie cnp}, whenever they are considered as elements of $\mathcal{A}_1$.
		Whenever we look at $t, v, w$, and $\bar{t}$ in $\mathcal{A}_{-1}$, we obtain elements of degree $1, -1, 0$ and $-2$.
		This describes the grading on $\mathcal{A}_1 \times \mathcal{A}_{-1}$.
		We will show that this decomposition defines a Lie algebra grading.
		
		An easy computation shows that $L_v \in \mathfrak{instr}(\mathcal{A})$ acts on the degrees as $+1$, and that $L_{w} \in \mathfrak{instr}(\mathcal{A})$ acts on the degrees as $-1$, while $X$, which is spanned by the elements of the form $[v_1, w_{-1}], [v_{-1}, w_1]$, and $[t_1,\bar{t}_{-1}]$ and $[\bar{t}_1, t_{-1}]$, using $a_{\pm 1}$ for elements of $A_{\pm 1}$, preserves the grading by Lemma \ref{lem: Vs}.
		We remark that the computations of Lemma \ref{lem: Vs} also show that $\mathfrak{instr}(A)$ is linearly generated by $X$ and elements of the form $L_v$ or $L_{w^*}$.
		
		Now consider $f(\lambda)$ mapping $e \in (\mathcal{A}_1 \oplus \mathcal{A}_{-1}) \otimes R[\lambda, \lambda^{-1}]$ contained in the space $L_{x \alpha + y \beta}$ to $\lambda^{ x - y} e$.
		The computations of $V$ performed in Lemma \ref{lem: Vs} (equations (8) and (9) determine brackets ending up in $L_J$, equations (9) and (10) determine the ones ending up $L_{J'}$, equations (12), (13), (6), and (7) the ones in $X$, and equation (5) says that other brackets are $0$) establish that 
		\[f(\lambda ) V_{x,y} f(\lambda)^{-1} = V_{f(\lambda) x, f(\lambda) y}\]
		whenever $x \in \mathcal{A}_{\pm 1}$ and $y \in \mathcal{A}_{\mp 1}$ are homogeneous with respect to the dotted grading.
		
		We conclude that $f(\lambda)$ defines an automorphism of $\mathcal{L}(\mathcal{A} \otimes R[\lambda,\lambda^{-1}])$. Using $R[\lambda,\mu, \mu^{-1} \lambda^{-1}])$ and $f(\lambda) f(\mu) = f(\lambda \mu)$ yields that $f(\lambda) u = \sum_{i = -2}^2 \lambda^i u_i$ defines projections onto the grading components. That $f$ is an automorphism implies that the grading is a Lie algebra grading.
	\end{proof}
\end{lemma}

\end{construction}

\begin{lemma}
	\label{lem: generators X}
	The space $X$ can be written as \[[J'_{\beta + 2 \alpha},J_{- \beta - 2 \alpha }] + R[t_{\beta + 3 \alpha},\bar{t}_{- \beta - 3 \alpha }] + R[t_{-\beta },\bar{t}_{\beta}].\] This sum is not necessarily direct.
	\begin{proof}
		We know that $X$ is linearly generated by the elements $\alpha_{v,w} = (V_{v,w^*}, - V_{w^*,v}),$ $\beta_{v,w} = (V_{w^*,v}, - V_{v,w^*}),$ $\gamma = (V_{t,\bar{t}}, - V_{\bar{t},t}),$ $\delta = (V_{\bar{t},t}, - V_{t,\bar{t}})$.
		Using that $\alpha_{v,w}+ \beta_{v,w} = T(v,w) (\gamma + \delta)$, we conclude that the lemma holds.
	\end{proof}
\end{lemma}

\begin{remark}
	The elements $\zeta_{\beta + 3 \alpha } = (V_{\bar{t},t}, - V_{t,\bar{t}})$ and $\zeta_\beta = (V_{t,\bar{t}}, - V_{\bar{t},t})$ are grading derivations. For example, $\zeta_{\beta + 3 \alpha }$ acts as $[\zeta_{\beta + 3 \alpha} , x] = i x$ whenever $x$ is $i$-graded in the dotted grading, i.e., $x \in \bigoplus_{\lambda \in \mathbb{Z}} U_{i \alpha + \lambda( \beta + \alpha)}$ with $U_\delta$ the $\delta$-graded component of $\mathcal{L}(\mathcal{A})$.
	
	We will use $\zeta = \zeta_{2 \beta + 3\alpha}$ to denote the grading derivation associated to the usual $\mathbb{Z}$-grading.
	Note that $\zeta = \zeta_\beta + \zeta_{\beta + 3 \alpha}$.
\end{remark}

\begin{remark}
	The previous lemma says that $X$ is the sum of a central cover of the inner derivation algebra of $(J'_{\beta + 2 \alpha},J_{- \beta - 2 \alpha})$ with the torus corresponding to the $R_\gamma$.
\end{remark}

\begin{remark}
	Any group grading of the Lie algebra $\mathcal{L}(\mathcal{A})$ induces a grading on $\End(\mathcal{L}(\mathcal{A}))$ with an endomorphism being $g$ graded if it acts as $+g$ on the grading of the Lie algebra.
	As such, the $\mathbb{Z}^2$-grading components of $\End(\mathcal{L}(\mathcal{A}))$ are have grading $x \alpha + y \beta$ for $x,y \in \mathbb{Z}^2$.
\end{remark}

\subsection{Divided power representations and automorphisms}

In this subsection, we construct divided power representations for all Jordan pairs $(R_\gamma,R_{- \gamma})$ and $(J'_{\gamma},J_{\gamma})$, using the ``natural" divided power representations of these Jordan pairs, and show that these divided power representations correspond to automorphisms of the Lie algebra $\mathcal{L}(\mathcal{A})$.

\begin{remark}
	In this section and the following ones, we will need to verify equalities between homogeneous maps mapping to $\End(\mathcal{L}(\mathcal{A})) \otimes K$.
 Observe that $\End(\mathcal{L}(\mathcal{A})) \otimes K$ embeds into $\End_K(\mathcal{L}(\mathcal{A} \otimes K))$ whenever $K$ is a polynomial ring. 
 In addition, note that two homogeneous maps that agree with each other on all polynomial rings, have to agree on all $R$-algebras.
	Thus, we can as well verify all equalities in $\End_K(\mathcal{L}(\mathcal{A} \otimes K))$. Using that cubic norm pairs define cubic norm pairs over all scalar extensions, it is sufficient to verify the equalities in $\End(\mathcal{L}(\mathcal{A}))$ for arbitrary cubic norm pairs.
\end{remark}

\begin{lemma}
	\label{lem: TKK rep}
	Suppose that $\gamma$ is a non-zero-graded root and write $U_\gamma$ for the subspace of $\mathcal{L}(\mathcal{A})$ which is $\gamma$ graded, then 
	\[ \exp_{\pm \gamma,[t]}(x) : U_{-\gamma} \oplus X \oplus U_{\gamma} \longrightarrow U_{-\gamma} \oplus X \oplus U_{\gamma} \]
	defined by $\exp_{\pm \gamma,[t]}(x) = 1 + t\ad x + t^2Q_x$
	is a divided power representation. 	Moreover, for $y \in U_{-\gamma}$ we have
	\[ \exp_{\gamma, [t]}(x) (\ad y) \exp_{\gamma, [t]}(-x) = (\ad y) + t \ad [x,y] + t^2 \ad (Q_x y) = \ad (\exp_{\gamma, [t]}(x) \cdot y)\]
	in $\End(\mathcal{L}(\mathcal{A}))$.
	\begin{proof}
		One easily computes that $V_{x,y} z = D_{x,y} z$, using the operator $D$ for quadratic Jordan pairs whenever $\gamma$ is $\pm 1$-graded in the usual $\mathbb{Z}$-grading, using Lemma \ref{lem: Vs}. Using the identification $(R_{2 \beta + 3\gamma}, R_{-2 \beta - 3 \gamma})$ with $(\mathcal{S}_2, \mathcal{S}_{-2})$, allows one to show that $[s_{+2},s'_{-2}] = - (ss') \zeta$.		
		This shows that $Q^{(1,1)}_{r,s} = (\ad r)(\ad s)$. Hence, we obtain a binomial divided power representation. 
		
		The defining Equation (\ref{def: dprep}) for divided power representations, follows for $(k,l) = (3,1)$ and $(k,l) = (4,2)$ from the same equations for the usual TKK representation \cite[Example 3]{faulkner2000jordan}. 
		The case $(k,l) = (2,1)$ corresponds to verifying \cite[Equation (12)]{faulkner2000jordan}.
		The verification on zero graded $D$, follows since the space $X$ acts as derivations of the Jordan pair.
		On $-1$-graded $w$, the resulting computation yields that it holds if and only if		
		\begin{equation} \label{eq: help eq} [Q_x y, w] + [Q_x w, y] = [x, D_{y,x} w].\end{equation}
		
		Whenever $U_\gamma \cong R$, this follows immediately from (JP2) since $[x,y] \mapsto D_{x,y}$ is injective for $(x,y) \in R_{\gamma} \times R_{-\gamma}$.
		Whenever $U_\gamma = J_\gamma$, we need to use that
		\[ V_{v,w} \begin{pmatrix}
			a & b \\ c & d
		\end{pmatrix} = \begin{pmatrix}
		0 & D(v,w) b \\ - D_{w,v} c + T(v,w)c & T(v,w) d
	\end{pmatrix}\]
		and
		\[ V_{w,v}  \begin{pmatrix}
			a & b \\ c & d
		\end{pmatrix} = \begin{pmatrix}
		T(v,w) a & T(v,w) b - D_{v,w} b \\ D_{w,v} c & 0
	\end{pmatrix},\]
		to see that $[v,w] \mapsto ((D_{v,w} , -D_{w,v}), T(v,w))$ is injective for $v \in J_\gamma$ and $w \in J'_{-\gamma}$. 
		For the projection $(D_{v,w} , - D_{w,v})$, Equation \ref{eq: help eq} follows from (JP2).
		The other projection is also compatible, since
		\begin{align*} T(Q_x y,w) + T(Q_x w, y) & = T(T(x,y)x - x^\sharp \times y, w) + T(y, T(x,w)x - x^\sharp \times w) \\ & = 2 T(x,y)T(x,w) - T(x \times x, y \times w) \\ & = T(T(y,x)w + T(w,x)y - (y \times w) \times x, x) \\ & = T(x, D_{y,x} w).\end{align*}
		The moreover-part follows from the definition of a divided power representation.
	\end{proof}
\end{lemma}

\begin{construction}
	\label{def: dp rep R}
	Consider a pair $(R_\gamma, R_{-\gamma})$. We can define a divided power representation of the Jordan pair $(R,R)$ on $\mathcal{L}(\mathcal{A})$.
	First, on $R_{-\gamma} \oplus X \oplus R_{\gamma}$ we act as in Lemma \ref{lem: TKK rep}, using
	\[ \exp_{\gamma,[t]}(r) = 1 + t \ad r + t^2Q_r\]
	with $Q_r : R_{-\gamma} \longrightarrow R_{\gamma}$ and the similarly defined $\exp_{-\gamma}$.
	
		Now, write $U_\delta$ subspace of $\mathcal{L}(\mathcal{A})$ which is $\delta$-graded.
	For the short roots $\delta$ which are orthogonal to $\gamma$, we use the trivial representation on $U_\delta$.
	Third, other roots come in pairs $\delta, \gamma + \delta$.
	On $U_{\gamma + \delta} \oplus U_\delta$, we act as
	\[ \exp_{\gamma,{[t]}(r)} = \begin{pmatrix}
	1 & t\ad r\\
	& 1\\
\end{pmatrix}, \quad \exp_{-\gamma,[t]}(r) = \begin{pmatrix}
1 \\ t\ad r & 1
\end{pmatrix}.\]
We remark that this representation is a quotient of the ``natural" represenation in $M_2(R)$.
	This defines a divided power representation on $\mathcal{L}(\mathcal{A})$, as it is the direct sum of divided power representations.
\end{construction}

Since $\exp_{\pm \gamma,[t]}$ is a polynomial in $t$, we will often write $\exp_{\pm \gamma}$ for the $\exp_{\pm \gamma,[1]}$.

\begin{lemma}
	For each pair $(R_\gamma,R_{-\gamma})$, consider the grading on $\mathcal{L}(\mathcal{A})$ induced by the root $\gamma$. There exists a grading reversing map \[\tau_\gamma = \exp_\gamma(1) \exp_{-\gamma}(1) \exp_{\gamma}(1).\]
	Moreover, since the $\exp_{\pm \gamma}(r)$ are automorphisms for all $r$, $\tau_\gamma$ is an automorphism as well.
	\begin{proof}
		A straightforward computation shows that $d \in X$ maps to the element $\tau_i(d) = d + [[d,1_\gamma],1_{- \gamma}] \in X$. A similar computation shows $\tau_i(r_{\pm \gamma}) = r_{\mp \gamma}$. 
	Hence $\tau_i$ reverses the grading on $R_{- \gamma} \oplus X \oplus R_{\gamma}$.
		
		An easy computation involving the matrices that define the divided power representation, which were introduced in Definition \ref{def: dp rep R}, also shows that $\tau_i$ reverses the grading on the rest of the Lie algebra.
		
		So, what is left to prove is that the $\exp_{\pm \gamma}(r)$ are automorphisms.		
		We prove it first for $Rt \cong R_{\beta + 3 \alpha}$. We remark that $\exp_+(r) = \exp_\gamma(r) = 1 + \ad rt + Q_{rt}$, where we wrote $rt$ instead of $r$ to make it clear that we are working with $Rt \subset \mathcal{A}_1$. 
		Note that $\exp_+(r)$ is of the form $1 + r_2 + r_4$ with $r_2 = \ad rt$ and $Q_{rt} = r_4$ where $r_i$ acts as $+i$ on the grading corresponding to $\beta + 3 \alpha$ (i.e., the dotted grading).

		For $\ad rt$ we know that the Jacobi identity holds, hence we know that
		\[ r_2 [a,b] = [r_2a,b] + [a,r_2b].\]
		We also know that $[r_4 a, r_2 b] + [r_2 a, r_4 b] = 0$, since either term being nonzero implies that we can assume $a,b \in R_{- \beta - 3 \alpha}$ and $D_{r,a} Q_r b = Q_r D_{a,r} b = Q_r D_{b,r} a = D_{r,b} Q_r a$ in any quadratic Jordan pair by (JP1). 
		We also note that $[r_4a,r_4 b] = 0$ for all $a$ and $b$ by the grading.
		If we can show that
		\[ r_4[a,b] = [r_4a,b] + [r_2a,r_2b] + [a,r_4 b],\]
		then we will have shown that $\exp_+(r)$ is an automorphism, since all homogeneous parts of
		\[ \exp_+(r) [a,b] = [\exp_+(r)a, \exp_+(r)b]\]
		match.

		We remark that this definitely holds if either $a,b \in R_{- \beta - 3 \alpha}$, i.e., either $a$ or $b$ is $-2$-graded, by Lemma \ref{lem: TKK rep}
		
		Suppose now that neither $a$ nor $b$ is $-2$-graded. If they are not both $-1$-graded, both sides are $0$ due to the grading. Hence, assume that $a$ and $b$ are $-1$-graded.
		We can assume that $a$ and $b$ are homogeneous with respect to the $G_2$-grading.
		We note that $a \in U_{-\alpha - k(\beta + \alpha)}$ and $b \in U_{-\alpha - (l - k)(\beta + \alpha)}$ for some $k$ and $l$. If $l \neq 1$, both sides are $0$ due to the grading. 
		If $k = 0$ (or $k = 1$ and if one reverses the role of $a$ and $b$), we have \[Q_{rt} [ (-L_{a},L_{a}) , b] = T(ra,rb) t = [ ra, (-L_{rb}, L_{rb})] = [[rt, (-L_{a},L_{a})],[rt,b]].\]
		The case $k = -1$ or (k = 2) follows similarly.
		
		We conclude that
		\[ \exp_{\gamma}(r) [a,b] = [\exp_{\gamma}(r)a, \exp_{\gamma}(r)b]. \]
		For $(R_{-\beta },R_{\beta})$ we can argue analogously.
		For $(R_{2 \beta + 3 \alpha },R_{- (2 \beta + 3 \alpha)})$ we can use that $\tau_\beta \cdot R_{\pm (\beta + 3 \alpha)} = R_{\pm (2 \beta + 3 \alpha)}$ and $\tau_{2\beta + 3 \alpha} = \tau_{\beta} \tau_{\beta + 3 \alpha} \tau_{\beta}^{-1}$.
	\end{proof}
\end{lemma}

\begin{construction}
	\label{def: action J2}
	Now, we construct the divided power representation for $(J'_{\beta + 2 \alpha}, J_{- \beta - 2 \alpha})$. We write $\gamma$ for $\beta + 2 \alpha$.
	
	As before, on $J_{-\gamma} \oplus X \oplus J_{\gamma}$ we use the TKK representation as $\exp_{\gamma,[t]}$.
	Now, consider a root $\delta$ orthogonal to $\gamma$. On the corresponding spaces $R_\delta$ we use the trivial representation.
	
	On
	$R_{- \beta - 3 \alpha} \oplus J'_{- \alpha} \oplus J_{\beta + \alpha} \oplus R_{2 \beta + 3\alpha}$ we use a divided power representations isomorphic to the one of Example \ref{ex: dprepr}.
	The isomorphism is determined by
	\[ \exp_{{\beta + 2 \alpha,[x]} }(j^*)(\bar{t}_{- \beta - 3 \alpha}) = \bar{t}_{- \beta - 3 \alpha} + x (-L_{j^*},L_{j^*})_{- \alpha} + x^2 (j^{*\sharp})_{\beta + \alpha} - x^3 (N(j)L_{(t - \bar{t})})_{2 \beta + 3 \alpha}\]
	and
	\[\exp_{- \beta - 2 \alpha,[x]}(j)((t - \bar{t})_{2 \beta + 3 \alpha}) = (L_{(t - \bar{t})})_{2 \beta + 3 \alpha}+ x j_{ \beta + \alpha } + x^2 (L_{j^\sharp},- L_{j^\sharp})_{- \alpha} +  x^3 N(j) \bar{t}_{- \beta - 3 \alpha}.\]
	on $R_{- \beta - 3 \alpha} \oplus J'_{- \alpha} \oplus J_{\beta + \alpha} \oplus R_{2 \beta + 3\alpha}$.
	When we act on \[R_{- 2 \beta - 3 \alpha } \oplus J'_{- \beta - \alpha} \oplus J_{\alpha} \oplus R_{\beta + 3 \alpha},\] we use $\tau_\beta \exp_{- \beta - 2 \alpha}(j) \tau_\beta^{-1}$ as definition.
	
	The signs used in these representations are precisely the ones such that
	\[ \exp_{{\beta + 2 \alpha, [\epsilon]}}(j) = 1 + \epsilon (\ad j) \mod (\epsilon^2).\]
\end{construction}

\begin{lemma}
	The divided power representation of $(J'_{\beta + 2 \alpha },J_{- \beta - 2 \alpha})$ is a representation as automorphisms $\exp_{\pm(\beta + 2 \alpha)}(j)$ of the Lie algebra.
	\begin{proof}
		We remark that the Lie algebra is generated by $\mathcal{A}_1$ and $\mathcal{A}_{-1}$.
		In fact, only the subspaces $R_{\pm \beta}, R_{\pm ( \beta + 3 \alpha)},$ and $J'_{\beta + 2 \alpha}$ and $J_{- \beta - 2 \alpha}$ are necessary to generate the Lie algebra.
		The action on $R_{\beta}$ and $R_{- \beta}$ of $(J'_{\beta + 2 \alpha },J_{- \beta - 2 \alpha})$ is trivial. It is easily checked that the action of $(J'_{\beta + 2 \alpha},J_{- \beta - 2 \alpha})$ commutes with $\ad R_\beta$ and $\ad R_{- \beta}$, using the definition of $\exp_{\pm (\beta + 2 \alpha)}$ and the grading.
		
		For brevity, we will write $\gamma$ for $\beta + 2 \alpha$.		
		Furthermore, the action of $(J'_{\gamma},J_{- \gamma})$ interacts correctly with $\ad J'_\gamma$ and $\ad J_{- \gamma}$, i.e., \[\exp_\gamma(j) (\ad \; u) \exp_{\gamma}(j)^{-1} = \ad \; \exp_{\gamma}(j) \cdot u,\]  by Lemma \ref{lem: TKK rep}.
		
		So, we only have to prove that $\exp_\gamma(j) \ad r \exp_\gamma(- j) = \ad (\exp_\gamma(j) \cdot r)$ for $r \in R_{\beta + 3 \alpha} \oplus R_{- \beta - 3 \alpha}$ and similarly for $\exp_{-\gamma}(j)$.
		
		 For $r \in R_{\beta + 3 \alpha}$ this follows from \[\exp_\gamma(j) \ad r \exp_\gamma(- j) = \ad r + f\] with $f$ acting as $+ 2 \beta + 5 \alpha$ on the grading, as other parts are at least $3 \beta + 7 \alpha$ graded and therefore act trivially. The element $f$ has to be $\ad [j,r]$ by the Jacobi identity. Hence, $f = 0$. 
		 
		 Now, we verify that
		 \[ \exp_\gamma(j) (\ad a) \exp_{\gamma} (-j) = \ad (\exp_\gamma(j) \cdot a) \]
		 for $a \in R_{- \beta - 3 \alpha}$.
		 This is equivalent to verifying that $\exp_\gamma(j)$ acts as an automorphism on elements of the form $[a,b]$ with $a \in R_{- \beta - 3 \alpha}$ and $b$ an arbitrary element contained in the Lie algebra. We do this by assuming that $b$ is homogeneous with respect to the grading and considering each of the possibilities for $b$ individually, in increasing order of difficulty. 
		
		For $b \in R_{\beta}$ and $R_{2 \beta + 3\alpha}$ this holds trivially by the grading. For $b \in R_{- \beta - 3 \alpha}$ we note that $[\exp(j_\gamma) a, \exp(j_\gamma) b] = 0$ since $a$ and $b$ are linearly dependent.
		
		For $b \in J_{\beta + \alpha}, R_{\beta + 3 \alpha}$ this is easily checked using the Jacobi identity, since
		\[ \exp_\gamma(j) [a,b] \in [a,b] + L_{\delta + \gamma}\]
		with $\delta$ root corresponding to $[a,b]$.
		
		For $b \in R_{-\beta}$ it follows from the definition of $\exp(j_\gamma)$ used on $R_{- \beta - 3 \alpha}$ involving $\tau_\beta$, since the restriction to $R_{-\beta - 3 \alpha}$ of $\tau_\beta$ is $\ad \pm 1_{- \beta}$.
		
		For $b \in J'_{2 \alpha + \beta}$ and $J_{- 2 \alpha - \beta}$ it follows from the fact that we have a divided power representation by Lemma \ref{lem: TKK rep}. For $b \in X$ we can use that $X = [J'_{\gamma},J_{-\gamma}] + R^2$ with $R^2$ the natural torus corresponding to the $R_\delta$, with which $\exp_{\gamma}(j)$ interacts nicely, so that $\exp_{\gamma}(j)$ interacts nicely with $X$ by Lemma \ref{lem: TKK rep}.
		
		For $b \in J_{\alpha}, J'_{- \beta - \alpha}$ and $R_{- 2 \beta - 3 \alpha}$ the check is a bit more involved.
		A direct computation shows that \[\exp_\gamma(j) [\bar{t}_{- \beta - 3 \alpha}, (L_{v},- L_{v})] = \exp_\gamma(j) v_{- \gamma} = v_{- \gamma} + (- V_{j,v}, V_{v,j}) + (Q_{j} v)_{\gamma}. \]
		On the other hand, if we compare this to 
		\[[\exp_\gamma(j) \cdot \bar{t}_{- \beta - 3 \alpha}, \exp_\gamma(j) \cdot (L_v,- L_v)] =   - L_{v} (j^\sharp) + L_{j^*} T(j,v)  \mod X \oplus J_{- \beta - 2 \alpha},\]
		using \[ - L_{v} (j^\sharp) + L_{j^*} T(j,v) =-  (j^\sharp \times v - T(j,v)j)^* = (Q_j v)^*,\]
		we obtain that $\exp_\gamma(j)$ interacts as an automorphism with $[a,b]$ when $b \in J_\alpha$ (the part in the spaces we dropped will always be correct, as the only check reduces to the Jacobi identity).
		A similar computation works for $b \in J'_{- \beta - \alpha}$. Namely, one should only check whether the parts of the results in $X \oplus J'_{\beta + 2 \alpha}$ are what they should be. For $J'_2$ this follows from $(a^\sharp)^\sharp = N(a)a$. 
		A direct computation for the $X$-part yields the same, if one uses that $- T(a,b) V_{t,\bar{t}} + [L_a,L_{b^*}] = V_{a,b^*}$.
		
		Lastly, we should check what happens if $b \in R_{- 2 \beta - 3 \alpha}$.
		Using the automorphism $\tau_{\beta + 3 \alpha}$ and $R$-linearity, we can assume that we want to check
		\[ [ t_{\beta + 3 \alpha} + j^*_{\beta + 2\alpha} + j^\sharp_{\beta + \alpha} + N(j)\bar{t}_{- \beta}, t_{- \beta} - j^*_{ - \beta - \alpha} + j^\sharp_{- \beta - 2 \alpha} - N(j)\bar{t}_{- \beta - 3 \alpha}] = 0.\]
		Checking the parts which are homogeneous in $j$ and not trivially zero by the grading, we obtain that we should verify
		\[ V_{t,j^\sharp} - V_{j^*,j^*} + V_{j^\sharp,t} = 0,\]
		\[ N(j)(V_{\bar{t},t} - V_{t,\bar{t}}) + V_{j^*,j^\sharp} - V_{j^\sharp,j^*} = 0,\]
		and
		\[ - N(j) (V_{\bar{t},j^*} + V_{j^*,\bar{t}}) + V_{j^\sharp,j^\sharp} = 0. \]
		For the first equation, we have
		\[ V_{t,j^\sharp} + V_{j^\sharp,t} - V_{j^*,j^*} = L_{j^\sharp} + L_{j^\sharp} - L_{j \times j} = 0.\]
		For the second, we remark that $D_{j^\sharp,j} a = - (a \times j^\sharp) \times j + T(a,j) j^\sharp + 3 N(j) a = 2 N(j) a$ by the third axiom for cubic norm pairs and that $D_{j,j^\sharp} a = 2 N(j) a$ as well, so that
		\[ V_{j^*,j^\sharp} - V_{j^\sharp,j^*} = N(j) M(\begin{pmatrix}
			3 & 1 \\ -1 & -3
		\end{pmatrix})\]
		where $M$ stands for entry-wise matrix multiplication instead of the ordinary matrix multiplication. Computing $N(j)(V_{\bar{t},t} - V_{t,\bar{t}})$, yields the desired result.
		The last equation follows from the first equation if one substitutes $j^\sharp$ for $j$ and uses that $(j^\sharp)^\sharp = N(j) j$.
		
		This finalizes the proof that $\exp_\gamma(j) (\ad \; a )\exp_\gamma(j)^{-1} = \ad \;(\exp_\gamma j \cdot a)$ for $a$ in a certain generating set. Hence, $\exp_\gamma(j)$ is an automorphism. Analogous argumentation shows that $\exp_{- \gamma}(j)$ is an automorphism as well.
	\end{proof}
\end{lemma}

\begin{construction}
	For two pairs of opposite short roots $(\gamma, - \gamma)$ and $(\delta, - \delta)$ that correspond to the pair $(J',J)$, there exists a unique $\tau_\rho$ for a long root $\rho$ which maps one pair to the other pair. Hence, we have divided power representations for all $(J'_\gamma, J_\gamma)$.
\end{construction}

\subsection{Commutator relations} Now, we determine how the groups $R_\gamma, J_\gamma$ and $J'_\gamma$, as long as these spaces do not have opposite grading, commute. First, we prove some lemmas that diminish the amount of necessary computations. Afterwards, we determine the commutators.

First, we try to get a grip on the possible derivations of our Lie algebra.
To begin, it is useful to prove that homogeneous parts of derivations are derivations. Afterwards, we can decompose possibly annoying derivations into manageable parts.
Once we have a solid grasp on the derivations, we can start to compute the commutator relations. 

\begin{lemma}
	Suppose that $d$ is a derivation of a $\mathbb{Z}$-graded algebra $A = \bigoplus A_i$ with projection operators $\pi_i : A \longrightarrow A_i$, then each of the homogeneous components of $d$ is a derivation as well.
	\begin{proof}
			If $a,b \in A$ are homogeneous of degree $k$ and $l$, then we have
		\[ d_i(ab) = \pi_{k + l + i}d(ab) = \pi_{k + l + i}(d(a)b + ad(b) ) = d_i(a)b + ad_i(b),\]
		using $d_i$ to denote the part of $d$ that maps $A_j \longrightarrow A_{i + j}$ for all $j \in \mathbb{Z}$ and $\pi_i$ for the projection $A \longrightarrow A_i$.
	\end{proof}
\end{lemma}

\begin{lemma}
	\label{lem: no deriv with deg > 2}
	Consider the $\mathbb{Z}$-grading of the Lie algebra $\mathcal{L}(\mathcal{A})$ induced by a long root $\gamma$ and let $\delta, \delta'$ be short roots such that $\delta + \delta' = \gamma$. Each derivation of homogeneous degree $> 2$ or $< -2$ decomposes into $\pm 3\delta$-graded and $\pm 3\delta'$-graded derivations.
	Any such derivation is uniquely determined by its restriction to $R_{\pm\gamma}$.
	\begin{proof}
		Suppose, without loss of generality, that we look at the usual $\mathbb{Z}$-grading (i.e., $\gamma = \pm (2 \beta + 3 \alpha)$).
		The Lie algebra is generated by $R_{\beta}$, $R_{\beta + 3 \alpha}$, and $\mathcal{A}_{-1}$.		
		Hence, any derivation which acts as zero on these spaces, acts trivially on the whole Lie algebra. In particular, any derivation that acts as at least $-3$ on the grading must be $- 3 \delta$ or $-3 \delta'$ graded for the two short roots such that $\delta + \delta' = 2 \beta + 3 \alpha$, since these are $- (2 \beta + 3 \alpha) - \beta$ and $- (2 \beta + 3 \alpha) - (\beta + 3 \alpha)$.
		Considering the generating set $\mathcal{A}_{-1} \oplus R_{2 \beta + 3 \alpha}$, shows that the derivation is uniquely determined by its restriction to $R_{2 \beta + 3 \alpha}$.
		
		A similar argument works for derivations that are at least $+3$-graded. 
	\end{proof}
\end{lemma}

\begin{lemma}
	\label{lem: deriv deg 2 easily computed}
	Let $\gamma = 2 \alpha + \beta$.
	Suppose that $d$ is a derivation with grading $2 \gamma$, then $d$ is uniquely determined by $d(1)$ for $1 \in R_{- 2\beta - 3 \alpha}$ (or $1 \in R_{- \beta - 3 \alpha}$).
	The same holds for all other short roots $\gamma$, i.e., then the action of any $2 \gamma$-graded derivation on $1_\delta$ for any $\delta$ which is $-3$-graded with respect to $\gamma$ determines the derivation uniquely.
	\begin{proof}
		Using that $R_{- \beta - 3 \alpha}$ and $\mathcal{A}_1$ generate the Lie algebra and that $d$ is a derivation, shows that the lemma holds. For different roots, one finds similar generating sets.
	\end{proof}
\end{lemma}

\begin{lemma}
	\label{lem: explanation easy computation}
	\begin{enumerate}
		\item Consider $R_\gamma$, An element $e = 1 + e_1 + e_2$ is an automorphism of $\mathcal{L}(\mathcal{A})$ with $e_i$ having grading $i \gamma$, if and only if $e = \exp_\gamma(r)$ with $e_1 = \ad r$.
		\item 	If $\gamma$ is a root corresponding to a $J$ (or $J'$), and $e = 1 + e_1 + e_2 +e_3 + e_4$ is an automorphism with $e_i$ having grading $i \gamma$, then $e$ is uniquely determined by $(e_1,e_2,e_3)$. Furthermore, we can uniquely decompose $e = \exp_\gamma(j) f$ with $j \in J_\gamma$ (or $J'_\gamma$) and a certain automorphism $f = 1 + f_2 + f_3 + f_4$, with $f = 0$ if and only if the derivations $f_2, f_3$ are $0$.
		\item If $\gamma$ is not $0$, not a root, nor twice a short root, nor thrice a short root, then all automorphisms of the form $e = 1 + e_1 + e_2 + e_3 + \dots $ with $e_i$ acting as $+ \gamma i$ on the grading are trivial, i.e., $e = 1$.
	\end{enumerate}	
	\begin{proof}
		This follows from Lemma \ref{lem: no deriv with deg > 2}, using that two automorphisms $e, e'$ satisfying the conditions of this lemma can be used to determine a new automorphism $f = e(e)'^{-1}$ satisfying the same conditions. We apply this technique to each of the cases.
		
		We remark that the first nonzero $g_i$ in $g = 1 + \sum_i g_i$ has to be a derivation for $g$ to be an automorphism, which follows from comparing homogeneous parts in $[g \cdot a, g \cdot b] = g \cdot [a,b]$. 
		
		For the first case, note that $e_1 = \ad r$ where $r$ is the unique element such that $\pm e_1 = [e_1, [1_{\delta},1_{-\delta}]] = \pm r$, as long as $\delta$ is a long root different from $\pm \gamma$, with the sign in front of $e_1$ and $r$ depending on $\delta$ and $\gamma$.
		Now, $e \exp_\gamma( - r) = 1 + \tilde{e}_4$. We conclude that $\tilde{e}_4 = 0$ by Lemma \ref{lem: no deriv with deg > 2}.
		
		For the second case, we can also use $\pm e_1 = [e_1, [1_{\delta},1_{{-\delta}}]] = \pm j$, as long as $J_\gamma$ (or $J'_\gamma$ ) is not trivially graded with respect to $\delta$.
		Now, $e \exp_\gamma(-j) = 1 + f_2 + f_3 + f_4$. One verifies that $f_2$ and $f_3$ are derivations. Moreover, if they are $0$, Lemma \ref{lem: no deriv with deg > 2} shows that $f_4 = 0$.
		
		In the last case, one can show inductively that all $e_i$ are $0$ using Lemma \ref{lem: no deriv with deg > 2}.
	\end{proof}
\end{lemma}

\begin{definition}
	We call $(J,J')$ a \textit{well behaved} cubic norm pair if $N(v^\sharp) = N(v)^2$ holds over all scalar extensions.
	
	We remark that all unital cubic norm pairs are well behaved, using \[N(v)^4 = N(Q_v v^\sharp) = N(v)^2 N(v^\sharp)\] on the $v \in J \otimes K$ such that $N(v)$ is invertible and applying \cite[Proposition 12.24]{Skip2024}.
	In the non-unital case, one has $D_{j,j^\sharp} a = T(a,j^\sharp) j + 3N(j) a - (j \times a) \times j^\sharp = 2 N(j) a$ and $D_{j^\sharp,j} a = 2 N(j) a$.
	From $V_{N(j)j,j^\sharp} = V_{(j^\sharp)^\sharp,j }$ as computed in Lemma \ref{lem: Vs}, one concludes that  $x = N(j^\sharp) - N(j)^2$ acts trivially on the spaces $J$ and $J'$, and $3x = 0$.
\end{definition}

\begin{lemma}
	\label{lem: lin Nvsharp}
	Suppose that $(J,J')$ is a well behaved cubic norm pair.
	We have
	\[ N(v \times w) = - N(v)N(w) + T(v,w^\sharp)T(w,v^\sharp).\]
	\begin{proof}
		First, observe that
		$N(v^\sharp + v \times w + w^\sharp) = (N(v) + T(v,w^\sharp) + T(w,v^\sharp) + N(w))^2$.
		Comparing terms of the same degree yields \[N(v \times w) + T(v \times w, v^\sharp \times w^\sharp) = 2 N(v)N(w) + 2 T(v,w^\sharp) T(w,v^\sharp)\]
		Now, $v^\sharp \times (v \times w) = N(v) w + T(w,v^\sharp )v$ by the $(3,1)$-linearisation of axiom (2).
		Hence, $T(v \times w, v^\sharp \times w^\sharp) = 3N(w)N(v) + T(w,v^\sharp) T(v,w^\sharp)$.
		Thus, we obtain the desired result.
	\end{proof}
\end{lemma}

Recall that $r \in R_{2\beta + 3 \alpha}$ represents $ r L_{t - \bar{t}} = -r [t,\bar{t}]$ and $[v,w] = - T(v,w) L_{t - \bar{t}}$ for $v \in J$ and $w \in J'$.
We also note that we use $j \in J_{\pm \alpha}$ to represent $(L_j,-L_j)$.

\begin{lemma}
	\label{lem: com formulas}
	Assume that $(J,J')$ is a well behaved cubic norm pair.
	We have the following relations in $\Aut(\mathcal{L}(\mathcal{A}))$, in which we wrote $x_\gamma$ for $\exp_\gamma(x)$ when $x = j,k,r,r'$ for $j \in J$, $k \in J'$, $r, r' \in R$:
	\begin{enumerate}
		\item $j_{\beta + \alpha} k_{\beta + 2 \alpha} = k_{\beta + 2 \alpha} (-T(j,k))_{2 \beta + 3 \alpha} j_{\beta + \alpha}$,
		\item $j_{\beta + \alpha} k_{\alpha} = k_{\alpha} (T(k^\sharp,j))_{\beta + 3 \alpha} (- k \times j)_{\beta + 2 \alpha} (T(k,j^\sharp))_{2 \beta + 3 \alpha } j_{\beta + \alpha}$,
		\item $j_\alpha r_{- \beta - 3 \alpha } = r_{- \beta - 3 \alpha } (- rj)_{-\beta - 2 \alpha} ( r^2 N(j))_{- 2 \beta - 3 \alpha} (r j^\sharp)_{- \beta - \alpha} (- r N(j))_{- \beta }j_\alpha$,
		\item $j_{\alpha} r_{2 \beta + 3 \alpha } = r_{2 \beta + 3 \alpha} j_{\alpha}$,
		\item $ r_{\beta + 3\alpha} r'_{\beta} = r'_{\beta} (-rr')_{2 \beta + 3 \alpha } r_{\beta + 3 \alpha}$,
		\item $j_\alpha r_{\beta + 3 \alpha } = r_{\beta + 3 \alpha} j_\alpha$.
	\end{enumerate}
	Use $U_\gamma$ to denote $R_\gamma, J_\gamma$, or $J'_\gamma$ depending on $\gamma$. 
	Whenever we have groups $U_\gamma$ and $U_\delta$ such that $\gamma + \delta \neq 0$, there exists an analogous relation showing that
	\[ U_{\gamma} U_\delta \subset U_\delta \left(\prod_{(k,l )\in \mathbb{N}^2 \setminus \{(0,0)\}} U_{k \delta + l \gamma}\right)U_\gamma\]
	going over all roots $k \delta + l \gamma$ between $\delta$ and $\gamma$ in increasing order\footnote{This is not the order employed in the relations. However, such mismatches correspond to commuting groups.} of $k/l$.
	Moreover, equations not involving $\times$ or $\sharp$ hold for arbitrary cubic norm pairs.
	\begin{proof}
		We first prove (5).
		Consider rescaled automorphism \[t \cdot r'_{\beta} = 1 + (\ad r'_\beta) \otimes t + (r'_\beta)_2 \otimes t^2\] of $\mathcal{L}(\mathcal{A}) \otimes R[s,t]$, with $(r'_\beta)_i$ acting as $+ i \beta$ on the grading.
		In this way, we obtain an automorphism $({sr})_{\beta + 3 \alpha} ({tr'})_{\beta}$ of $\mathcal{L}(\mathcal{A}) \otimes R[s,t]$.
		We can write
		\[ ({sr})_{\beta + 3 \alpha} ({tr'})_{\beta} = \sum_{k,l \in \mathbb{N}} s^kt^l f_{k,l}\]
		with $f_{k,l}$ acting as $+ (k + l) \beta + 3 k \alpha$ on the grading. 
		We can decompose this into a product
		\[ \prod_{(k,l) \in S} \left(1 + \sum_{i = 1}^\infty s^{ki}t^{li} \tilde{f}_{ki,li}\right)\]
		where $S$ denotes the ordered set of coprime $(k,l)$ with $(a,b) < (c,d)$ if $a/b < c/d$ (and maximum $(1,0)$). Induction on $a + b$ shows $\tilde{f}_{a,b}$ to be well defined.
		Reduction $\mod (s^{k + 1}, t^{l + 1})$ in the induction, also shows that the formal power series
		\[ a_{k,l}(s,t) = \left(1 + \sum_{i = 1}^\infty s^{ki}t^{li} \tilde{f}_{ki,li}\right)\]
		in $\End(\mathcal{L}(\mathcal{A}))[[s,t]]$ define automorphisms of $\mathcal{L}(\mathcal{A})[[st]]$.
		Using the grading, it is easy to verify that $a_{k,l}(s,t)$ is a polynomial of degree at most $4(k + l)$, so that the automorphism $a_{k,l}(1,1)$ exists.
		Hence, the first nonzero $\tilde{f}_{ik,il}$ with $i > 0$ is a derivation for each $(k,l) \in S$.
		Lemma \ref{lem: explanation easy computation} shows that only $a_{1,0}, a_{0,1},$ $a_{1,1}$, $a_{2,1}$ and $a_{1,2}$ can differ from $1$, these are the only options such that $(k + l) \beta + 3k \alpha$ corresponds to a long root or a multiple of a short root. Lemma \ref{lem: explanation easy computation} shows that $a_{1,1}$ lies in $R_{2 \beta + 3 \alpha}$. Using $f_{1,0} f_{0,1} = f_{1,1}$ and $\tilde{f}_{1,1} = f_{1,1} - f_{0,1} f_{1,0}$, shows that $\tilde{f}_{1,1} = \ad [r t, r' \bar{t}] = \ad - r r' L_{t - \bar{t}}$.
		The elements $\tilde{f}_{2,1}$ and $\tilde{f}_{1,2}$ have to be $0$ since $r_{\beta + 3\alpha}$ and $r'_{\beta}$ are automorphisms. Namely, for $\tilde{f}_{2,1}$ one has 
		\[ \ad \left( r'_{\beta} + [r,r']_{2 \beta + 3 \alpha} \right) =  r_{\beta + 3\alpha} (\ad r'_{\beta})  (-r)_{\beta + 3\alpha} = \ad r'_{\beta} + \tilde{f}_{1,1} + \tilde{f}_{2,1}.\]
		For $\tilde{f}_{1,2}$ one argues analogously.
		
		Now, we consider the other cases.
		The same trick we used before of decomposing a modified left hand side of the equation we want to prove in the reverse order over $R[s,t]$, can be used to show the other cases.
		Lemma \ref{lem: explanation easy computation} is useful, since all $a_{k,l}(s,t)$ lie in groups \[F_{\gamma} = \{ 1 + f_1 + f_2 + f_3 + f_4 : f_i \text{ acts as } + \gamma \text{ on the grading}\}\] with $\gamma$ a root. Whenever $\gamma$ is long, we should only compute $\tilde{f}_{k,l}$ to fully determine $a_{k,l}$. If $\gamma$ is short, we should compute $\tilde{f}_{k,l} = \ad j$ and verify whether $\tilde{f}_{2k,2l}$ and $\tilde{f}_{3k,3l}$ coincide with the parts of $\exp_\gamma(j)$ that act as $+ 2 \gamma$ and $+3 \gamma$ on the grading.
		If we look at $x_\gamma y_{\gamma'}$ all $\tilde{f}_{k,1}$ coincide with the $k\gamma + \gamma'$-graded part of $\ad (x_\gamma \cdot y)$ where we consider $y$ as an element of the Lie algebra, since $x_\gamma$ is an automorphism of the Lie algebra. Similarly, $\tilde{f}_{1,k}$ has to coincide with the $\gamma + k \gamma'$ graded part of $\ad(y^{-1}_{\gamma'} \cdot x)$.
		
		These considerations immediately show that (1) and (6) hold, since there are no short roots between the roots $\gamma$ and $\gamma'$.
		For (4), note that only the short root $\beta + 2 \alpha$ can lead to problems.
		Observe that $a_{1,1}(s,t) = 1 + st \tilde{f}_{1,1} + \tilde{f}_{2,2}$ with $\tilde{f}_{i,i}$ being $2i (\beta + 2 \alpha)$-graded. Hence, the equation also follows since $\tilde{f}_{1,1} = \ad [j_\alpha , r_{2 \beta + 3 \alpha}] = 0$.
		
		Only equations (2) and (3) remain. These equations involve $\times$ or $\sharp$. Hence, the moreoverpart is proven.		
		For (2), we only need to verify that $\tilde{f}_{2,2}$ and $\tilde{f}_{3,3}$ act correctly since $\beta + 2 \alpha$ is the only short root between $\beta + \alpha$ and $\alpha$, since the terms occurring in the long roots are either of degree $1$ in $j$ or of degree $1$ in $k$ (i.e., correspond to $\tilde{f}_{1,i}$ or $\tilde{f}_{i,1}$). That $a_{1,2}$ and $a_{2,1}$ occur in the wrong order in (2) is not a problem, using an analogous equation to (6). 
		For (3), we need to verify that that all $a_{i,j}$ with $(i,j) \notin \{(1,0), (0,1), (3,1)\}$ coincide with what we wrote there.
		
		\textbf{Equation (2)}
		We start by proving (2). We note that \[\tilde{f}_{2,2} = f_{2,2} - \tilde{f}_{0,1}\tilde{f}_{2,1} - \tilde{f}_{1,2} \tilde{f}_{1,0} -  \tilde{f}_{0,2}\tilde{f}_{2,0}  - \tilde{f}_{0,1} \tilde{f}_{1,1} \tilde{f}_{1,0},\]
		where
		\begin{enumerate}
			\item $f_{2,2} \cdot  1_{R_{-\beta - 3 \alpha}} = \tilde{f}_{2,0} \tilde{f}_{0,2} \cdot  1_{R_{-\beta - 3 \alpha}} = Q_j k^\sharp$,
			\item $\tilde{f}_{0,1} \tilde{f}_{2,1} \cdot  1_{R_{-\beta - 3 \alpha}} = (\ad k_\alpha)(\ad T(k,j^\sharp)_{2 \beta + 3 \alpha}) \cdot  1_{R_{-\beta - 3 \alpha}} = (L_k,-L_k) \cdot T(k,j^\sharp) \bar{t}_1 = - T(k,j^\sharp) k$,
			\item $\tilde{f}_{1,2} \tilde{f}_{1,0} \cdot  1_{R_{-\beta - 3 \alpha}}  = 0$,
			\item $\tilde{f}_{0,2}\tilde{f}_{2,0} \cdot  1_{R_{-\beta - 3 \alpha}} = 0$,
			\item $\tilde{f}_{0,1} \tilde{f}_{1,1} \tilde{f}_{1,0} \cdot 1_{R_{-\beta - 3 \alpha}} = 0$,
		\end{enumerate}
		of which the last $3$ are forced to be $0$ by the grading.
		Combining these yields
		\[ \tilde{f}_{2,2} \cdot  1_{R_{-\beta - 3 \alpha}} = T(j,k^\sharp) j - j^\sharp \times k^\sharp + T(k,j^\sharp)k = (j \times k)^\sharp\]
		using the (2,2)-linearisation of axiom (2).
		
		Hence $\tilde{f}_{2,2}$ is what it should be by Lemma \ref{lem: explanation easy computation}, Lemma \ref{lem: deriv deg 2 easily computed}. 
		Now, for $\tilde{f}_{3,3}$, we once again evaluate on $1_{R_{-\beta - 3 \alpha}}$.
		As before, we will use that $\tilde{f}_{i,0} \cdot 1_{R_{-\beta - 3 \alpha}} = 0$ if $i > 0$. This shows that
		\[ (\tilde{f}_{3,3} + \tilde{f}_{1,2}\tilde{f}_{2,1} + \tilde{f}_{0,1} \tilde{f}_{1,1} \tilde{f}_{2,1}) \cdot 1_{R_{-\beta - 3 \alpha}} = f_{3,3} \cdot  1_{R_{-\beta - 3 \alpha}},\]
		since $\tilde{f}_{3,1} = 0$ and since $\tilde{f}_{3,2} = 0$ since it is a $3 \beta + 5 \alpha$-graded derivation. 
		The contribution of $\tilde{f}_{0,1} \tilde{f}_{1,1} \tilde{f}_{2,1}$ is $0$ by the grading.
		We conclude
		\[ \tilde{f}_{3,3} \cdot 1_{R_{-\beta - 3 \alpha}} = N(k)N(j)(t - \bar{t}) - (\ad T(k^\sharp,j)t) \cdot (- T(k,j^\sharp) \bar{t}) = - N(k \times j) (t - \bar{t})\]
		using Lemma \ref{lem: lin Nvsharp}. This shows that $\tilde{f}_{3,3}$ is what it should be by Lemma \ref{lem: deriv deg 2 easily computed}.
		
		\textbf{Equation (3)}
		Equation (3) is proved similarly. We should compute $\tilde{f}_{2,2}$, $\tilde{f}_{3,2}$, $\tilde{f}_{4,2}$, $\tilde{f}_{3,3}$, and $\tilde{f}_{6,3}$.
		We start with \[\tilde{f}_{2,2} = f_{2,2} - \tilde{f}_{0,1}\tilde{f}_{2,1} - \tilde{f}_{1,2} \tilde{f}_{1,0} - \tilde{f}_{0,2} \tilde{f}_{2,0} - \tilde{f}_{0,1} \tilde{f}_{1,1} \tilde{f}_{1,0}\]
		using $f_{2,2} = \tilde{f}_{2,0}\tilde{f}_{0,2}$.
		This element has grading $- 2 \beta - 4 \alpha$. We evaluate on $R_{2 \beta + 3 \alpha}$ and obtain
		\[ \tilde{f}_{2,2} \cdot 1_{2\beta + 3 \alpha} = 0 - (\ad r)_{-\beta - 3 \alpha} (-r j^\sharp)_{\beta + 2\alpha}  - 0 - 0 - 0\]
		using the grading.
		We conclude 
		\[ \tilde{f}_{2,2} \cdot 1_{2\beta + 3 \alpha} = r^2(L_{j^\sharp}, - L_{j^\sharp}).\]
		As before, this shows that $a_{1,1}$ coincides with what was claimed if $\tilde{f}_{3,3}$ is of the right form.
		
		Now,
		\[ (\tilde{f}_{3,2} + \tilde{f}_{0,1}\tilde{f}_{3,1} + \tilde{f}_{1,1}\tilde{f}_{2,1}) \cdot 1_{\beta + 3 \alpha} = f_{3,2} \cdot 1_{\beta + 3 \alpha} \]
		since all $\tilde{f}_{i,0}$ with $i > 0$ act trivially on $1_{\beta + 3 \alpha}$.
		We note that $\tilde{f}_{3,2} = \ad x_{- 2 \beta - 3 \alpha}$ for some $x$. This $x$ can be determined using the above evaluation. We also note that $\tilde{f}_{3,1}$ and $\tilde{f}_{2,1}$ are $-\beta$-graded and $-\beta - \alpha$-graded, hence the terms containing these $\tilde{f}_{i,j}$ will not contribute.
		We conclude
		\[ \tilde{f}_{3,2} \cdot 1_{\beta + 3 \alpha} = (j_\alpha)_3 (r_{- \beta - 3 \alpha})_2 \cdot 1_{\beta + 3 \alpha} = (L_j, - L_j)_3 \cdot (r^2 \bar{t})_{-1} = - N(j)r^2t.\]
		Hence, $\tilde{f}_{3,2} \cdot t_{1} = - N(j)r^2 L_{t - \bar{t}} \cdot t_{1} = ( N(j) r^2)_{- 2\beta - 3 \alpha} \cdot t_1$.
		
		The element $\tilde{f}_{4,2}$ can be determined by evaluating on $R_{2 \beta + 3 \alpha}$, as was the case for $\tilde{f}_{2,2}$.
		In this case 
		\[\tilde{f}_{4,2} \cdot 1_{2\beta + 3 \alpha} = - \tilde{f}_{1,1}\tilde{f}_{3,1} \cdot 1_{2\beta + 3 \alpha} = (\ad rj)_{- \beta - 2 \alpha} \cdot ( r N(j))_{\beta + 3 \alpha} = r^2N(j) ( L_j, - L_j)\]
		since $\tilde{f}_{3,1}$ is the only $\tilde{f}_{k,l}$ with $k/l > 2$ which can contribute and since $\tilde{f}_{4,0} = 0$.
		Hence, it acts as the part of $\exp_{- \beta - \alpha}(r j^\sharp)$ that acts as $- \beta - 2 \alpha$ on the grading.
		
		Now, we evaluate $\tilde{f}_{3,3}$ on $R_{\beta + 3 \alpha}$.
		We see that
		\[ (\tilde{f}_{3,3} + \tilde{f}_{0,1}\tilde{f}_{3,2} ) \cdot 1_{2 \beta + 3\alpha} = f_{3,3}  \cdot 1_{\beta + 3\alpha} = 0\]
		since all other $\tilde{f}_{k,l}$ with $k/l > 3/2$ act trivially.
		We conclude
		\[ \tilde{f}_{3,3} \cdot 1_{\beta + 3 \alpha} = - r^2N(j) (\ad r\bar{t})_{-1} \cdot t_{-1} = r^3 N(j) (t - \bar{t}).\]
		
		Finally, $\tilde{f}_{6,3}$ can be determined by evaluating on $R_{2 \beta + 3 \alpha}$.
		Using the grading, one sees
		\[ (\tilde{f}_{6,3} + \tilde{f}_{3,2} \tilde{f}_{3,1} + \tilde{f}_{1,1} \tilde{f}_{2,1} \tilde{f}_{3,1} + \tilde{f}_{2,2}\tilde{f}_{3,1} ) \cdot 1_{2 \beta + 3 \alpha} = 0\]
		using that $\tilde{f}_{i,0}$ acts trivially and thus forcing contributions of $\tilde{f}_{3i,i}$ with $i > 0$ for terms different from $\tilde{f}_{6,3}$ and $i < 2$ follows from the grading. We also note that $\tilde{f}_{2,2}\tilde{f}_{3,1} \cdot 1_{2 \beta + 3 \alpha} = 0 = \tilde{f}_{1,1} \tilde{f}_{2,1} \tilde{f}_{3,1} \cdot 1_{2 \beta + 3 \alpha} $.
		We conclude that $\tilde{f}_{6,3} \cdot  1_{2 \beta + 3 \alpha}$ equals
		\[ -  \tilde{f}_{3,2} \tilde{f}_{3,1} \cdot 1_{2 \beta + 3 \alpha} =(- r^2N(j) L_{t - \bar{t},-2}) (rN(j)t)_{\beta + 3 \alpha} = - r^3 N(j)^2 t_{-\beta - 3 \alpha}.\]
		Since $(J,J')$ is well behaved, this shows that $\tilde{f}_{6,3}$, and the part of $\exp_{-\beta - \alpha}(rj^\sharp)$ that act as $-3(\alpha + \beta)$ on the grading, act as $r^3N(j^\sharp)$.
		
		We conclude that (3) follows.

		The more general remark for groups $U_\gamma$ and $U_\delta$ can be proved analogously.
	\end{proof}
\end{lemma}

\begin{remark}
	The lemma above holds if and only if the cubic norm pair is well behaved, since being well behaved was necessary for the $\tilde{f}_{6,3}$ of equation (3) to act correctly.
\end{remark}

\begin{remark}
	\label{rmk: commutator relations}
	We will refer to the relations and the generalizations to arbitrary pairs of root groups proved in the previous lemma as the \textit{commutator relations}.
	This is justified, as we proved relations of the form $ xy = y g x$, which are equivalent to $y^{-1} x y x^{-1} = g$. Using $aba^{-1}b^{-1}$ as the commutator $[a,b]$, shows that $[y^{-1},x] = g$.
\end{remark}

\subsection{Tighter control using power series}

\begin{assumption}
	In this subsection, we assume that we work with a \textit{well behaved cubic norm pair}, i.e., $N(v)^2 = N(v)^\sharp$, so that we can use the commutator relations.
\end{assumption}

In this subsection, we try to get a better grasp on the interaction on the groups
\[G^+ = R_{\beta} R_{\alpha + \beta} R_{2 \beta + 3 \alpha} R_{2 \alpha + \beta} R_{3 \alpha + \beta}\] and \[G^- = R_{-\beta} R_{-\alpha - \beta} R_{-2 \beta - 3 \alpha} R_{-2 \alpha - \beta} R_{-3 \alpha - \beta}.\]
What we prove here will be used to prove the main theorem and basically amounts to proving that we have an operator Kantor pair; such pairs will be introduced and studied in the final section.

\begin{lemma}
	\label{lem: exp prop}
	For each Jordan pair $(J,J',Q)$ and each divided power representation $\rho^\pm$ in an algebra $A$, we have
	\[ \rho^+_{[t]}(x) \rho^{-}_{[s]}(y) = \rho^-_{[s]}(y + st Q_y x + s^2t^2 Q_yQ_x y + \dots ) \beta_{[st]} \rho^+_{[t]}(x + st Q_x y + s^2t^2 Q_xQ_y x + \dots)\]
	where the $\dots$ denote sums of terms of the form $Q_xQ_y\dots Q_yx$ or $Q_xQ_y\dots Q_xy$ and $\beta_{[st]}$ stands for $1 + \sum_{i = 1}^\infty (st)^i u_i$ for certain $u_i$ in $A$.
	\begin{proof}
		This is precisely the exponential property proved in \cite{faulkner2000jordan}.
	\end{proof}
\end{lemma}

Now, we reintroduce the elements $a_{i,j}$ and $\tilde{f}_{i,j}$ that popped up in the proof of Lemma \ref{lem: com formulas} in a more general setting.

\begin{definition}
	\label{def: oij}
	Suppose that $x(t) = 1 + \sum_{i = 1}^\infty t^i x_i$ and $y(t) = 1 + \sum_{j = 1} t^j y_j$ are formal power series in $A[[t]]$ for some associative unital algebra $A$. We remark that $x(t)$ and $y(t)$ are invertible, with inverses $x^{-1}(t)$ and $y^{-1}(t)$.
	For all $(a,b) \in \mathbb{N}^2$, we define elements $\nu_{i,j}(x,y) \in A$ and for all coprime $i,j \in \mathbb{N}$ we define new formal power series $o_{i,j}(x,y,t) = \sum t^k \nu_{ki,kj}(x,y)$, as the unique elements such that
	\[ y^{-1}(t) x(s) y(t) x^{-1}(s) = \prod_{i/j \in \mathbb{Q}_{> 0}} o_{i,j}(x,y,s^it^j),\]
	where the order of the product is the order of $i/j \in \mathbb{Q}_{> 0}$ (we assume $i$ and $j$ to be coprime).
	In Lemma \ref{lem: nu well def}, we prove this to be a good definition.
\end{definition}

\begin{remark}
	If we have binomial divided power representations $\sigma_{[s]}$ and $\rho_{[t]}$ of some modules $J$ and $K$ into an algebra $A$, then we also see $o_{i,j}(\cdot,\cdot,t)$ as maps $J \times K \longrightarrow A[[t]]$.
\end{remark}

\begin{lemma}
	\label{lem: nu well def}
	All $\nu_{i,j}$ are well defined.
	\begin{proof}
		This follows from induction on $i + j$ using the evaluation morphisms from $A[[t,s]]$ to $ A[t,s]/(t^{i+1},s^{j+1})$ to determine $\nu_{i,j}$.
	\end{proof}
\end{lemma}

\begin{lemma}
	Suppose that $\rho^\pm$ is a divided power representation of a Jordan pair $(J,J')$.
	Consider the $o_{i,j}(x,y,s^it^j)$ defined from $\rho^-_{[t]}(-y) \rho^+_{[s]}(x) \rho^-_{[t]}(y) \rho^+_{[s]}(-x)$ for $x \in J$ and $y \in J'$.
	For $i, k, l \in \mathbb{N}$, we have
	\begin{itemize}
		\item $o_{2i + 1, 2i}(x, y,t) = \rho^+_{[t]}(Q_x Q_y \dots Q_y x),$
		\item $o_{2i + 2, 2i + 1}(x, y,t) = \rho^+_{[t]}(Q_x Q_y \dots Q_x y),$
		\item $o_{2i + 1, 2i+2}(x, y,t) = \rho^-_{[t]}(Q_y Q_x\dots Q_y x),$
		\item $o_{2i, 2i + 1}(x, y,t) = \rho^-_{[t]}(Q_y Q_x \dots Q_x y),$
		\item $o_{k,l}(x,y,t) = 1$, if $|k - l| \neq 1$ and $(k,l) \neq (1,1)$.
	\end{itemize}
	\begin{proof}
		This is precisely Lemma \ref{lem: exp prop}.
	\end{proof}
\end{lemma}

\begin{remark}
	For a divided power representation, we note that $\nu_{i,j}$ is homogeneous of degree $i$ in the first component and of degree $j$ in the second component. That these are natural transformations, follows since these are formed by multiplying and adding natural transformations to $A \otimes \cdot$.
\end{remark}

\begin{definition}
	We use $I_{[t]}(\rho^+, \rho^-)$ to denote the group generated by the elements $o_{1,1}(x,y,t)$. 
	These correspond exactly to the $\beta_{[t]}$ of Lemma \ref{lem: exp prop}.
	It also shows that
	\[ \beta_{[st]}(x,y) = \rho^-_{[t]}(u) \rho^+_{[s]}(x) \rho^-_{[t]}(y) \rho^+_{[s]}(v)\]
	for certain $u \in t J'[[st]]$ and $v \in s J[[st]]$ (note that $(J,J')$ can mean $(R,R)$ in this sentence, since it represents a Jordan pair).
\end{definition}

Now, we use these notions on the algebra $A = \End(\mathcal{L}(\mathcal{A}))$.
\begin{lemma}
	\label{lem: I stabilizes}
	Let $\gamma$ short be a root.
	The groups $I_{[t]}(\exp_{\gamma}, \exp_{- \gamma})$ for the divided power representations of exponentials $(J'_{\gamma},J_{- \gamma}) \longrightarrow \Aut(\mathcal{L}(\mathcal{A}))$ stabilize, under the conjugation action, all $\exp_{\delta}(J)$, $\exp_{\delta}(J')$ and $\exp_\delta(R)$. The same holds for the group $I_{[t]}(\exp_\gamma, \exp_{ - \gamma})$ obtained from the exponentials $(R_{-\gamma},R_{\gamma}) \longrightarrow \Aut(\mathcal{L}(\mathcal{A}))$, whenever $\gamma$ is a long root.
	Moreover, this holds over arbitrary scalar extensions.
	\begin{proof}
		The commutator relations proved in Lemma \ref{lem: com formulas} and the definition of $I_{[t]}$ guarantees this if we are not looking at the conjugation action on $J'_{\gamma}$ or $J_{- \gamma}$.
		
		Take a long root $\delta$ orthogonal to $\gamma$. Consider groups $G^+_{\delta}$ and $G^-_{\delta}$ generated by root groups positively graded with respect to $\delta$ and the ones negatively graded by $\delta$. For $\delta = 2 \beta + 3 \alpha$, we note that
		\[G^+_{\delta} = R_{\beta} R_{\alpha + \beta} R_{2 \beta + 3 \alpha} R_{2 \alpha + \beta} R_{3 \alpha + \beta}\]
		as a set. For each $\delta$ one can find such a decomposition.
		This shows that $G^\pm_\delta$ are stabilized by $I_{[t]}$, as can be proved using commutator relations proved in Lemma \ref{lem: com formulas}, noting that any element of $I_{[t]}$ is a product of elements in $\exp_{\gamma,[t]}(J'[[t]])$ and $\exp_{-\gamma,[t]}(J[[t]])$.
		This action is compatible with scalar extensions, since the commutator relations are equalities of homogeneous maps.
			We also note that $I_{[t]}$ is $0$-graded by construction. Hence, it preserves all root groups with root different from $\pm \gamma$. 
		
		Suppose that $\gamma = \beta + 2 \alpha$ and $\delta = - \beta$.
		We can act on Equation (3) of Lemma \ref{lem: com formulas}. Take $i(t) \in I_{[t]}$ and suppose that $i(t) \cdot j_\alpha = j(t)_\alpha$ and $i(t) \cdot r_{-\beta - 3 \alpha} = r(t)_{-\beta - 3 \alpha}$. 
		The commutator relation on which we act, forces
		\[ i(t) \cdot (-rj)_{- \gamma} = (-r(t) j(t))_{- \gamma}.\]
		For all other short roots one argues analogously.
		This argument works for arbitrary scalar extensions, since the commutator relations are equalities of natural transformations.
		
		The same argument works for long roots, if one uses the grading with respect to a short root orthogonal to the long root and Equation (5) of Lemma \ref{lem: com formulas}.
	\end{proof}
\end{lemma}

\begin{construction}
	\label{con: H[t]}
	Set $H_{[t]}$ to be the group generated by the $I_{[t]}$ for all pairs $(J',J)$ and $(R,R)$.
	Note that $H_{[t]}$ stabilizes all root groups over all scalar extensions.
\end{construction}

\begin{lemma} \label{lem: oij roots} Use $U_\gamma$ to denote the root group corresponding to $\gamma$. Set $W_\gamma = U_\gamma$ whenever $\gamma$ is a root, $W_0 = H_{[t]}$, and $W_x = 1$ otherwise.
	Consider $x,y$ which are either roots or $0$ and $(i,j) \in \mathbb{N}^2_{> 0}$ with $\gcd(i,j) = 1$. The map $o_{i,j}(\cdot ,\cdot ,t)$ restricted to $U_x \times U_y$ maps to $W_{ix + jy}$
	\begin{proof}
		This follows from the commutator relations of nonopposite $J$, $J'$, and $R$, Lemma \ref{lem: exp prop} for opposite $J$ and $J'$ or oppositely lying $R$. \qedhere
		
	\end{proof}
\end{lemma}

\begin{definition}
	\label{def: vec group}
	Consider a pair of $R$-modules $A, B$, a bilinear map \[\psi : A \times A \longrightarrow B\] and the group $A \times B$ with operation $(x,y)(u,v) = (x + u, y + v + \psi(x,u))$.
	A subgroup $G \le A \times B$ is called an \textit{almost vector group} if $\lambda \cdot_1 (g_1,g_2) = (\lambda g_1, \lambda^2 g_2)$ is an element of $G$ if $(g_1,g_2) \in G$ and $\lambda \in R$, and if $\lambda \cdot_2 (0,g_2) = (0, \lambda g_2) \in G$ whenever $(0,g_2) \in G$ and $\lambda \in R$.

	For an almost vector group $G$ and $K \in R\textbf{-alg}$, we define $\hat{G}(K)$ as the smallest almost vector group contained in $(A \times B) \otimes K$ containing $(1 \otimes \eta) (G \otimes 1)$ with $\eta$ the structure morphism.	
	We use $\Lie(G)$ to denote $\{(\epsilon a, \epsilon b) \in \hat{G}(R[\epsilon]/(\epsilon^2))\}.$
	
	We call $G$ a \textit{nice vector group} if $\Lie(\hat{G}(K)) \cong \Lie(G) \otimes K$ for all $K$.
	
	For what we try to achieve here, we only have to look at nice vector groups. For the relation to the terminology of \cite{OKP} and some basic properties, we refer to Appendix \ref{ap: vec groups}.
\end{definition}

\begin{remark}	
The subsets \[G^+ = R_{\beta} R_{\alpha + \beta} R_{2 \beta + 3 \alpha} R_{2 \alpha + \beta} R_{3 \alpha + \beta}\] and \[G^- = R_{-\beta} R_{-\alpha - \beta} R_{-2 \beta - 3 \alpha} R_{-2 \alpha - \beta} R_{-3 \alpha - \beta}\] of the group $\Aut(\mathcal{L}(\mathcal{A}))$ are easily shown to be subgroups using the commutator relations. We now show that these define nice vector groups.

Write $G(K)$ for the analogous subgroup of $\left(\End(\mathcal{L}(\mathcal{A})) \otimes K\right)^\times$, we note that $G(K)$ can as well be identified with a subgroup of $\Aut(\mathcal{L}(\mathcal{A}))(K) = \Aut(\mathcal{L}(\mathcal{A} \otimes K))$. However since all the equalities we proved are in $\End(\mathcal{L}(\mathcal{A})) \otimes K$, we prefer to stay there, as any equality in $\Aut(\mathcal{L}(\mathcal{A}))(K)$ results from that.

We proved the commutator relations for well behaved cubic norm pairs and noted that some relations also hold for not necessarily well behaved cubic norm pairs. The commutator relations one needs to verify that $G^+$ and $G^-$ are groups, do not need well-behavedness, as only (2) and (3) of Lemma \ref{lem: com formulas} require this assumption. Hence, $G^+$ and $G^-$ make sense for arbitrary cubic norm pairs.
\end{remark}

\begin{lemma}
	\label{lem: descr G+}
	Let $(J,J')$ be a cubic norm pair (not necessarily well behaved).
	There is an isomorphism
	\[G^+ \cong U = \{ ((a\bar{t},v,w^*,b t),(c \bar{t}, bv + w^\sharp, aw + (v^\sharp)^*, dt ) \in \mathcal{A}^2) | c + d = ab + T(v,w) \} \]
	with $(x,y) \cdot (u,v) = (x + u, y + v + x\bar{u})$ as the group operation for the latter group.
	Each $g \in G^+$ maps to the unique $(x,y) \in \mathcal{A}^2$ such that
	\[ \exp(g) = 1 + (\ad x_1) + g_2 + g_3 + g_4\]
	with $g_i$ acting as $+i$ on the grading, $x_1 \in \mathcal{A}_1$, and $g_2 \cdot b_{-1} = ((x\bar{b})x - y b)_1$ for all $b_{-1} \in \mathcal{A}_{-1}.$
	
	Moreover, this isomorphism is compatible with scalar extensions, i.e, $G^+(K) \cong U(K)$ for all $K \in R\textbf{-alg}$ and $G^+$ is nice.
	\begin{proof}
		First, we note that $U$ is obviously a nice vector group.		
		The idea is to construct a morphism $G^+ \longrightarrow \mathcal{A}^2$ that maps all generators of $G^+$ into $U$. We will write $x_g$ and $y_g$ for the $x$ and $y$ determined by $g$. 
		
		So, we first construct part of the map $G^+ \longrightarrow \mathcal{A}^2$ using the identification $G^+ \subset \Aut(L)$.
		To do that, recall the grading element $\zeta = \zeta_{2 \beta + 3 \alpha}$ associated to the usual $\mathbb{Z}$-grading.
		So, by computing $g \cdot \zeta = \zeta - x_g + z_g \in \mathfrak{instr}(A)_{0} \oplus \mathcal{A}_1 \oplus \mathcal{S}_2$ for $g \in G^+$ we can determine $x$.
		The map $g \mapsto x_g$ has as kernel $R_6 \cong \mathcal{S}_2$.
		
		Now, we claim that for each $g$, there exists a unique $y$ and certain $b$ and $c$ such that \begin{equation}
			\label{eq: g_2 related to Qxy}
			g \cdot a = a + b + ((x_g \bar{a}) x_g - y_g a) + c \in \mathcal{A}_{-1}\oplus \mathfrak{instr}(A)_{0} \oplus \mathcal{A}_{1} \oplus \mathcal{S}_{2} \quad \text{for } a \in \mathcal{A}_{-1}.
		\end{equation}
		These $b$ and $c$ always exist. Evaluating in $a = 1$ shows that $y_g$ is unique if it exists.
		
		First, we show that $y_g$ exist for each of the generating root groups of $G^+$.
		Namely, direct computations show that
		\begin{itemize}
			\item $R_{\beta} \ni \lambda \mapsto (\lambda \bar{t}, 0) \in U \subset \mathcal{A}^2$,
			\item $R_{\beta + 3 \alpha } \ni \lambda \mapsto( \lambda t, 0) \in U$,
			\item $R_{2 \beta + 3 \alpha} \ni \lambda \mapsto (0, - \lambda(t - \bar{t}))$,
			\item $J_{\beta + \alpha} \ni j \mapsto (j,j^\sharp)$,
			\item $J_{\beta + 2 \alpha} \ni j \mapsto (j,j^\sharp)$.
		\end{itemize}
		
		Second, we note that if there exists such an $y_g$ and $y_h$ for $g$ and $h$, then there exists a $y_{gh} = y_g + y_h + x_g \overline{x_h}$ using the definition of $V$ and $\mathcal{L}(A)$ combined with the adjoint action of $R_{2\beta + 3\alpha}$ on $\mathcal{A}_{-1}$. So, we obtain a homomorphism $G^+ \longrightarrow U$.
		Now, remark that $R_{2\beta + 3\alpha} \longrightarrow U: r(t - \bar{t}) \longrightarrow (0,- r(t - \bar{t}))$ is injective and thus that $G^+ \longrightarrow U$ is an isomorphism.
		
		Note that any $g \in G^+(K)$ is of the form
		\[ \exp_\beta(1 \otimes k) \exp_{\beta + \alpha}(\sum x_i \otimes k_i) \exp_{2 \beta + 3 \alpha}(1 \otimes m) \exp_{\beta + 2 \alpha}(\sum y_i \otimes l_i) \exp_{\beta + 3 \alpha}(1 \otimes l).\]
		Using evaluation on $\zeta$, we obtain \[x = \bar{t} \otimes k + \sum x_i \otimes k_i+ \sum y_i \otimes l_i + t \otimes l.\]
		If $x = 0$, we obtain $g = \exp_{2 \beta + 3 \alpha}(1 \otimes m)$.
		Evaluating on $1$ yields $1 \otimes m$ in this case.
		Hence, this isomorphism is compatible with scalar extensions.
		
		This also shows that $G^+$ is nice.
	\end{proof}
\end{lemma}

\begin{proposition}
	\label{lem: oij}
	Let $(J,J')$ be a well behaved cubic norm pair.
	All $o_{i,j}(\cdot,\cdot,t) : G^+(K) \times G^-(K) \longrightarrow (\End(\mathcal{L}(\mathcal{A})) \otimes K)^\times$ with $i > j$ map to $G^+(K)$ and those $o_{i,j}(\cdot,\cdot,t)$ with $i < j$ map to $G^-(K)$. Moreover, $o_{1,1}$ maps $G^+(K) \times G^-(K) \longrightarrow P$ with $P$ the group generated by $\exp_{-\alpha}( (J' \otimes K)[[st]])$, $\exp_{\alpha}((J \otimes K)[[st]])$, and $H_{[[st]]}$.
	\begin{proof}
		In this proof, we first pretend that $K = R$. Afterwards, we explain why it works for arbitrary scalar extensions.

		Write $sG^+[[st]]$ for the subgroup $\{ (s x(st), s^2 y(st)) \in U[[s,t]] \subset \mathcal{A}[[s,t]]^2\}$, using $a(st)$ to denote a formal power series in the variable $st$, and use $(t G^+[[st]])$ for the similar subgroup of $G^-[[s,t]]$.		
		First, note that $P$ stabilizes $sG^+[[st]]$ and $tG^-[[st]]$, as can be proved using the commutator relations and the fact that $H_{[st]}$ stabilizes these groups.
		We first show that $(s G^+[[st]])(t G^-[[st]]) \subset (t G^-[[st]]) P (s G^+[[st]])$. Afterwards, we will recover the $o_{i,j}$ from that.
		
		Using this, Lemma \ref{lem: exp prop}, and the commutator relations, one can show that $(s V[[st]])_\gamma G^- \subset (t \cdot G^-[[st]])P (s \cdot G^+)$ for $\gamma$ positively graded and $V$ the space corresponding to the root $\gamma$. 
		 In what follows, we do not write the $s$, $t$, and $st$ for brevity.
		For $R_{\beta}$ this is easy, since
		$G^- = T R_{-\beta} R_{-\beta - \alpha }$ 
		with $T$ stabilized under $R_{\beta}$, so that
		\[ R_{\beta} G^- \subset T R_{- \beta} H_{[t]} R_{\beta} R_{- \beta - \alpha}\subset  T R_{-\beta} H_{[t]} R_{- \beta - \alpha} R_{- \alpha} R_{\beta} \subset T R_{-\beta} R_{-\beta - \alpha} P R_{\beta}.\]
		The same holds for $R_{\beta + 3\alpha}$.
		Now, using $R_{2 \beta + 3\alpha} = [R_{\beta}, R_{\beta + 3\alpha} ]$, proves it for all $R$.
		
		Finally, one can prove it for $J'_{2\alpha + \beta}$ using 
		\[ G^- =  R_{- \beta} R_{- 2\beta - 3 \alpha } J'_{- \beta - \alpha} J_{- \alpha - 2 \beta} R_{- \beta - 3 \alpha}.\]
		Namely, only the part $J_{- \alpha - 2 \beta} R_{- \beta - 3 \alpha}$ is not positively graded with respect to $- \beta$, which shows that
		\[ J'_{2\alpha + \beta} R_{- \beta} R_{- 2\beta - 3 \alpha } J'_{- \beta - \alpha} \subset R_{- \beta} R_{- 2\beta - 3 \alpha } J'_{- \beta - \alpha} J_\alpha R_{\beta + 3 \alpha }J'_{\alpha + 2 \beta} \subset  G^- P J'_{\alpha + 2 \beta} R_{\beta + 3 \alpha}\]
		follows from the commutator relations.
		This shows that
		\begin{align*}
			J'_{2\alpha + \beta} R_{- \beta} R_{- 2\beta - 3 \alpha } J'_{- \beta - \alpha} J_{-2\alpha - \beta} & \subset G^- P J_{-2\alpha - \beta} H_{[st]} J_\alpha J'_{2\alpha + \beta} R_{\beta + 3 \alpha} \\ & \subset G^- P J'_{2\alpha + \beta} R_{\beta + 3 \alpha}
		\end{align*}
		using the commutator relations and Lemma \ref{lem: exp prop}.
		Finally, one sees that
		\[ J'_{2\alpha + \beta} G^- \subset G^- P J'_{2\alpha + \beta } R_{- \beta - 3 \alpha} H_{[st]} R_{\beta + 3 \alpha} \subset G^- P  G^+, \]
		once again by Lemma \ref{lem: exp prop} and the commutator relations.
		
		To recover the $o_{i,j}$ note that
		\[ (sx)(ty) = ( t y_1(st), t^2 y_2(st)) p (s x_1(st), s^2 x_2(st))\]
		with $x_i(st), y_i(st) \in \mathcal{A}[[st]]$ and $p \in P$.
		One can decompose $(s x_1(st), s^2 x_2(st))$ as
		\[  \left( \prod_{i = \infty}^1 o_{i + 1,i}(x,y,s^{i + 1} t^i) \right) \left( \prod_{i = \infty}^0 o_{2i + 3,2i + 1}(x,y,s^{2i + 3} t^{2i + 1}) \right) (sx)\]
		in $sG^+[[st]]$ using recursion on the degree of $s$ (that $G^+$ is a vector group guarantees that the found $o_{i,j}$ also lie in $G^+$). Observe that $o_{i,j}(x,y,t) \in G^+[[t]]$. Using these $o_{i,j}$ as the only nonzero $o_{i,j}$ when $i > j$, it is obvious that $p = o_{1,1}(x,y,st)$.

		Now, note that the argument builds on the equalities of homogeneous maps of	Lemma \ref{lem: exp prop} and the homogeneous maps that express the commutator relations, to rewrite an element of $G^+(K) G^-(K)$. Hence, this shows that the element $(s x_1(st), s^2 x_2(st))$ is natural in the scalar extension $K$. The $o_{i,j}(x,y,t)$ one finds are also natural in $K$.
	\end{proof}
\end{proposition}

\subsection{Main theorem for cubic norm pairs}

Here, we work toward proving the cubic norm pair version of main theorem advertised in the introduction.

We use $G^+(K)[[st]]$ to denote the inverse limit of the groups $G^+(K[st]/(s^nt^n))$.
We note that any $h \in H[t]$ is of the form $1 + t h_1 + t^2 h_2 + \dots$.

In Lemma \ref{lem: oij} we used the group $P$ generated by $J'_{-\alpha}[[st]], J_{\alpha}[[st]]$ and $H_{[[st]]}$, we will now denote this group as $P_{[[st]]}$.

Recall that $\lambda \cdot (a,b) = (\lambda a, \lambda^2 b)$ for $(a,b)$ in a vector group $G$.

\begin{lemma}
	\label{lem: omega[[st]] subgroup}
	Suppose that $(J,J')$ is a well behaved cubic norm pair.
	The subset 
	\[ (t \cdot G^-(K)[[st]]) P_{[[st]]} ( s\cdot G^+(K)[[st]])\]
	of $\left((\End((\mathcal{L}(\mathcal{A}))) \otimes K)[[s,t]]\right)^\times $ is a subgroup.
	\begin{proof}
		The group $P_{[[st]]}$ stabilizes $G^+(K)[[st]]$ and $G^-(K)[[st]]$ under the conjugation action, since $J'_{- \alpha}[[st]]$ and $J_{\alpha}[[st]]$ stabilize them, as follows from the commutator relations, and since $H_{[[st]]}$ stabilizes them as well.
		The desired result is a direct application of Lemma \ref{lem: oij} and the fact that $P_{[[st]]}$ stabilizes $G^+(K)[[st]]$ and $G^-(K)[[st]]$, using the definition of the $o_{i,j}$:
		\[ (s \cdot x)(t \cdot y) = (t \cdot y) \left( \prod_{i/j\in \mathbb{Q}_{>0}} o_{i,j}(x,y,s^it^j) \right) (s \cdot x). \qedhere\]
	\end{proof} 
\end{lemma}

Let $G(K)$ denote the subgroup of $\Aut(\mathcal{L}(\mathcal{A}))(K)$ or its lift in $(\End((\mathcal{L}(\mathcal{A}))) \otimes K)^\times$ generated by $G^+(K)$ and $G^-(K)$.
We remark that $G(K)$ is closed under the grading automorphism \[\Phi(t) : \sum_{i = -4}^4 g_i \mapsto \sum_{i = -4}^4  t^i g_i\] for $t \in K$ invertible, since its generating subgroups are.
For $g = 1 + g_1 + g_2 + g_3 + g_4 \in G(K)$ with $g_i$ acting\footnote{Being $i$-graded in the $\mathbb{Z}$-grading induced by the $\mathbb{Z}$-grading of $\mathcal{L}(\mathcal{A}) \otimes K$ is what we mean by this for arbitrary $K$.} as $+i$ on the grading of the Lie algebra, we define
\[g_\eta = \Phi(1 + \eta)(g) =  1 + \sum (1 + \eta)^i g_i\]
in $G(K[\eta]/(\eta^m))$.

\begin{lemma}
	\label{lem: g(eta2n) g(-1) in G+}
	Suppose that $(J,J')$ is a well behaved cubic norm pair.
	If \[ g = 1 + g_1 + g_2 + g_3 + g_4 \in (\End((\mathcal{L}(\mathcal{A}))) \otimes K)^\times,\]
	with $g_i$ acting as $+i$ on the grading, is an element of $G(K)$, then there exists an $n$ such that
	\[ g_{\eta^{2^{{2n}}}} g^{-1} \in G^+(K[\eta]/(\eta^m))\]
	for all $m$.
	\begin{proof}
		We can find $g(s,t) = \prod_{i = 1}^n (s \cdot x_i)(t \cdot y_i)$ with $x_i \in G^+(K)$ and $y_i \in G^-(K)$ such that $g(1,1) = g$, since $g \in G(K)$.
		The $n$ which occurs in the definition of $g(s,t)$ is the $n$ for which we will prove the lemma.
		
		First, we introduce some subgroups that will play a role in our inductive argument.
		Define \[H^-_k = \{ (\eta^{2^{k}} tx, \eta^{2^{k}} t^2 y) \in t\eta^{2^{k-1}} \cdot G^-(K[st\eta^{2^{k}}, \eta^{2^{k-1}}]/(\eta^m))\},\]
		i.e., the kernel under $\eta^{2^k} = 0$ of $t \cdot G^-(K[st\eta^{2^{k}}, \eta^{2^{k-1}}]/(\eta^m))$. 
		Note that 
		\[t \cdot G^-(K[st\eta^{2^{k}}, \eta^{2^{k-1}}]/(\eta^m)) \subset G^-(K[s,t, \eta^{2^{k - 1}}]).\]
		Secondly, we introduce the analogous subgroup
		\[H^+_k = \{ (s\eta^{2^{k}} x, s^2\eta^{2^{k}} y) \in s\eta^{2^{k - 1}} \cdot G^+(K[st\eta^{2^{k}}, \eta^{2^{k - 1}}]/(\eta^m))\}.\]
		Thirdly, we introduce the group $K_k$ that corresponds to $P_{[[st\eta^{2^{k}}, \eta^{2^{k - 1}}]]}$. To be precise, it is the image of this group under $\eta^m = 0$ and we also assume that it only contains elements that map to $1$ under $\eta = 0$.
		Set 
		\[ \Omega_k(\eta) = \Omega_k = H^-_k K_k H^+_k.\]
		By Lemma \ref{lem: omega[[st]] subgroup} with $s \eta^{2^{k - 1}}$ and $t \eta^{2^{k - 1}}$ instead of $s$ and $t$, $\Omega_k$ is a subgroup; the additional assumption that $K_k \equiv 1 \mod (\eta)$ corresponds to working with the kernel under $\eta^{2^{{k - 1}}} \mapsto 0$ of the image of the group of Lemma \ref{lem: omega[[st]] subgroup}.
		
		We shall prove using induction on $n$ that for each
		\[ g(s,t) = \prod_{i = 1}^n (s \cdot x_i)(t \cdot y_i)\]
		one has
		\[ \left( \Phi(1 + \eta^{2^{2n}}) \cdot g(s,t) \right) g(s,t)^{-1} \in \Omega_1(\eta) . \]

		We prove first that $(s \cdot x) \Omega_k \subset \Omega_{k - 1}(s \cdot x)$ for $x \in G^+$.
		First, $(s \cdot x) H^-_k$ is contained in
		\[ \left(s \cdot G^+(K[st\eta^{2^{k - 1}}, \eta^{2^{k - 1}}]/(\eta^m))\right) \left(t\eta^{2^{k - 1}} \cdot G^-(K[st\eta^{2^{k - 1}}, \eta^{2^{k - 1}}]/(\eta^m))\right).\]
		Lemma \ref{lem: omega[[st]] subgroup}, shows that it is contained in
		\[ H^-_{k - 1} P^-_{k - 1} \left(s \cdot G^+(K[st\eta^{2^{k - 1}}, \eta^{2^{{k-2}}}]/(\eta^m))\right).\]
		Moreover, $(s \cdot x) H^-_k$ maps to $(s \cdot x)$ if we set $\eta^{2^{k - 1}} = 0$.
		Hence, it lies in
		\[ \Omega_{k - 1} (s \cdot x).\]
		A similar argument shows that $(s \cdot x) K_k \subset K_k H^+_k (s \cdot x) \subset \Omega_k (s \cdot x) \subset \Omega_{k - 1}(s \cdot x)$. We also note that $(s \cdot x) H^+_k = H^+_k (s \cdot x)$ since $H^+_k$ is normal.
		We conclude that $(s \cdot x)\Omega_k \subset \Omega_{k - 1}(s \cdot x)$.		
		Analogously, one proves that $(t \cdot y) \Omega_k \subset \Omega_{k - 1} (t \cdot y)$.
		
		Now, we are ready to perform the induction.
		The element $(s \cdot x_i)_{\eta^{2^{2n}}}$ lies in $H^+_{2n} (s \cdot x_i)$ as it maps to $(s \cdot x_i)$ under $\eta^{2^{2n}} = 0$.
		Hence, $(s \cdot x_i)_{\eta^{2^{2n}}}$ must lie in $\Omega_{2n} (s \cdot x_i)$ for all generators $x_i \in G^+$. We also have $(t \cdot y_i)_{\eta^{2^{2n}}} \in \Omega_{2n} (t \cdot y_i)$.
		
		For the base case, we have
		\[ g(s,t)_{\eta^{2^{2^{2}}}} g(s,t)^{-1} \in \Omega_2 (s \cdot x_1) \Omega_2 (t \cdot y_1) (t \cdot y_1)^{-1} (s \cdot x_1)^{-1} \in \Omega_2 \Omega_1 \subset \Omega_1^2 = \Omega_1.\]
		For the inductive step, we note that $\eta^{2^{2n + 2}} = \eta^{{4 \cdot 2^{2n}}} = (\eta^4)^{2^{2n}}$. Now, $\Omega_1(\eta^4) = \Omega_3(\eta)$, since $(\eta^4)^{2} = \eta^{2^3}$, which yields that 
		\[ g(s,t)_{\eta^{2^{2n+2}}} g(s,t)^{-1} \in \left( \prod_{i = 1}^n \Omega_{2n} (s \cdot x_i) \Omega_{2n} (t \cdot x_i) \right) \prod_{i = 1}^n (t \cdot y_{n + 1 - i})^{-1} (s \cdot x_{n + 1 - i})^{-1}\]
		reduces to
		\[ g(s,t)_{(\eta^4)^{2^{2n}}} g(s,t)^{-1} \in \Omega_{1}(\eta) (s \cdot x_1) \Omega_{2}(\eta) (t \cdot y_1) \Omega_1(\eta^4) (t \cdot y_1)^{-1} (s \cdot x_1)^{-1} \in \Omega_1(\eta).\] 
		This finalizes the induction.
		
		Since $g(1,1) = 1 + g_1 + g_2 + g_3 + g_4$, we see that $g_{\eta^{2^{2n}}}(1,1) g^{-1}(1,1) = y k x$ with $y \in G^-(K[\eta]/(\eta^m)),$ $k$ in the image of $K_1$, and $x \in G^+(K[\eta]/(\eta^m))$. This forces 
		\[ g_{\eta^{2^{2n}}}(1,1) g^{-1}(1,1)  x^{-1} = 1,\]
		since $yk$ cannot increase gradings.
		Thus, we conclude $g_{\eta^{2^{2n}}}(1,1) g^{-1}(1,1) = x \in G^+(K[\eta]/(\eta^m))$.
	\end{proof}
\end{lemma}

\begin{lemma}
	\label{lem: helpp}
	Suppose that $(J,J')$ is a well behaved cubic norm pair. 
	If the element
	\[g = 1 + g_1 + g_2 + g_3 + g_4 \in (\End((\mathcal{L}(\mathcal{A}))) \otimes K)^\times\] with $g_i$ acting as $+i$ on the grading, is an element of the subgroup $G(K)$, then $g \in G^+(K)$.
	\begin{proof}
		From Lemma \ref{lem: g(eta2n) g(-1) in G+}, we obtain that there exists an $n$ such that
		\[ g_{\eta^{2^{2n}}} g^{-1} \in G^+(K[\eta]/(\eta^m))\]
		for all $m$.
		We also see that
		\[ g_{\eta^{2^{2n}}} g^{-1} = 1 + \eta^{2^{2n}} g_1 + (\eta^{2\cdot 2^{2n}} g_2 + \eta^{2^{2n}}(2 g_2 - g_1^2)) + \dots\]
		grouping terms acting homogeneously on the grading. The part $(\eta^{2^{2n}} g_1, \eta^{2\cdot 2^{2n}} g_2 + \eta^{2^{2n}}(2 g_2 - g_1^2))$ we wrote down, determines the element uniquely as an element $G^+(K[\eta]/(\eta^m))$, using Lemma \ref{lem: descr G+}.
		Observe that $G(K[\eta]/(\eta^m))$ is generated by elements $\eta^i \cdot (a_1,a_2) = (\eta^i a, \eta^{2i} b)$ and $\eta^i \cdot_2 (0,b) = (0, \eta^{i} b)$ for $(a_1,a_2) \in G(K)$, $(0,b) \in G(K)$, and $i \in \mathbb{N}$, by \cite[Lemma 1.3 and Lemma 1.5]{OKP}. This shows that $(g_1,g_2) \in G(K)$.
		The higher graded components show that $g \in G(K)$.
	\end{proof}
\end{lemma}

In the theorem below, we mean by $\Lie(G)$ the kernel of $\epsilon \mapsto 0$ in $G(R[\epsilon]/(\epsilon^2))$.
One way to compute the Lie bracket is $[1 + \epsilon_1 g, 1 + \epsilon_2 h] = 1 + \epsilon_1 \epsilon_2 [g,h]$ over $R[\epsilon_1, \epsilon_2]/(\epsilon_1^2, \epsilon_2^2)$. 

\begin{theorem}
	\label{thm: cnp main}
	Let $(J,J')$ be a well behaved cubic norm pair, and $G(K)$ the subgroup of $\Aut(\mathcal{L}(\mathcal{A}))$ generated by $G^+(K)$ and $G^-(K)$.
	Any element $1 + \sum_{i = 1}^\infty g_i$ with $g_i$ acting as $+i$ on the grading of the Lie algebra that contained in $G(K)$, is contained in $G^+(K)$.
	Moreover, one has $[\Lie(G),\Lie(G)] \cong \mathcal{L}(\mathcal{A})$.
	\begin{proof}
		The main statement follows from Lemma \ref{lem: helpp} (noting that $G^\pm(K)$ are automorphism groups of $\mathcal{L}(\mathcal{A})) \otimes K$).
		
		Now, we are ready to prove the second part of the moreover-claim. 
		Start by observing that each root group $U_\gamma$ is contained in $G$.
		This implies that $\Lie(G^\pm) \subset [\Lie(G),\Lie(G)] \subset \Lie(G)$, using that there exists a long root $\alpha$ and a root $\beta$ such that $U_\gamma = [R_\alpha, U_\beta]$.
		Observe that $\mathcal{L}(\mathcal{A})$ is isomorphic to the Lie subalgebra generated by $\Lie(G^\pm)$, as both are isomorphic to the Lie algebra generated by $\ad l$ for all $l \in L$.
		This proves $\mathcal{L}(\mathcal{A}) \subseteq [\Lie(G),\Lie(G)]$.
		
		On the other hand, $\mathcal{L}(\mathcal{A})$ is an ideal of $\Lie(G)$ (since $G(K)$ acts on $\mathcal{L}(\mathcal{A}) \otimes K$ by construction). Any element $\delta \in \Lie(G)/\mathcal{L}(\mathcal{A})$ corresponds to a zero graded element of $\Lie(G)$, since the parts that carry nonzero grading are contained in $\Lie(G^\pm)$. This observation yields $G^\pm \cdot \delta = \delta$ for all $\delta \in \Lie(G)/\mathcal{L}(\mathcal{A})$.
		Hence, $G$ acts trivially on $\Lie(G)/\mathcal{L}(\mathcal{A})$. We conclude that $[\Lie(G), \Lie(G)] \subset \mathcal{L}(\mathcal{A})$.
	\end{proof}
\end{theorem}

\begin{remark}
	In general, one can define $\Lie(G)(K)$ as a certain subgroup of $G(K[\epsilon]/(\epsilon^2))$.
	The statement for Lie algebras more or less holds.
	Namely, one obtains a surjection $\mathcal{L}(\mathcal{A}) \otimes K \longrightarrow [\Lie(G)(K),\Lie(G)(K)]$. The only barrier to be an isomorphism is that $[\Lie(G)(K),\Lie(G)(K)] \subset \End_K(\mathcal{L}(\mathcal{A}) \otimes K)$ and that $(\ad \otimes \text{Id})$ on $\mathcal{L}(\mathcal{A}) \otimes K$ is injective if and only if the Lie algebra has no center.
	Working with $\mathcal{L}(\mathcal{A} \otimes K)$ instead, guarantees that the map is a bijection.
\end{remark}

\section{Hermitian cubic norm structures}

In this section, we first introduce quadratic étale extensions and hermitian cubic norm structures defined over such extensions. Splitting the quadratic étale algebra lets us relate the hermitian cubic norm structure to a cubic norm pair, and allows us to generalize the constructions of the structurable algebra, Lie algebra, and automorphism groups $G^+$ and $G^-$.

We also prove the version of main theorem for hermitian cubic norm structures.

\subsection{Quadratic étale algebras}

We first introduce quadratic étale algebras. These will play a role analogous to $R[t]/(t^2 - t) \cong R^2$, i.e., the diagonal subalgebra of the structurable algebra for a cubic norm pair.

\begin{definition}
	Consider the localization $R_\mathfrak{p}$ for prime ideals $\mathfrak{p}$ of our base ring $R$.
	Let $R(\mathfrak{p})$ denote the quotient field of $R/\mathfrak{p}$, or equivalently the localization $R_\mathfrak{p}$ modulo its maximal ideal $\mathfrak{p}_{\mathfrak{p}}$.
	We call a projective $R$-algebra $K$ \textit{quadratic étale}, following \cite[Definition 19.19]{Skip2024}, if for all prime ideals $\mathfrak{p}$ over $R$, the algebra $K(\mathfrak{p}) \cong K \otimes R(\mathfrak{p})$ is either a quadratic field extension of $R(\mathfrak{p})$ or isomorphic to $R(\mathfrak{p})^2$. In particular, a quadratic étale extension is projective of rank $2$, and can therefore be given the structure of a composition algebra \cite[Definition 19.19]{Skip2024} with quadratic norm $n$. Hence, we have an involution $x \mapsto \bar{x} = n(x,1) 1 - x$ using the linearisation of the norm $n$.
	
	We also remark that any quadratic étale extension is faithfully flat, \cite[Proposition 25.9.(ii)]{Skip2024}.	
	The most basic example of a quadratic étale extension, cfr., \cite[Definition 19.19.(iv)]{Skip2024}, is $K = R[t]/(t^2 - t + \alpha)$ with $1 - 4 \alpha$ invertible.
	We will write $K/R$ to denote a quadratic étale extension, if we want to stress the base ring.
\end{definition}

\begin{lemma}
	\label{lem: quadr etale alg ids}
	Let $K$ be a quadratic étale extension, then $K \otimes K \cong K \times K$ under $k \otimes l \mapsto (kl, \bar{k}l)$. The corresponding involution $\bar{\cdot} \otimes \text{Id}$ is given by $(a,b) \mapsto (b,a)$.
	 Moreover, following subspaces of $K$ can be identified:
	\begin{itemize}
		\item $R \cong R \cdot 1_K \cong \{ x \in K : x = \bar{x}\}$,
		\item $\{ x \in K : x + \bar{x} = 0\} \cong \{ x - \bar{x} \in K\}$.
	\end{itemize}
	\begin{proof}
		The first part is \cite[Exercise 19.36]{Skip2024} and the involution follows from that.
		
		To identify the subspaces, we shall use that $K$ is faithfully flat, i.e., $K$ is flat and $M \longrightarrow M \otimes K$ is injective for all $R$-modules $M$ \cite[Proposition 25.4]{Skip2024}.
		For the first one, we note that $R \longrightarrow R \otimes K \cong K$ is injective, hence $R \cong R \cdot 1_K$.
		Now, consider the short exact sequence
		\[ 0 \longrightarrow \{(k,k) \in K^2\} \longrightarrow K^2 \longrightarrow \{ (a,-a) \in K^2\} \longrightarrow 0\]
		defined by the map $(a,b) \mapsto (a,b) - (b,a) = (a,b) - \overline{(a,b)}$.
		Both $(R \cdot 1_K) \otimes K$ and $\{x \in K : x = \bar{x}\} \otimes K$ coincide with $\{(k,k) \in K^2\}$. Now, the exact sequence $0 \longrightarrow (R \cdot 1_K) \longrightarrow \{x \in K : x = \bar{x}\} \longrightarrow U \longrightarrow 0$ and the faithful flatness of $K$ shows that $U \cong 0$, or in  other words $R \cdot 1_K \cong \{x \in K : x = \bar{x}\}$.
	
		We also note that $\{x \in K : x + \bar{x} = 0\} \otimes K \cong \{ (a,b) \in K \times K : (a,b) + (b,a) = 0 \} \cong \{ (a,-a) \in K \times K\} \cong \{ x - \bar{x} : x \in K \times K\} \cong \{x  -\bar{x} \in K\} \otimes K$. Now, using that each $u = x - \bar{x}$ satisfies $u + \bar{u} = 0$, shows, as before, that $\{ x \in K : x + \bar{x} = 0\} = \{ x - \bar{x} \in K\}$.
	\end{proof}
\end{lemma}

\begin{lemma}
	\label{lem: quadr etale alg mods}
	Take an $R$-module $M$ and a quadratic étale extension $K$. Consider the map $\alpha: m \otimes k \mapsto m \otimes \bar{k}$ defined on $M \otimes K$. We can identify $M \subset M \otimes K$ as the fixpoints under $\alpha$.
	\begin{proof}
		Recall that $K$ is faithfully flat, so that $M$ can be identified with a subset of $M \otimes K$. We also note that $M$ is a subset of the fixpoints of $\alpha$.
		
		Take a short exact sequence \[0 \longrightarrow M \otimes R \longrightarrow \{ m \in M \otimes K : \alpha(m) = m\} \longrightarrow U \longrightarrow 0.\]
		Apply $\cdot \otimes K$, and we obtain
		\[ M \otimes K \longrightarrow \{ (m,m) \in (M \otimes K)^2 \} \longrightarrow U \otimes K \longrightarrow 0.\]
		We see that $U \otimes K = 0$, hence $U = 0$ since $K$ is faithfully flat.
		We conclude, $M \cong \{ m \in M \otimes K : \alpha(m) = m\}.$
	\end{proof}
\end{lemma}

We shall write $\mathcal{S}$ for $\{ x \in K : x + \bar{x} = 0\} = \{x - \bar{x} \in K\}$, as we did before for the diagonal subalgebra $R^2$ of the structurable algebra.

\subsection{Hermitian cubic norm structures and structurable algebras}
\label{subsec: hcns}

\begin{definition}
	Let $\Phi$ be an associative unital $R$-algebra with involution $a \mapsto \bar{a}$ and let $A$ be a $\Phi$-module.
	We call an $R$-bilinear map $f : A \times A \longrightarrow \Phi$ \textit{hermitian} if $f(\phi a, b) = \phi f(a,b)$ and $f(a,b) = \overline{f(b,a)}$ for all $a,b \in A$ and $\phi \in \Phi$.
\end{definition}

\begin{definition}
	Consider a quadratic étale extension $K/R$. Let $J$ be a $K$-module endowed with the following maps:
	\begin{itemize}
		\item a $K$-cubic map $N : J \longrightarrow K$,
		\item an $R$-quadratic map $\sharp : J \longrightarrow J$ such that $(k j)^\sharp = \bar{k}^2 j$ for $k \in K$ and $j \in J$,
		\item the linearisation $(a,b) \mapsto a\times b$ of $\sharp$ which is assumed to be $\bar{\cdot}$-antilinear in both arguments,
		\item a hermitian map $T : J \times J \longrightarrow K$. 
	\end{itemize}
	This quadruple $(J,N,\sharp,T)$ is called a \textit{hermitian cubic norm structure} over $K/R$ if
	\begin{itemize}
		\item $T(a,b^\sharp) = N^{(1,2)}(a,b)$,
		\item $(a^\sharp)^\sharp = N(a)a$,
		\item $(a^\sharp \times b) \times a = \overline{N(a)}b + T(b,a) a^\sharp$,
		\item $N(T(j,k)j - j^\sharp \times k) = N(j)^2 \overline{N(k)}$
	\end{itemize}
	hold over all scalar extensions $K \otimes L$ of $K$ induced by extending $R$ to $L$.
	
	As for cubic norm pairs, we call a hermitian cubic norm structure \textit{well behaved} if $N(v)^2 = \overline{N(v^\sharp)}$.
\end{definition}

\begin{lemma}
	\label{lem: split herm cubic norm}
	Let $K = R[t]/(t^2 - t)$ with involution $t \mapsto 1 - t$.
	Suppose that $(J,N,\sharp,T)$ is a hermitian cubic norm structure over $K/R$ and let $J_1 = tJ$, $J_2 = (1 - t) J$, then $(J_1,J_2,N_1,N_2,\sharp,\sharp,T_1, T_2)$ is a cubic norm pair over $R$ using	
	\begin{itemize}
		\item $N_1(j) t = N(j)$ \item $N_2(k) (1 - t) = N(k)$ \item $t T_1(j,k) = T(j,k)$\item  $(1 - t) T_2(k,j) = T(k,j)$ 
	\end{itemize} 
	for $j \in J_1$ and $k \in J_2$.
		Conversely, any cubic norm pair $(J_1,J_2)$ over $R$ defines such a hermitian cubic norm structure $J$ over $K/R$.
		Moreover, the cubic norm pair is well behaved if and only if the hermitian cubic norm structure is well behaved.
	\begin{proof}
		This is not hard to verify.
	\end{proof}
\end{lemma}

\begin{definition}
	We call a hermitian cubic norm structure $J$ over $R[t]/(t^2 - t) = R \times R$ a \textit{split} hermitian cubic norm structure. The previous lemma establishes a one to one correspondence between split hermitian cubic norm structures over $R \times R$ and cubic norm pairs over $R$.
\end{definition}

\begin{construction}
	To a hermitian cubic norm structure $J$ over $K/R$ we associate an $R$-algebra $\mathcal{B} = \mathcal{B}(J,K)$ with involution.
	Set $\mathcal{B} = K \times J$, use $(k,j)(k',j') = (kk' + T(j,j'), kj' + \overline{k'}j + j \times j')$ as product, and $\bar{\cdot} : (k,j) \mapsto (\bar{k},j)$ as involution. Using that $T$ is hermitian and $\times$ is symmetric, one sees that $\overline{xy} = \bar{y}\bar{x}$ for all $x ,y \in \mathcal{B}$.
\end{construction}

Recall the structurable algebra $\mathcal{A}$ for a cubic norm pair constructed in \ref{con: A}.
Lemma \ref{lem: split herm cubic norm} proves that we can define $\mathcal{A}$ for an arbitrary split hermitian cubic norm structure.

\begin{lemma}
	\label{lem: B and A the same for split hcns}
	For a split hermitian cubic norm pair $J$ over $R \times R$ we have $\mathcal{A} \cong \mathcal{B}(J,K)$.
	\begin{proof}
		We will prove that we have an isomorphism for a hermitian cubic norm structure coming from a cubic norm pair $(J_1,J_2)$ over $R$. Lemma \ref{lem: split herm cubic norm} shows this to be sufficient.
		
		The diagonal elements in $\mathcal{A}$ form $R \times R \cong R[t]/(t^2 - t)$. 
		We use
		\[ \alpha t + \beta (1 - t) \mapsto \begin{pmatrix}
			\alpha \\ & \beta
		\end{pmatrix}\]
		as the isomorphism $R[t]/(t^2 - t) \longrightarrow R \times R \subset \mathcal{A}$.
		The elements with zero diagonal form $J$. We see that $tJ = J_1$ and $(1 - t) J = J_2$ by construction. Hence, the hermitian cubic norm operators are listed in Lemma \ref{lem: split herm cubic norm}.
		
		It is easy to verify that the products and involutions of $\mathcal{B}(J,K)$ and $\mathcal{A}$ match.
	\end{proof}
\end{lemma}

\begin{remark}
	\label{rmk: B structurable}
	Lemma \ref{lem: B and A the same for split hcns} shows one advantage of working with hermitian cubic norm structures over cubic norm pairs, as the associated algebra $\mathcal{B}(J,K)$ is easier to describe than $\mathcal{A}$. 
	At this moment, it is only justified to call $\mathcal{B}(J,K)$ structurable if $J$ is a split hermitian cubic norm structure. 
	
	In Theorem \ref{thm: split extension}, we shall establish that $J \otimes K$ is split for all hermitian cubic norm structures. Given that $K$ is faithfully flat, this shows that $\mathcal{B}(J,K) \subset \mathcal{B}(J,K) \otimes K \cong \mathcal{B}(J \otimes K, K \otimes K) \cong \mathcal{A}$.
	We are thus justified in calling $\mathcal{B}(J,K)$ structurable for all hermitian cubic norm structures.
\end{remark}

\begin{theorem}
	\label{thm: split extension}
	Consider a hermitian cubic norm structure $J$ over $K/R$.
	Then, $J \otimes K$ is isomorphic to a split hermitian cubic norm structure over $(K \times K)/K$. 
	Moreover, there exists a $K$-semilinear automorphism $\alpha : J \otimes K \longrightarrow J \otimes K$ of the hermitian cubic norm structure such that $J$ coincides with the fixpoints of $\alpha$.
	\begin{proof}
		This is a straightforward application of Lemmas \ref{lem: quadr etale alg ids}, \ref{lem: quadr etale alg mods}, and \ref{lem: split herm cubic norm}.
	\end{proof}
\end{theorem}

\begin{remark}
	\label{rem: hcns splits}
	The previous theorem says that any hermitian cubic norm structure can be seen as the fixpoints of a split hermitian cubic norm structure under a certain (semilinear) automorphism.
	
	We can apply our knowledge to $\mathcal{B}(J,K)$. In Remark \ref{rmk: B structurable}, we explained how \[\mathcal{B}(J,K) \subset \mathcal{A} \cong \mathcal{B}(J \otimes K, K \times K)\]
	since $K$ faithfully flat.
	The $K$-semilinear automorphism $\alpha$ of Theorem \ref{thm: split extension}, shows that $\mathcal{B}(J,K)$ is formed by the elements of $\mathcal{A}$ fixed under a semilinear automorphism $\tilde{\alpha}$ induced by $\alpha$.
\end{remark}

\begin{construction}
	\label{con: Lie B}	
	Recall the Lie algebra $\mathcal{L}(\mathcal{A})$ constructed in \ref{con: lie}. Since $\mathcal{B} = \mathcal{B}(J,K) \subset \mathcal{A}$, we can apply this Lie algebra construction to $\mathcal{B}$ as well and obtain a subalgebra $\mathcal{L}(\mathcal{B})$ of $\mathcal{L}(\mathcal{A})$. When we apply this construction, we should use the space of skew elements $\mathcal{S} = \{ x - \bar{x} : x \in \mathcal{B} \} = \{x \in \mathcal{B} : x + \bar{x} = 0 \}$ using Lemma \ref{lem: quadr etale alg ids} for the identification. This Lie algebra carries a $\mathbb{Z}$-grading
	\begin{equation}  \mathcal{L}(\mathcal{B}) = \mathcal{S}_{-2} \oplus \mathcal{B}_{-1} \oplus \mathfrak{instr}(\mathcal{B})_0 \oplus \mathcal{B}_1 \oplus \mathcal{S}_2 \label{eq: grading comp Lie b},\end{equation}
	writing $X_i$ for the $i$-th grading component, and is isomorphic to the Lie algebra generated by the Kantor triple subsystem $\mathcal{B}$ of $\mathcal{A}$.
	
	This definition introduces $\mathfrak{instr}(\mathcal{B})$ as a subalgebra of $\mathfrak{instr}(\mathcal{A})$.
	Observe that $\mathfrak{instr}(\mathcal{B}) = \text{span}\{ (V_{x,y}, - V_{y,x}) : x, y \in \mathcal{B} \}$ acts faithfully on $\mathcal{B}^2 \subset \mathcal{A}^2$ since these maps are $K$-linear and since $\mathcal{A}$ is spanned by $\mathcal{B}$ over $K$.

\end{construction}

\subsection{Main theorem}
\label{subsec: main thm hcns}

In the previous subsection we defined hermitian cubic norm structures $J$ over $K/R$, the associated structurable algebra $\mathcal{B} = \mathcal{B}(J,K)$ and the Lie algebra $\mathcal{L}(\mathcal{B}) = \mathcal{L}(\mathcal{B}(J,K))$.

Now, we will define automorphism groups $G^+$ and $G^-$ of $\mathcal{L}(\mathcal{B})$, that can be thought of $\exp(\mathcal{B}_1 \oplus \mathcal{S}_2)$ and $\exp(\mathcal{B}_{-1} \oplus  \mathcal{S}_{-2})$, using the grading components of $\mathcal{L}(\mathcal{B})$ introduced in (\ref{eq: grading comp Lie b}).

\begin{construction}
	\label{con: UKJ}
	We define \[U{(K,J)} = \{((k,j),(u,kj + j^\sharp)) \in (K \times J) \times (K \times J) : u + \bar{u} = k\bar{k} + T(j,j)\}.\]
	This is a group with product
	\[ (a,b)(c,d) = (a + c,b + d + a\bar{c}),\]
	using the algebra structure of $\mathcal{B}(J,K) \cong K \times J$.
	Whenever $J$ is split, this coincides with $G^+$, the automorphism group generated by all the root groups for which the root has positive $\mathbb{Z}$-grading, as proven in Lemma \ref{lem: descr G+}.
	Any element in $G^+ \subset \Aut(\mathcal{L}(\mathcal{A}))$ is thus of the form $1 + g_1 + g_2 + g_3 + g_4$ with $g_i$ acting as $+i$ on the grading of the Lie algebra.
	The isomorphism, was proved by proving that there exist a unique $(x,y) \in U(K,J)$ such that $g_1 = \ad x$ and $g_2 \cdot a_{-1} = (x\bar{a}) x - y a$ for $a \in \mathcal{A}_{-1}$. This is a nice vector group \ref{lem: U nice vector group}.
	
	For an arbitrary hermitian cubic norm structure $J$, this shows that $U(K,J)$ acts as automorphisms of $\mathcal{L}(\mathcal{B} \otimes K)$. We shall prove that it acts on  $\mathcal{L}(\mathcal{B})$.
\end{construction}

\begin{lemma}
	\label{lem: help1}
	Suppose that $u_j = ((0,j),(v,j^\sharp))$ and $u_a = ((a,0),(v,0))$ are elements of $U(K,J)$.
	Write each $g \in U(K,J)$ as $1 + \sum_{i = 1}^4 g_i$ with \[g_i \in \End(\mathcal{L}(\mathcal{B}(K \otimes K, J \otimes K)))\] acting as $+i$ on the grading.
	The equations 
	\begin{equation}
		\label{eq: T((0,j))} 
		(u_j)_3 \cdot (b,k) =N(j)b + T(j,Q_jk) + vT(j,k) - \overline{N(j)b + T(j,Q_jk) + uT(j,k)}.\end{equation}
	and 
	\begin{equation}
		\label{eq: T((a,0))}
		(u_a)_3 \cdot (b,k) = vb \bar{a} - a\overline{bv}.
	\end{equation}
	hold for all $(b,k)$ in the grading component $\mathcal{B}(K \otimes K, J \otimes K)_{-1}$ of the Lie algebra.
	\begin{proof}
		We will show that these equations hold whenever $J$ is split. The general case follows from the embeddings $J \longrightarrow J \otimes K$ and $K \longrightarrow K \otimes K$.

		Write\[ s = \begin{pmatrix}
			1 & \\ & -1
		\end{pmatrix}. \]
		We note that $(u_j)_3$ with $j = \begin{pmatrix}
			0 & v \\ w & 0
		\end{pmatrix}$ is the part of 
		\[ \exp_{\beta + \alpha }(v) \exp_{\beta + 2 \alpha}(w) \exp_{2 \beta + 3\alpha}\left( v_j - T(v,w) \begin{pmatrix}
			1 & \\ & 
		\end{pmatrix}\right),\] that acts as $+3$ on the grading of the Lie algebra. We write $f_j$ for the part of $u_3$ obtained by restricting the domain to $\mathcal{A}_{-1}$.
		Note that $v_j = \begin{pmatrix} (v_j)_1 & \\ & (v_j)_2 \end{pmatrix}$ with $(v_j)_1 + (v_j)_2 = T(v,w)$.
		We obtain
		\begin{itemize}
			\item $f_j \begin{pmatrix}  \\ & 1 \end{pmatrix} = - N(w) s$, using the grading to conclude that only the contribution of $\exp_{\beta + 2 \alpha}(w)$ can be nonzero,
			\item $f_j \begin{pmatrix}
				1 & \\ & 
			\end{pmatrix}$ = $N(v) s$, using the grading to conclude that only the contribution of $\exp_{\beta + \alpha}(v)$ can be nonzero,
			\item $f_j \begin{pmatrix}
				& a \\ &
			\end{pmatrix} = (T(v,Q_w a) - (v_j)_2 T(w,a)) s$, where the nonzero terms are obtained by evaluating $\exp_{\alpha + \beta }(v)_1\exp_{\beta + 2 \alpha}(w)_2$, and \[\exp_{\beta + 2 \alpha}(w)_1 \exp_{2 \beta + 3 \alpha}\left( v_j - T(v,w) \begin{pmatrix}
			1 & \\ & 
		\end{pmatrix}\right)_1\] where $\exp_\gamma(\cdots)_i$ is the part of that exponential that acts on the grading as $+ i \gamma$.
			\item $f_j \begin{pmatrix}
				& \\ b & 
			\end{pmatrix} = (- T(w,Q_v b) + (v_j)_1 T(v,b)) s$.
		\end{itemize}
		From the first two cases, we obtain that $f_j(b,0) = N(j) b - \overline{N(j) b} = [N(j)b,1]$.
		The last two cases yield that $(0,k) \mapsto [T(j,Q_j k) + v_j T(j,k), 1]$ coincides with the action of $f_j$. 
		This establishes that Equation \ref{eq: T((0,j))} holds.
		
		For Equation \ref{eq: T((a,0))}, we immediately get that $T_{((a,0),(u,0))} (0,k) = 0$ if we closely observe the grading. For $T_{((a,0),(u,0))} (b,0)$, one argues analogously to the previous cases.
	\end{proof}
\end{lemma}

\begin{proposition}
	\label{prop: action UKJ}
	Consider a hermitian cubic norm structure $J$ over $K/R$. The group $U(K,J) \subset U(K \times K, J \otimes K)$ acts on $\mathcal{L}(\mathcal{B}) \subset \mathcal{L}(\mathcal{B} \otimes K)$ as automorphisms.
	\begin{proof}
		We can see $U(K,J)$ as an automorphism group acting on $\mathcal{L}(\mathcal{B} \otimes K)$.
		Since $\mathcal{L}(\mathcal{B})$ is generated by $\mathcal{B}_1$ and $\mathcal{B}_{-1}$, it is sufficient to prove that $U(K,J)$ maps these generators to $\mathcal{L}(\mathcal{B}) \subset \mathcal{L}(\mathcal{B} \otimes K)$, to prove that $U(K,J)$ acts on $\mathcal{L}(\mathcal{B})$.
		
		Take $u \in U(K,J)$ and write $u = 1 + u_1 + u_2 + u_3 + u_4$ with $u_i$ acting as $+i$ on the grading of $\mathcal{L}(\mathcal{B} \otimes K)$.
		We know that $u_1 = \ad b$ for a $b \in \mathcal{B}$. Hence $u$ maps $\mathcal{B}_1$ to $\mathcal{L}(\mathcal{B})$.
		We also know that $u_2 \cdot \mathcal{B}_{-1} \in \mathcal{B}$.
		So, if we can show that $u_3 \cdot \mathcal{B}_{-1} \subset \{ x - \bar{x} : x \in \mathcal{B}\}$, then we can conclude that $u \cdot \mathcal{B}_{-1} \subset \mathcal{L}(\mathcal{B})$, which would finalize the proof.
		
		For $u_j$ and $u_a$ as in Lemma \ref{lem: help1} we know that this is the case.
		For arbitrary $u = u_a u_j \in U(K,J)$,
		we use
		\[ u_3 \cdot x = \left( (u_a)_3 + (u_a)_1 (u_j)_2  + (u_j)_3    + (u_a)_2 (u_j)_1 \right) \cdot x. \]
		For the final term $(u_a)_2 (u_j)_1 \cdot x$, we note that $u_a$ and $u_{t j} = ((0,tj),(t^2 u_j, t^2 j^\sharp))$ commute in $U(K[t],J[t])$, and thus $(u_a)_2 (u_j)_1 = (u_j)_1 (u_a)_2$.
		Hence, all terms are of the form $f_3 \cdot x$ or $f_1 g_2 \cdot x$ with $f_3 \cdot x \in \mathcal{L}(\mathcal{B})$ and $[f_1, g_2 \cdot x] \in \mathcal{L}(\mathcal{B})$.		
	\end{proof}
\end{proposition}

\begin{lemma}
	\label{lem: help hom spaces}
	Suppose that $K$ is an $R$-algebra and that $K$ is projective and finitely generated as an $R$-module, then
	\[ \Hom_R(X,Y) \otimes_R K \cong \Hom_K(X \otimes K, Y \otimes K)\]
	for all $R$-modules $X$ and $Y$.
	\begin{proof}
		Since $K$ is finitely generated projective, we have \[\text{Hom}_R(X,Y) \otimes K \cong \text{Hom}_R(X, Y \otimes K),\] which follows from (\cite[Chapter II, No. 4, Proposition 4]{Bou89} with $K \cong \text{Hom}_R(R,K)$. The tensor-hom adjunction shows
		\[\text{Hom}_R(X, Y \otimes K) \cong \text{Hom}_R(X, \text{Hom}_K(K,Y \otimes K)) \cong \text{Hom}_K(X \otimes K, Y \otimes K) .\qedhere\]
	\end{proof}
\end{lemma}

\begin{lemma}
	\label{lem: G embeds into larger G}
	Suppose that $J$ over $K/R$ is a hermitian cubic norm pair, then
	\[ \mathcal{L}(\mathcal{B}) \otimes K \cong \mathcal{L}(\mathcal{B} \otimes K)\]
	and
	\[ \End_R(\mathcal{L}(\mathcal{B})) \otimes K \cong \End_K(\mathcal{L}(\mathcal{B} \otimes K)).\]
Moreover, $\End_R(\mathcal{L}(\mathcal{B})) \otimes L \longrightarrow \End_K(\mathcal{L}(\mathcal{B} \otimes K)) \otimes_K (K \otimes L)$ is injective for all $R$-algebras $L$.
	\begin{proof}
		First note that $K$ has constant rank $2$, which implies that $K$ is finitely generated projective \cite[Proposition 1.4]{Vasc69}. Hence, we can use Lemma \ref{lem: help hom spaces}.
		
		In particular, we see that $\End_R(\mathcal{B}) \otimes K \cong \End_K(\mathcal{B} \otimes K)$, which shows that the surjection $\mathfrak{instr}(\mathcal{B}) \otimes K \longrightarrow \mathfrak{instr}(\mathcal{B} \otimes K)$ is an isomorphism. This proves the isomorphism of Lie algebras.
		
		A second application of Lemma \ref{lem: help hom spaces} and the isomorphism of Lie algebras yields the second isomorphism.
		
		The moreover-statement follows since $K$ is faithfully flat as an $R$-algebra and $ M \otimes_R L \cong M \otimes_K (K \otimes_R L)$ for $K$-modules $M$.
	\end{proof}
\end{lemma}

Lemma \ref{lem: G embeds into larger G} allows us to embed $\Aut_R(\mathcal{L}(\mathcal{B}))(L) \subset \Aut_K(\mathcal{L}(\mathcal{B} \otimes K))(K \otimes L)$. We note that the groups $G^\pm(L)$ are subgroups of $\Aut_R(\mathcal{L}(\mathcal{B}))(L)$, which we prove below.

\begin{lemma}
	The groups $G^\pm(L)$ are subgroups of \[(\End_R(\mathcal{L}(\mathcal{B})) \otimes L)^\times ) \subset (\End_K(\mathcal{L}(\mathcal{B} \otimes K) \otimes_K (L \otimes K))^\times.\]
	Moreover, these define subgroups of $\Aut(\mathcal{L}(\mathcal{B}))(L)$ by applying the map \[\left(\End_R(\mathcal{L}(\mathcal{B})) \otimes L\right)^\times \longrightarrow \End_R(\mathcal{L}(\mathcal{B}) \otimes L)^\times.\]
	\begin{proof}
		This is obviously the case if $L = R$. We also note that $G^\pm(L)$ is a subgroup of $(\End_K(\mathcal{L}(\mathcal{B} \otimes K) \otimes_K (L \otimes K))^\times$, since this is how we constructed $G^\pm$ in the split case.
		
		Consider the generators $l \cdot (g_1 \otimes 1, g_2 \otimes 2)$ and $l \cdot_2 (0,g \otimes 1)$ for all $(g_1,g_2), (0,g) \in G(R)$, and $l \in L$, of $G^\pm(L)$ identified in \cite[Lemma 1.3]{OKP}. Note that these generators have immediate interpretations in $\End_R(\mathcal{L}(\mathcal{B}) \otimes L)^\times$,
		using \[l \cdot_1 (g_1,g_2) = 1 + g_1 \otimes l + g_2 \otimes l^2 + g_3 \otimes l^3 + g_4 \otimes l^4\] and similar for $l \cdot_2 (0, g \otimes 1)$.
		Using the embedding \[\End_R(\mathcal{L}(\mathcal{B})) \otimes L \longrightarrow \End_K(\mathcal{L}(\mathcal{B} \otimes K)) \otimes_K (K \otimes_R L),\] we see that these generators generate a subgroup isomorphic to $G^\pm(L)$.
		This immediately shows that $G^\pm(L)$ defines a subgroup of $\Aut(\mathcal{L}(\mathcal{B}))(L)$ since its generators are automorphisms.
	\end{proof}
\end{lemma}

Recall that we use $G(L)$ to denote the subgroup generated by $G^\pm(L)$.
In the theorem below, we mean by $\Lie(G)$ the kernel of $\epsilon \mapsto 0$ in $G(\Phi[\epsilon]/(\epsilon^2))$.
One way to compute the Lie bracket is $[1 + \epsilon_1 g, 1 + \epsilon_2 h] = 1 + \epsilon_1 \epsilon_2 [g,h]$ over $\Phi[\epsilon_1, \epsilon_2]/(\epsilon_1^2, \epsilon_2^2)$. 

If we write $\Lie(G)(K)$, we mean the analogous subgroup of $G(K[\epsilon]/(\epsilon^2))$.

\begin{theorem}
	\label{thm: hcns lie alg}
	Let $J$ over $K/R$ be a well behaved hermitian cubic norm structure.
	There exist group functors \[R\textbf{-alg} \longrightarrow \textbf{Grp}: L \mapsto G^\pm(L) \subset \Aut_L(\mathcal{L}(\mathcal{B}) \otimes L),\] such that \[\mathcal{L}(\mathcal{B}) \otimes L \cong \Lie(G^-)(L) \oplus L_0 \otimes L \oplus \Lie(G^+)(L)\] with $L_0$ the $0$-graded component of $\mathcal{L}(\mathcal{B})$.
	Moreover, any element $1 + \sum_{i = 1}^\infty g_i$ with $g_i$ being $+i$-graded in $\End(\mathcal{L}(\mathcal{B}) \otimes L)$ that is contained in the subgroup of $\Aut_L(\mathcal{L}(\mathcal{B}) \otimes L)$ generated by $G^+(L)$ and $G^-(L)$, is contained in $G^+(L)$, for all $R$-algebras $L$.
	Furthermore, one has $[\Lie(G),\Lie(G)] \cong \mathcal{L}(\mathcal{B})$.
	\begin{proof}
		The first part is Proposition \ref{prop: action UKJ} with an additional description of $\mathcal{L}(\mathcal{B})$ in terms of the Lie algebras $\Lie(G^\pm)$.
		
		The other claims follow from Theorem \ref{thm: cnp main} since $G(L)$ embeds in the analogous group for the split hermitian cubic norm structure by using that $K$ is faithfully flat and using Lemma \ref{lem: G embeds into larger G} applied to $\mathcal{L}(\mathcal{B}) \otimes L$. We note that $G^\pm(L)$ are the fixpoints of the analogous groups under the semilinear automorphism $\alpha$ of Theorem \ref{thm: split extension}.
	\end{proof}
\end{theorem}

\begin{remark}
	If the base ring is a field and $\mathcal{L}(\mathcal{B})$ is finite dimensional, \cite[Exposition VIB, Proposition 7.1,(i) and Corollary 7.2.1]{demazure1970proprietes} guarantees that $G$ is a smooth connected algebraic group, since both $G^+$ and $G^-$ are smooth and connected (and remain so over all field extensions).
	
	In case this $G$ is reductive, the usual $\mathbb{Z}$-grading defines a cocharacter of $G$. Cocharacters define parabolic groups. The groups $G^\pm$ are the unipotent radicals of a pair of opposite parabolics.
\end{remark}

\section{Operator Kantor pairs and recognition theorems}

In this section, we first define operator Kantor pairs and see that any cubic norm pair defines an operator Kantor pair.
Afterwards we characterize the operator Kantor pairs one can obtain in such a way, using $G_2$-gradings.
Finally, we prove that each hermitian cubic norm structure defines an operator Kantor pair and look at the class of operator Kantor pairs $(G^+,G^-)$ such that $(G^+_L,G^-_L$) can be constructed in such a way for faithfully flat $L$.

\subsection{Operator Kantor pairs from cubic norm pairs}

Recall that we defined nice vector groups in Definition \ref{def: vec group}.
For such a group $G \le A \times B$, we use $G_2$ to denote $\{b \in B : (0,b) \in G\}$.
Note that this is a submodule of $B$ and that $G_2 \otimes K \cong  G_2(K) =\{b \in B \otimes K : (0,b) \in G(K)\}$, since $G$ is nice.

\begin{definition}
	Consider a nice vector group $G \le A \times B$, an associative algebra $C$, and natural transformations $\rho_i : G(K) \longrightarrow C \otimes K$ for all $i \in \mathbb{N}_{>0}$.
	If
	\[ g \mapsto \rho_{[t]}(g) = 1 + \sum_{i = 1}^\infty t^i \rho_{i}(g)\]
	is a group homomorphism into $((C \otimes K)[[t]])^\times$, and
	\begin{itemize}
		\item $\rho_{[t]}(\lambda a, \lambda^2 b) = \rho_{[\lambda t]}(a,b)$ for all $\lambda \in K$,
		\item $\rho_{[t]}(0, \lambda b) = 1 + \sum_{i = 1}^\infty \lambda^i t^{2i} \rho_{2i}(g)$ for all $\lambda \in K$,
		\item $\rho_{[t]}(g) = 1 + \sum_{i = 1}^\infty t^{2i} \rho_{2i}(g)$ then $g \in G_2(K) (\ker \rho_{[t]})$,  
	\end{itemize}
	hold for all $K$, then we call $\rho_{[t]}$ a \textit{faithful vector group representation} of the nice vector group $G$ if $(\rho_1 + \rho_2) : \Lie(G) \longrightarrow A$ is injective. A comparison with \cite[Definition 1.21]{OKP} shows that this is a vector group representation.
\end{definition}

Recall the power series $o_{i,j}(x,y,s^it^j)$ introduced in Definition \ref{def: oij} from power series $x(s)$ and $y(t)$. We will write $g(t)$ for $\rho_{[t]}(g)$ in the next definition for brevity.

\begin{definition}
Consider a pair of nice vector groups $(G^+,G^-)$ and recall that $G^+(K)$ is natural in $K$, and consider natural transformations
\begin{itemize}
	\item $Q^\text{grp} : G^\pm \longrightarrow \text{Nat}(G^\mp, G^\pm)$,
	\item $T :  G^\pm \longrightarrow \text{Nat}(G^\mp, [G^\pm,G^\pm])$,
	\item $P :  G^\pm \longrightarrow \text{Nat}(G^\mp, G^\pm/G^\pm_2)$.
\end{itemize}
For example, to $x \in G^\pm(K)$ we associate a natural transformation $T_x$ such that
\[ T_x^L : G^\mp(L) \longrightarrow [G^\pm(L), G^\pm(L)]\]
for all $K$-algebras $L$. 
Suppose that $(a,b)(-a,b) \in [G^\pm,G^\pm]$ for all $(a,b) \in G^\pm$, and that $s \mapsto Q_{(0,s)}$ is injective over $R$.
We call $(G^+,G^-)$ an \textit{operator Kantor pair} if and only if there exists an associative algebra $A$ and a pair of faithful vector group representations $\rho^\pm$ of $G^\pm$ into $A$, such that
\begin{itemize}
	\item $o_{2,1}(x,y,s^2t) = (Q^\text{grp}_xy)(s^2t)$,
	\item $o_{3,1}(x,y,s^3t) = (T_x y)(s^3t)$,
	\item $o_{k,l}(x,y,s^kt^l) = 1$ if $k/l > 2$ and $k/l \neq 3$,
	\item $o_{3,2}(x,y,\epsilon) = (P_x y)(\epsilon)$ over $K[\epsilon]/(\epsilon^2),$
\end{itemize}
whenever $x \in G^\pm(K)$ and $y \in G^\mp(K)$.
\end{definition}

\begin{remark}
	The definition imposes some implicit homogeneity restrictions on the operators.
	For example, it requires that $Q^\text{grp}_x y$ is homogogeneous of degree $2$ in $x$ and of degree $1$ in $y$, since $o_{2,1}(\lambda \cdot x, \mu \cdot y, s^2t) = o_{2,1}(x,y,(\lambda s)^2 (\mu t))$.
\end{remark}

\begin{lemma}
	\label{lem: which okp are jp}
	Consider nice vector groups $G^\pm \subset A^\pm \times B^\pm$ such that $G^\pm \longrightarrow A^\pm$ is a bijection. Any operator Kantor pair on $(G^+,G^-)$ is a quadratic Jordan pair $(A^+,A^-)$. Conversely, any quadratic Jordan pair defines such an operator Kantor pair.
	\begin{proof}
		The assumption on $G^\pm$ forces $T = 0$.
		The definition of an operator Kantor pair reduces to Equation (\ref{def: dprep}) and an additional assumption on $\nu_{3,2}(x,y)$, using that $o_{k,l}(x,y) = 1$ for all $k > 2l$. Hence, Lemma \ref{lem: dprep is JP} shows that $(A^+,A^-)$ with $Q = Q^\text{grp}$ forms a quadratic Jordan pair.
		
		For the converse, Lemma \ref{lem: exp prop} shows that $P_x y = Q_x Q_y x$ satisfies the the assumption on $o_{3,2}(x,y,\epsilon)$.
	\end{proof}
\end{lemma}

\begin{lemma}
	Any operator Kantor pair in our sense, is an operator Kantor pair in the sense one employed in \cite{OKP}. 
	\begin{proof}
		That our definition implies that it is an operator Kantor pair in the sense of \cite{OKP}, follows from \cite[Theorem 3.28]{OKP}.
	\end{proof}
\end{lemma}

\begin{remark}
	Any operator Kantor pair in the sense of \cite{OKP} has an analog to the TKK representation for Jordan pairs.
	Namely, one can associate a Lie algebra to the operator Kantor pair \cite[Lemma 3.16]{OKP} and vector group representations in the endomorphism algebra of the Lie algebra \cite[Construction preceding Remark 3.19]{OKP}.
	Using this, one can prove that any operator Kantor pair in the sense of  \cite{OKP}, is an operator Kantor pair in our sense. To be precise, this is the verification of the conditions of \cite[Theorem 2.9]{OKP} performed the proof of \cite[Theorem 3.26]{OKP}.
	
	Hence, the definition of an operator Kantor pair of \cite{OKP} is equivalent to the one we use here.
\end{remark}

\begin{theorem}
	\label{thm: cns defines okp}
	Consider a well behaved cubic norm pair, the vector groups $G^\pm$ associated to it, and the action of them on the Lie algebra. 
	The pair $(G^+,G^-)$ with operators $Q^\text{grp} = o_{2,1}, T = o_{3,1}$ and $P = \nu_{3,2}$ forms an operator Kantor pair.
	\begin{proof}
		This follows immediately from Lemma \ref{lem: oij}, using the $o_{i,j}$ as definitions for $Q^\text{grp}$, $T$ and $P$. Now, $G^\pm(K) \subset \Aut(\mathcal{L}(\mathcal{A}))(K)$ for all $K$ lets us recover all the operators.
	\end{proof}
\end{theorem}

\begin{remark}
	In \cite{OKP}, we wrote $Q^\text{grp}_x y = (Q_x y, R_x y) \in G^\pm \le A^\pm \times B^\pm$.
	This operator $Q$ is very important.
	
	For cubic norm pairs, we have $Q_{(x,y)} z = (x\bar{z}) x - y z$ for $(x,y) \in G^+$ and $z \in \mathcal{A}$, which was ``used" to derive Lemma \ref{lem: descr G+}.
	
	We did not provide nice formulas for the operator Kantor pair operators yet.
	Later, we shall give formulas for $Q$ and $T$ in terms of the hermitian cubic norm structure.
	These operators determine the other operators uniquely, cfr. \cite[Remark 3.24]{OKP}, and often in an ``algebraic" way \cite[Lemma 5.2]{OKP}.
\end{remark}

\subsection{Operator Kantor pairs with a $G_2$-grading}
In this subsection we look at operator Kantor pairs with a $G_2$-grading.
First, we put an assumption on the long roots, to characterize operator Kantor pairs coming from cubic norm pairs. Second, we look at what happens if we replace the assumption on the long roots with a different assumption on the short roots.

\begin{definition}
	\label{def: G2 grading}
	\begin{itemize}
		\item 
		Consider a basis $\{\alpha,\beta\}$ of $\mathbb{Z}^2$. 
		As before, the $G_2$ root system associated to $\alpha$ and $\beta$ is $R^+\cup R^- \subset \mathbb{Z}^2$ with
		\[ R^+ = \{ \alpha, \beta, \alpha + \beta, 2 \alpha + \beta, 3 \alpha + \beta, 3 \alpha + 2 \beta\}\]
		and $R^- = - R^+$. For $\alpha$ denote with $\alpha^\vee$ the map $\mathbb{Z}^2 \longrightarrow \mathbb{Z}$ given by $\alpha \mapsto 2$ and $\beta \mapsto - 3$. We also define $\beta^\vee$ using $\alpha \mapsto -1$ and $\beta \mapsto 2$. For other roots $\gamma$, one can define similar maps.
		\item 
		Consider a (not necessarily associative) $\mathbb{Z}^2$-graded algebra $A$.
		If \[A = A_0 \oplus \bigoplus_{r \in G_2} A_r\] for a $G_2$ root system, using $A_i$ to denote homogeneous elements of $A$ of degree $i$, we call the algebra $G_2$-graded. If all $A_r$ are nonzero, we call this grading \textit{proper}.
		To a root $\gamma$, we associate an automorphism $\Phi_\gamma(t)$ of $A[t,t^{-1}]$, for $t \in R[t,t^{-1}]$, that acts as $\Phi_\gamma(t) a = t^{ \gamma^\vee(\delta)} a$ for $a \in A_\delta$. We remark that one can recover the $\mathbb{Z}^2$ grading on $A$ from these automorphisms for $\gamma = \alpha$ and $\gamma = \beta$.
		\item We call an operator Kantor pair \textit{(properly) $G_2$-graded} if	
		its $\mathbb{Z}$-graded Lie algebra (\cite[Construction below Lemma 3.15]{OKP}) is (properly) $G_2$-graded with respect to a refinement of this $\mathbb{Z}$-grading.
		We assume that the conjugation action within the automorphism group of the Lie algebra, induces an action of the grading automorphisms act on the operator Kantor pair.
	\end{itemize}
\end{definition}

\begin{remark}
	Our definition of a $G_2$-graded algebra is related though not equivalent to the usual notion of a $G_2$-graded Lie algebra, as studied, for example, in \cite{Ben96}. There it is required that the Lie algebra $G_2$ is a subalgebra, which does not have to be the case for our $G_2$-graded Lie algebras. 
\end{remark}

\begin{lemma}
	For a properly $G_2$-graded operator Kantor pair, we can choose the $\alpha$ and $\beta$ that define the $G_2$ root system in such a way that $\alpha$ is $0$-graded with respect to the usual $\mathbb{Z}$-grading and $\beta$ is $1$-graded. 
	\begin{proof}
		Recall that the $\mathbb{Z}$-grading of an operator Kantor pair only has nonzero grading components with grading in $\{-2,-1,0,1,2\}$ and that the $1$ or $-1$ grading component is definitely nonzero for nonzero operator Kantor pairs. We may assume that $\alpha$ and $\beta$ have grading $\ge 0$, using the action of the Weyl group. If $\alpha$ is strictly positively graded, the element $3 \alpha + \beta$ has grading $\ge 3$. Hence, $\alpha$ must be $0$ graded. Now, $\beta$ must be $1$-graded.
	\end{proof}
\end{lemma}

Henceforth, we assume that $\alpha$ is $0$-graded and that $\beta$ is $1$-graded.

We will assume that the long roots form a specific operator Kantor pair, which we will introduce now.

\begin{example}
	\label{example P}
	Consider, for a commutative, unital, associative, $R$-algebra $K$, the groups
	\[ P^+(K) = \left\{ \begin{pmatrix} 1 & k_1 & - k_3 \\ & 1 & - k_2 \\ & & 1 \end{pmatrix} \in \text{Mat}_3(K) \right\} \]
	and 
	\[ P^-(K) = \left\{ \begin{pmatrix} 1 &  &  \\ - k_2 & 1 &  \\ - k_3 & k_1 & 1 \end{pmatrix} \in \text{Mat}_3(K) \right\}. \]
	One can verify that $P^+$ and $P^-$ form an operator Kantor pair, by applying the definition. On the other hand $P^+$ and $P^-$ generate a subgroup of $\mathbf{GL}_3(K)$.
	
	One can also describe \[P^+(K) = \left\{ (k_1 t + k_2 \bar{t}, k_3 t + (k_1k_2 - k_3)\bar{t}) \in \left(K[t]/(t^2 - t) \right)^2\right\}\] using $\bar{t} = 1 -t$. Using the algebra structure of $K[t]/(t^2 - t)$ and involution induced by $t \mapsto \bar{t}$, the main operators are given by
	\[ Q_{(x,y)} z = x\bar{z} x - y z\]
	and 
	\[ T_{(x,y)} z = x\bar{z}\bar{y} -  y z \bar{x}\]
	for $(x,y) \in P^+$ and $z \in K[t]/(t^2 - t)$. These operators are, up to a sign for $Q$ (to be precise $(x,y) \mapsto (-x,y)$ is the isomorphism), the ones derived for the associative algebra with involution $K[t]/(t^2 - t)$ in \cite[Subsection 5.2]{OKP}.
	
	We remark that any natural transformation $(P^+,P^-) \longrightarrow (G^+,G^-)$ that preserves the operators $Q$ and $T$, will preserve all the operator Kantor pair operators, since all other operators are uniquely determined by them by \cite[Lemma 5.2]{OKP}.
\end{example}

\begin{remark}
	\label{rem: idempotents}
	The operator Kantor pair $(P^+,P^-)$ contains $3$ copies of the Jordan pair $(R,R)$ with $Q_k l = k^2 l$ for $(k,l) \in (R,R)$.
	The first copy is formed by $((tR,0),(\bar{t}R,0)) \subset (P^+,P^-)$,
	the second by $((\bar{t}R,0),(tR,0))$, and the third by $(0,(t - \bar{t})R), (0,- (t - \bar{t})R))$.
	These subpairs can be thought of as the $(R_\gamma, R_{-\gamma})$ we used for cubic norm pairs.
\end{remark}

If $\rho^\pm_{[t]}$ is a divided power representation of a Jordan pair mapping to $A[t] \subset A[[t]]$, we can consider $\rho^\pm_{[1]}$.

\begin{lemma}
	\label{lem: grading jordan idempt}
	Suppose that the quadratic Jordan pair $(R,R)$ has a divided power representation $\rho^\pm$ in an associative $R$-algebra $A$ and suppose that \[\rho^\pm_{[t]}(R) \in A[t] \subset A[[t]].\]
	Then $A$ is a $\mathbb{Z}$-graded algebra and there exists an $R$-group homomorphism\[ \Phi : R^\times \longrightarrow A^\times \] such that \[ \Phi(t) a \Phi(t^{-1}) = t^i a \text{ whenever $a$ is $i$-graded,} \] defined by \begin{equation} \label{eq: def Phi} \rho^+_{[1]}(k) \rho^-_{[1]}(l) = \rho^-_{[1]}( l/(1 - kl)) \Phi(1 - kl) \rho^+_{[1]}(k / ( 1 - kl))
	\end{equation} whenever $1 - kl$ invertible.
	Moreover conjugation by
	\[ \tau = \rho^-_{[1]}(1) \rho^+_{[1]}(1) \rho^-_{[1]}(1)\]
	reverses this grading, i.e., $\tau \Phi(t) = \Phi(t^{-1}) \tau$.
	\begin{proof}
			We note that $\Phi(1 - st) = b(s,t)$ with $b$ as in \cite[Lemma 9.9]{loos2019steinberg}. If we can apply \cite[Lemma 9.9]{loos2019steinberg}, we have that $\Phi(1 - sy - xy) = b(s + x,y) = b(s, y/(1 - xy)) b(x,y) = \Phi(1 - sy/(1 - xy)) \Phi(1 - xy)$ if $1 - xy$ and $1 - sy/(1 - xy)$ are invertible. Hence, in that case we can conclude that $\Phi$ is a homomorphism.

		We can apply this lemma, since the relations $\mathfrak{B}(s,t)$ (introduced in \cite[9.8]{loos2019steinberg}), which are the requirement to be able to use this lemma, hold.
		First, comparing the defining Equation \ref{eq: def Phi} of $\Phi(1 - kl)$, with the exponential property \cite{faulkner2000jordan} over $A[[k,l]]$, which contains $A[k,l,(1 - kl)^{-1}]$, yields that $\Phi(1 - kl)$ is the image of a specific element in the universal divided power representation. 		
		Now, \cite[Lemma 26]{faulkner2000jordan} shows that this universal representation is isomorphic to $\mathcal{Y} \otimes \mathcal{H} \otimes \mathcal{X}$ as a coalgebra for the free Jordan pair $(R,R)$, with $\mathcal{Y} \cong \mathcal{X}$ the universal binomial divided power representation of $R$.
		 Moreover there exist projections (e.g., $\pi_{\mathcal{Y}} = \text{Id} \otimes \epsilon \otimes \epsilon$) such that $(\pi_{\mathcal{Y}} \otimes \pi_{\mathcal{H}} \otimes \pi_{\mathcal{X}}) \circ \Delta^{(2)} = \text{Id}.$ This shows that any grouplike element in the universal representation is of the form $y h x$ with with $x = \pi_{\mathcal{X}}(yhx)$, $y = \pi_{\mathcal{Y}}(yhx)$, and $h = \pi_{\mathcal{H}}(yhx)$. Moreover, $x, y$ and $h$ have to be grouplike as well.
		 In particular, if $h \in \mathcal{H}$ and $x \in \mathcal{X}$ are grouplike, $h x h^{-1}$ is grouplike as well. The $\mathbb{Z}$-grading forces $h x h^{-1} \in \mathcal{X}$. Now, $\mathcal{X}$ is the distribution algebra \cite[Lemma 12]{Faulkner04} of the affine group scheme $(R,+)$, which forces grouplike $x \in \mathcal{X} \otimes K$ to correspond to $x \in (K,+)$.
		 Applying this discussion over the ring $R[k,l]/(k^nl^n)$ for arbitrary $n$, yields that the power series $\Phi(1 - kl)$ satisfies $\mathfrak{B}(k,l)$, using formulation \cite[9.8.(5)]{loos2019steinberg}.

	Now, we want to prove that the conjugation action of $\Phi(t)$ defines a grading and that the conjugation action of $\tau$ reverses the grading.
		We will write, for convenience, $g \cdot a = g a g^{-1}$ for $g \in A^\times$ and $a \in A$ and similarly use $f \cdot a = f(a)$ for $f \in \End(A)$.
		Now, we obtain a grading of $A$ satisfying the properties we wished for, since $\sum (st)^i a_i = \Phi(st) \cdot a = \sum t^i \Phi(s) \cdot a_i $.
		
		To finalize the proof, we verify that $\tau$ reverses the grading.		
		Consider the equality (which is equivalent to the equality defining $\Phi$ using $\mathfrak{B}(k,l)$) \[\rho^+_{[1]}(k) \rho^-_{[1]}(l) = \rho^-_{[1]}( l/(1 - kl)) \rho^+_{[1]}(k ( 1 - kl))  \Phi(1 - kl)\] over $A[[k,l]]$ and decompose $\rho^\pm_{[1]}(k) \cdot a$ as $\sum_{i \in \mathbb{N}} k^i (f^\pm_i \cdot a)$.
		Comparing graded components for homogeneous $a \in A$ of degree $u$, we obtain for $c \in \mathbb{Z}$ that
		\begin{equation}
			\label{eq: 1}
			\sum_{i - j = c} k^il^j (f^+_i f^-_j \cdot a) = (1 - kl)^{u + c} \sum_{v - w = c} k^vl^w (f^-_w f^+_v \cdot a).
		\end{equation}
		Terms with $u + c \ge 0$ are polynomials in $k$ and $l$ and can be evaluated in $k = l = 1$ and evaluate to $0$ if $u + c > 0$, which yields
		\[ \rho^+_{[1]}(1) \rho^-_{[1]}(1) \cdot a \equiv 0 \mod A_{\le - u}.\]
		Hence $\tau \cdot a \in A_{\le - u}$.
		Playing the same game for $\rho^-_{[1]}(l) \rho^+_{[1]}(k)$ yields that $\rho^-_{[1]}(1)\rho^+_{[1]} \cdot a \equiv 0 \mod A_{\ge - u}$ and thus $\tau \cdot a \in A_{\ge - u}$.
		Combining both observations yields that
		\[ \rho^-_{[1]}(1) \rho^+_{[1]}(1) \rho^-_{[1]}(1) \cdot a \in A_{-u} \qedhere\]
	\end{proof}
\end{lemma}

Now, when we let an operator Kantor pair act on its Lie algebra, Jordan subpairs have divided power representations.
We can thus define a grading automorphism $\Phi(1 - kl)$ of the Lie algebra of the operator Kantor pair from any subpair isomorphic to $(R,R)$.
The operator Kantor pair $(P^+,P^-)$ of Example \ref{example P} has three such pairs $((tR,0),(\bar{t}R,0)),$ $((\bar{t}R,0),(tR,0))$, and $((0, (t - \bar{t})R),((0, (- t + \bar{t})R)))$.

\begin{lemma}
	\label{lem: G2 grading}
	Consider the operator Kantor pair $(P^+,P^-)$ of Example \ref{example P}.
	Let $(G^+,G^-)$ be an operator Kantor pair over $R$ and suppose that there is a morphism $\alpha : (P^+,P^-) \longrightarrow (G^+,G^-)$ of operator Kantor pairs. Assume that the grading automorphism $\Phi_t(1 - kl)$ corresponding to the subpair $(tR,\bar{t}R)$ stabilizes $G^+, G^-$ under the conjugation action, assume that $\alpha[P^\pm,P^\pm] = G^\pm_2$, and assume that
	\[Q^\text{grp}_{(0,t - \bar{t})} : G^- \longrightarrow G^+\] is bijective.
	The operator Kantor pair $(G^+,G^-)$ is $G_2$-graded, i.e., its Lie algebra is $G_2$-graded and the grading automorphisms stabilize the groups $G^\pm$.	
	Moreover, the grading automorphisms are the grading automorphisms induced by $(P^+,P^-)$.
	\begin{proof}
			From the pair $([P^+,P^+],[P^-,P^-]) \cong (R,R)$, we start by recovering the usual grading of the Lie algebra associated to $(G^+,G^-)$.
		On the $\pm 2$-graded part of the Lie algebra associated to the operator Kantor pair (\cite[Lemma 3.16]{OKP} and \cite[Lemma 3.20]{OKP} for the action), which is $\alpha[P^\pm,P^\pm]$, the grade reversing $\tau$ acts as $Q^\text{grp}_{(0,\pm(t - \bar{t}))}$, since it does so for the pair $(P^+,P^-)$.
		From the definition of the grading associated to $(R,R)$ and Equation \ref{eq: def Phi}, we obtain that $\psi(t) \Phi(1 - kl) = \Phi(1 - kl) \psi(t)$ using the usual grading morphism $\psi(t)$ for arbitrary $t \in K^\times$ invertible. Hence, both gradings are compatible.
		For any $-1$-graded $x$ which is homogeneous of degree $i$ with respect to $\Phi(1 - kl)$, we apply Equation \ref{eq: def Phi} to $x$ and obtain
		\[ x + k Q^+_{(0,t - \bar{t})} x = (1 - kl)^{i}\left(x + kl Q^-_{(0,-(t - \bar{t}))}Q^+_{(0,t - \bar{t})} x\right) + (1 - kl)^{i + 1} k Q^+_{(0,t - \bar{t})} x,\]
		using that $Q^+_{(0,t - \bar{t})} x$ is homogeneous of degree $i + 2$.
		Using that $Q^+_{(0,t - \bar{t})}$ is injective, we derive $i = -1$ and 
		\[ Q^-_{(0,-(t - \bar{t}))}Q^+_{(0,t - \bar{t})} x = -x.\]
		Using that $\tau$ reverses the grading, one easily computes that $\tau$ is given by $Q^\pm_{(0,\pm(t - \bar{t}))}$ on $\mp 1$-graded elements.	
		Moreover, this $\tau$ induces an isomorphism \[Q^\text{grp}_{(0, t - \bar{t})} : G^- \longrightarrow G^+.\] Thus, it also interchanges the subpairs $(\alpha (tR,0),\alpha (\bar{t}R,0))$ and $(\alpha(\bar{t}R,0),\alpha(tR,0))$ of $(\alpha P^+,\alpha P^-)$ in $\Aut(L(G^+,G^-))$.
		
		Moreover, $(\alpha (tR,0),\alpha (\bar{t}R,0))$ also defines a grading on the Lie algebra associated to $(G^+,G^-)$ by Lemma \ref{lem: grading jordan idempt}. The $\tau'$ obtained from such pairs also interchanges the other pairs (at least on the Lie algebra level). For example, using $d$ to denote, at first, the undetermined $0$-graded part, the left hand side of Equation \ref{eq: def Phi} applied to $(t - \bar{t})_2$ becomes
		\begin{align*}
			\rho^+_{[1]}(k t)\rho^-_{[1]}(l \bar{t})  \cdot  (t - \bar{t})_2& = \rho^+_{[1]}(k t) \cdot \left( l^2 d - l \bar{t}_1 + (t - \bar{t})_2 \right) \\
			& = l^2 d - l \bar{t}_1 + (1 - kl) (t - \bar{t})_2,
		\end{align*}
	since $d$ maps to a trivial derivation of $(P^+,P^-)$ under the restriction map. This shows that $(t - \bar{t})_2$ is $1$-graded with $\tau (t - \bar{t})_2 = -\bar{t}_1$. This also shows that $d = 0$.
		This implies that, on the Lie algebra, all three gradings are conjugate to the ordinary $\mathbb{Z}$-grading.
		
		Hence, the $\tau$ corresponding to subpairs act transitively on the identified subpairs (in the Lie algebra).
		Moreover, the grading induced by the subpair containing $(tR,0)$ is conjugate to the usual $\mathbb{Z}$-grading.
		So, fix the grading obtained from the pair containing $[P^+,P^+]$ and the pair containing $(tR,0)$ and the corresponding grading automorphisms $\Phi_1(k)$ and $\Phi_2(k)$ over $K[k,k^{-1}]$. We say that $x$ is $(a,b)$-graded if $x$ is $a$ graded with respect to the usual grading and $b$ graded with respect to $\Phi_2(k)$.
		
		These two gradings combined form a $G_2$-grading. 
		 To be precise, we have grading components $(2,1), (1,2), (-1,1), (1,-1), (-2,-1), (-1,-2)$ (which are the long roots) coming from the subspaces of $(P^+,P^-)$. The only other possible grading components are $(0,0), (1,0), (1,1), (0,1), (-1,0), (-1,-1), (0,-1)$ since each component $(i,j)$ must have $i,j \in \{-2,-1,0,1,2\}$ with $j$ or $i$ being equal to $\pm 2$ implying that the other is $\pm 1$ since only the parts of $P^+$ can have grading $\pm 2$ in either grading (as they are conjugate to $\alpha [P^+,P^+]$ or $\alpha [P^-,P^-]$ under the $\tau$ obtained from the subpairs of $(P^+,P^-)$).
		
		Now, the assumption on $\Phi_t(1 - kl)$ and the usual grading automorphism show that all grading automorphisms, and in particular the ones associated to the roots $\alpha$ and $\beta$, stabilize $G^+$ and $G^-$.
	\end{proof}
\end{lemma}

In the previous lemma, we identified $G_2$-gradings coming from embeddings of $(P^+,P^-)$ as in Example \ref{example P}. The conditions precisely identify a $G_2$-grading where long roots form an operator Kantor pair $(P^+,P^-)$. A necessary and sufficient condition for the long roots to form an operator Kantor pair $(P^+,P^-)$ is, thus, having an embedding, $G^\pm_2 \cong P^\pm_2$, the grading automorphisms of $(P^+,P^-)$ having an action on $(G^+,G^-)$, with one grading automorphism being the usual $\mathbb{Z}$-grading.

\begin{lemma}
	\label{lem: subsystems}
	Let $(G^+,G^-)$ be a $G_2$-graded operator Kantor pair such that the long roots form an operator Kantor pair $(P^+,P^-)$ as in Example \ref{example P} and suppose that the $G_2$-grading is induced by the grading automorphisms of the subpairs of $(P^+,P^-)$.
	We can rewrite $G^+(K)$ as
	\[ \{ ((a\bar{t}, j, j', b t),u) \in \left((R\bar{t} \oplus J \oplus J' \oplus Rt) \times R[t]/(t^2 - t)\right) \otimes K : u + \bar{u} = a\bar{b} + T(j,j') \} \]
	for a certain bilinear $T$ and certain submodules $J$ and $J'$ of $\Lie(G^+)$ such that
	\[ \Phi_{\beta}(s)   ((a\bar{t}, j, j', b t),u) \Phi_{\beta}(s^{-1}) = ((s^2 a\bar{t}, s j, j', s^{-1 }b t), su).\]
	Furthermore, $(G^+,G^-) \cong (G^-,G^+)$ as operator Kantor pairs under
	\[ \tau = (Q^\text{grp}_{(0,-(t - \bar{t}))} ,Q^\text{grp}_{(0,t - \bar{t})} ).\]
	\begin{proof}
		We remark that $\tau$ is precisely the grading reversing automorphism associated to the long root.
		We prove that $\tau$ preserves the operator Kantor pair operators.
		Using that $P_{(0,t - \bar{t})}$ and $T_{(0,t - \bar{t})}$ are of degree $3$ and thus $0$, \cite[Definition 3.21, conditions 2 and 4]{OKP} show that 
		\[ \tau Q_y \tau^{-1} = - Q_{(0,(t - \bar{t}))}  Q_y Q_{(0,- (t - \bar{t}))} = Q_{\tau y}\]
		and
		\[ \tau T_y \tau^{-1} = - Q^\text{grp}_{(0,(t - \bar{t}))} T_y Q_{(0,- (t - \bar{t}))} = T_{\tau y}\]
		for $y \in G^-$ (writing $Q^\text{grp}_{(0,(t - \bar{t}))} \cdots = R_{(0,(t - \bar{t}))} \cdots$, as $R$ coincides with $\tau$ on $\pm 2$-graded elements). The same holds for $x \in G^+$. Since $P$ and $Q^\text{grp}_x y = (Q_x y , R_x y)$ can be recovered from $Q$ and $T$, cfr. \cite[Remark 3.25]{OKP}, we conclude that $\tau$ preserves all operator Kantor pair operators.	
		
		Now, let $J$ and $J'$ be the subspaces of the Lie algebra with the right grading.
		Define $T(v,w)(t - \bar{t}) = [(v,\cdot),(w, \cdot)]$, using the group commutator of $G^+$.
		We now construct for each $j \in J'$ a unique $j \in G^+$ that interacts correctly with the grading. 
		Take arbitrary $(j,a) \in G^+$ and note that \[\Phi_{\beta}(t) (j,a) \Phi_{\beta}(t^{-1}) = (j,j_2 + t s + \sum_{i \neq 1,2} t^i c_i)\] over $K[t,t^{-1}]$.
		Now, consider
		\[ (j,j_2 + t s + \sum_{i \neq 1,2} t^i c_i) (-j, -a + \psi(j,j)) = (0, (j_2 - a) + ts +  \sum_{i \neq 1,2} t^i c_i)\]
		and reapply $\Phi_\beta(u)$.
		Applying this directly establishes
		\[ (\Phi_{\beta}(tu) \cdot (j,a)) (\Phi_{\beta}(u) \cdot (-j,-a)) = (0, u(j_2 - a + ts + \sum_{i \neq 1,2} t^i c_i) \]
		Recomputing this using the definitions shows that $c_i = 0$ for all $i$, that $j_2 + s = a$, and $s \in G^+_2$.
		Hence, $(j,j_2)$ is fixed under $\Phi_\beta(t)$. The found $j_2$ is unique, as it corresponds to the unique element $(j,j_2)(0,s)$ that remains fixed.
		Similar argumentation works for all root subgroups.
	\end{proof}
\end{lemma}

Recall that we constructed an operator Kantor pair using a well behaved cubic norm pair.
This construction could make sense for arbitrary cubic norm pairs, i.e., we could look at cubic norm pairs for which the automorphism groups $G^+$ and $G^-$ form an operator Kantor pair.

\begin{lemma}
	\label{lem: well behaved}
	Each cubic norm pair $(J,J')$ that defines an operator Kantor pair is well behaved.
	\begin{proof}
		That the cubic norm pair is well behaved follows from \cite[Definition 3.21.4]{OKP}.
	Namely, one has
	\[ Q^\text{grp}_{(j,j^\sharp)} (t,0) = (j^\sharp, N(j)j) \]
	for $j \in J$ and $t \in R$ by construction.
	One also has $T_{(j,j^\sharp)} (t,0) = (0 , N(j) (t - \bar{t}))$.
	One can also compute \[P_{(j,j^\sharp)}(t,0)  = 2 Q_{T_{(j,j^\sharp)}(t,0)} (t,0) - (V_{j,t})_1 Q_{(j,j^\sharp)} (t,0) = -2 N(j)t + 3N(j) t = N(j)t\] using \cite[Lemma 5.2.3]{OKP}.
	Substituting this in \cite[Definition 3.21.4]{OKP} and evaluating on $\bar{t}$ yields
	\[ - N(j^\sharp) (t - \bar{t}) + N(j)^2(t - \bar{t}) = 0\]
	since the other terms are $0$ by the grading. We conclude that $(J,J')$ is well behaved.
	\end{proof}
\end{lemma}

\begin{theorem}
	\label{thm: char 1}
	Consider the operator Kantor pair $(P^+,P^-)$ of Example \ref{example P}.
	Let $(G^+,G^-)$ be an operator Kantor pair over $R$ and suppose that there is an injective morphism $\alpha : (P^+,P^-) \longrightarrow (G^+,G^-)$ of operator Kantor pairs and that the grading automorphism corresponding to the subpair $(tR,\bar{t}R)$ stabilizes $G^+$ and $G^-$ under the conjugation action. Suppose, additionally, that $[G^+,G^+] = \alpha[P^+,P^+]$, that the same holds for $G^-$ and $P^-$, and that \[Q^\text{grp}_{(0,t - \bar{t})} : G^- \longrightarrow G^+\] is bijective. In this case, the operator Kantor pair can be constructed from a well behaved cubic norm pair.
	Conversely, any operator Kantor pair arising from a cubic norm pair is of this form.
	\begin{proof}
		Note that the cubic norm pair has to be well behaved by Lemma \ref{lem: well behaved}, if it exists. The converse statement is more or less obvious, as the long roots form an operator Kantor pair isomorphic to $(P^+,P^-)$. The stabilization condition follows from Theorem \ref{thm: cnp main} and the bijectivity condition holds since $L_{t - \bar{t}}$ is bijective.
		
		Recall that we have a $G_2$-graded algebra and grade reversing automorphisms for the long roots by Lemma \ref{lem: G2 grading}. Lemma \ref{lem: subsystems} finds subgroups $J$, $J'$ and $P^+$ of $G^+$. Using the homogeneity of the grading, $J' \subset G^+$ and $J \subset G^-$ form an operator Kantor subpair, as both are groups with weight $\pm ( \beta + 2 \alpha)$. Moreover, the pair $J \subset G^+$, $J'\subset G^-$ forms an isomorphic subpair due to Lemma \ref{lem: subsystems}.		
		Lemma \ref{lem: which okp are jp} establishes that $(J,J')$ is a quadratic Jordan pair.
		
		We can consider the action of this pair $(J',J)$ on the spaces with grading $-\beta - 3 \alpha, - \alpha , \alpha + \beta , \beta + 3 \alpha$ and note that we can recover operators $N, \sharp, T$ so that this representation coincides with the one of Example \ref{ex: dprepr}, which implies, by Lemma \ref{lem: necessity def}, that $(J',J)$ with these operators forms a cubic norm pair over $K$. The first condition to apply that lemma holds, since $\rho^+_1(j) \cdot 1_{- \beta - 3 \alpha} = - \tau_{\beta + 3 \alpha}(j)$ for $j \in J'$. The second condition to apply the lemma can be checked using the matrix realization of $P^+$ and $P^-$ of Example \ref{example P} and letting it act on $J'_{\beta + 2 \alpha}  \oplus J'_{- \alpha} \oplus J'_{- \alpha - \beta}$ and $J_{- \beta - 2 \alpha} \oplus J_{\alpha} \oplus J_{\alpha + \beta} $.
		Computing $\tau_{\beta + 3 \alpha} \tau_{2\beta + 3 \alpha}$ in this realization, with $\tau_{\gamma}$ the grade reversing automorphism of Lemma \ref{lem: grading jordan idempt} corresponding to a subpair of $(P^+,P^-)$ with grading $\pm \gamma$, yields 
		\[ \tau_{\beta + 3 \alpha} \tau_{2\beta + 3 \alpha} = \begin{pmatrix} & 1 \\ & & 1 \\ 1
		\end{pmatrix}, \]
		which shows that
		\begin{align*}  - \tau_{\beta + 3 \alpha} \tau_{2\beta + 3 \alpha} \cdot T'(j',j) t_{\beta + 3 \alpha} &  = - \tau_{\beta + 3 \alpha} \tau_{2\beta + 3 \alpha} \cdot [j'_{\beta + 2 \alpha}, j_{\alpha}]\\ & = [ j_{- \beta - 2 \alpha}, j'_{- \alpha - \beta}] \\ &= \tau_{\beta} [j_{-\beta - 2 \alpha}, j'_{-\alpha}] \\ &= \tau_\beta T(j,j') \bar{t}_{- \beta - 3 \alpha} \\ & = T(j,j') (t - \bar{t})_{- 2 \beta - 3 \alpha } \\ & = - \tau_{\beta + 3 \alpha} \tau_{2\beta + 3 \alpha} \cdot T(j,j') t_{\beta + 3 \alpha},  \end{align*}
		and thus $T(j,j') = T(j',j)$. 
		The final condition $j^{\sharp\sharp} = N(j)j$ to apply Lemma \ref{lem: necessity def}, holds as well. Namely, write $(j,j^\sharp)$ for elements $j \in J' \subset G^+$.
		We note that $Q^\text{grp}_{(j,j^\sharp)} (\bar{t},0) = (j^{\sharp}, j^{\sharp\sharp})$ lying the analogous subgroup $J \subset G^+$. We note
		\[ Q_{(j^{\sharp}, j^{\sharp\sharp})} t = j^{\sharp\sharp}\]
		(this is the same $\sharp$, since $\tau_{2 \beta + 3 \alpha}$ is an isomorphism).
		On the other hand, \[\nu_{4,2}((j,j^\sharp),(\bar{t},0)) \cdot t_{-1} + \nu_{1,1}((j,j^\sharp), (\bar{t},0)) \nu_{3,1}((j,j^\sharp), (\bar{t},0)) \cdot t_{-1} = 0\]
		so that
		\[ \nu_{4,2}((j,j^\sharp),(\bar{t},0)) \cdot t_{-1} = \ad [j_1, \bar{t}_{-1}] \circ N(j) L_{t - \bar{t},2} \cdot t_{-1} =  N(j)[\zeta_{\beta + 3 \alpha}, j_1] = N(j)j_1.\]
		We conclude that
		\[ N(j)j = Q_{(j^{\sharp}, j^{\sharp\sharp})} t = j^{\sharp\sharp}.\]
		
		We now finalize the proof, by showing that the actions of $G^+$ and $G^-$ on the Lie algebra of the operator Kantor pair coincide with the analogous action of the analogous groups for the cubic norm pair.
		 First note that $(G^+,G^-) \supset (J,J') \cong (J,J') \subset (G^-,G^+)$ under $\tau$.
To verify that the Lie algebras are isomorphic, note that the zero-graded components of the Lie algebra act faithfully on the rest of the Lie algebra and that these are spanned by $(V_{v,w}, -V_{w,v})$ for \[(v,w) \in J' \times J \subset G^+ \times G^-\] and by the elements in $V_{P^+,P^-}$. The actions of these $(V_{v,w}, - V_{w,v}) = - \ad [v,w] = - [\ad v, \ad w]$ match on all non-zero-graded components by construction\footnote{We use the sign convention of this article, since $V_{v,w} x = - Q^{(1,1)}_{x,v} w$ in this article, while $V_{v,w} x = Q^{(1,1)}_{x,v} w$ in \cite{OKP}.}, and the isomorphism $\tau$ relates $V_{v,w}$ to $V_{\tau w, \tau v}$ using \[\tau (V_{v,w}, - V_{w,v}) \tau^{-1} = \tau [\ad w, \ad v] \tau^{-1} = \ad [ \tau \cdot w, \tau \cdot v] = ( - V_{\tau w, \tau v}, V_{\tau v, \tau w}).\] Since $V_{v,w}$ and $\tau$ have the same action on both algebras, this shows that the zero-graded components coincide.
		This immediately shows that $G^+$ and $G^-$ act on the Lie algebra of the operator Kantor pair as the analogous groups $G^+$ and $G^-$ for the cubic norm pair act on their Lie algebra, once again using the isomorphism $\tau$ to relate the actions of the short roots.
	\end{proof}
\end{theorem}

\subsubsection{A different kind of $G_2$-grading}

The requirement that the long roots must correspond to the operator Kantor pair $(P^+,P^-)$ is quite strong, but allows us to characterize operator Kantor pairs over $R$ with a $G_2$-graded Lie algebra that can be obtained from a cubic norm pair over $R$. On the other hand, if $K$ is a commutative, associative, unital $R$-algebra and $(J,J')$ a cubic norm pair over $K$, then the operator Kantor pair over $K$ induces an operator Kantor pair over $R$.
So, there are more operator Kantor pairs over $R$ that can be ``constructed" from a cubic norm pair.
When we start thinking of $G_2$-graded operator Kantor pairs in a sense similar to \cite{Torben}, we can almost guarantee that every $G_2$-graded operator Kantor pair comes from a cubic norm structure over an $R$-algebra $K$ (not simply pair). The only problem: if there is $3$-torsion the long roots may not form an associative algebra $K$. 

\begin{remark}
		Consider a properly $G_2$-graded operator Kantor pair $(G^+,G^-)$. 
	Its Lie algebra $L$ has a decomposition, $L = L_0 \oplus \bigoplus_{\gamma \in G_2} L_\gamma$.
	Since the grading automorphisms have an action on $G^\pm$, one can find vector subgroups $\tilde{L}_{\gamma}$ of $G^+$ and $G^-$ with $\tilde{L}_\gamma \cong (L_{\gamma}, +)$ as groups, on which the grading automorphisms act the same.
	
	To be precise, for each $l \in L_\gamma \otimes K$ with $\gamma$ positively graded there exists $(l,c) \in G^+(K)$.
	Take a grading automorphism $\psi(t)$ such that $\psi(t)(\ad l) \psi(t^{-1}) = t^i (\ad l)$ and $\psi(t) (\ad b) \psi(t)^{-1} = t^j \ad b$ for $b \in L_{2 \beta + 3 \alpha} \cong G^+_2$ for certain $i,j \in \mathbb{Z}$, then
	\[ \psi(t) (l,c) \psi(t^{-1}) = (t^i l, t^{2i} l_2 + t^j s + \sum_{k \neq 2i,j} t^k c_k)\]
	over $K[t,t^{-1}]$. One can show that $c = l_2 + s$, that all $c_k$ are $0$, and that $s \in G^+_2$, as in Lemma \ref{lem: subsystems}.
	The $l_2$ is unique if $j \neq 2i$, i.e., if $\psi(t)$ does not act as the usual $\mathbb{Z}$-grading automorphism on $(l,b)$.
\end{remark}

\begin{definition}	We call a properly $G_2$-graded operator Kantor pair \textit{root graded} if each of the individual operator Kantor subpairs $(\tilde{L}_\gamma, \tilde{L}_{-\gamma}, Q)$ can be seen as $(J,J,Q)$ for a Jordan algebra $(J,Q)$ with unit $1$. This means that there exists $1 \in \tilde{L}_{\pm \gamma}$ with $Q_1$ invertible and $Q_1 1 = 1$.
	We also assume that the grading automorphism associated to $\gamma$ coincides with the grading automorphism of Lemma \ref{lem: grading jordan idempt} for the Jordan subpairs $(R 1_{\gamma}, R 1_{-\gamma})$.
\end{definition}

\begin{remark}
	The definition above is necessary for the projective elementary group of an operator Kantor pair to be $G_2$-graded in the sense of \cite{Torben}.
	Namely, $(\tilde{L}_\gamma, \tilde{L}_{-\gamma}, Q)$ will always form a Jordan pair, but not necessarily a Jordan algebra. The existence of a $\gamma$-``weyl-element", which is required in \cite[Definition 2.5.2]{Torben} is asking that the quadratic Jordan pair has a unit, since the $\tau_\gamma$ of Lemma \ref{lem: grading jordan idempt} reverses the grading of the Lie algebra, which implies that $\tau_\gamma$ permutes the root groups. 	
	Namely, any $\gamma$-``weyl-element" is an automorphism
	\[ \tau_\gamma = \exp_{- \gamma}(u) \exp_\gamma(v) \exp_{- \gamma}(w),\]
	which acts as the reflection corresponding to $\gamma$ on the root spaces.
	Restricting to $L_{-\gamma} \oplus L_0 \oplus L_\gamma$ shows that $v$ is invertible \cite[Paragraph 1.10]{loos1975jordan}. One can check that $u = w = v^{-1}$.
	So, \cite[Proposition 1.11]{loos1975jordan} implies that $(L_\gamma, L_{- \gamma})$ is the Jordan pair derived from a Jordan algebra.
	
	Naturally, it is also sufficient, since the existence of ``weyl-elements" is the only part of the \cite[Definition 2.5.2]{Torben} which is not immediately satisfied.
\end{remark}

\begin{lemma}
	\label{lem: long roots}
	Suppose that $(G^+,G^-)$ is a root graded $G_2$-graded operator Kantor pair. There exists an alternative, commutative, unital algebra $K$ such that long roots are parametrized by $K$.
	Moreover, if $K$ is associative, the long roots form an operator Kantor pair as in Example \ref{example P}.
	\begin{proof}
		Let $\gamma_1, \gamma_2, \gamma_3$ be the long roots for which the corresponding root space is contained in $G^+$.
		We can assume that $\gamma_2$ corresponds to $G^+_2$. We also write $G_\alpha$ for the subgroup of $G^+$ or $G^-$ corresponding to the root $\alpha$.
		We write $\exp_{\gamma_i}$ for the action of these root groups on the Lie algebra.
		
		We note that the $E = \exp_{\gamma_i}(e) = 1 + e_{\gamma_i} + e_{2\gamma_{i}}$ with $e_{k\gamma i}$ acting as $k \gamma_i$ on the grading. The $e_{2\gamma_{i}}$ is uniquely determined by $e_{\gamma_i}$, the grading, and $E$ having to be an automorphism. 
		If $\gamma = \pm \gamma_2$, this is obvious since there are no $+4$-graded derivations on the Lie algebra of an operator Kantor pair. For other $\gamma$, note that all gradings with respect to long roots are conjugate to the usual $\mathbb{Z}$-grading.
		In particular, this shows that $\tau_\gamma E \tau_{\gamma}^{-1}$ lies in another root group for any grade reversing automorphism $\tau_\gamma$ associated to a short or long root, as it is $1 + (\tau_\gamma \cdot e_{\gamma_i}) + (\tau_\gamma \cdot e_{2\gamma_{i}})$. 
		
		We use the grade reversing isomorphism $\tau_{\gamma_2} : G_{\gamma_2} \longrightarrow G_{- \gamma_2}$ corresponding to the unit of the Jordan algebra $(G_{\gamma_2},G_{- \gamma_2})$. This $\tau_{\gamma_2}$ is also the what we use to identify $G_{- \gamma_3} \longrightarrow G_{\gamma_1}$. We also identify $G_{- \gamma_3} \longrightarrow G_{ \gamma_2}$ using the isomorphism $\tau_\delta$ corresponding to the unit for the specific short root $\delta$ that induces such an action.
		
		We will need that $\tau_\delta^2(x) = -x$ for $x \in G_{-\gamma_3}$ to prove that the multiplication, which we will define shortly, is commutative.
		To prove this, note that
		\[ \tau_\delta (x) = \exp_{- \delta }(1) \exp_{+ \delta}(1) \exp_{- \delta}(1) \cdot x = \exp_{- \delta }(1) \exp_{+ \delta}(1) \cdot x,\]
		which implies
		$\exp_{- \delta}(1) \cdot (x + [1_\delta,x] + (1_\delta)_2 \cdot x + \tau_\delta \cdot x) = \tau_\delta \cdot x$, by decomposing $\exp_{-\delta}(1)$ into components $(1_\delta)_i$ that are $i \delta$-graded.
		Therefrom, one can derive that $[1_{- \delta }, \tau_\delta \cdot x] = - (1_\delta )_2 \cdot x$ and $[1_\delta ,x] = (1_{-\delta })_2(1_{\delta })_3 \cdot x $, $(1_{-\delta })_3 \cdot \tau_\delta  x = - x$. Using that $\tau_\delta ^2 x = (1_{-\delta })_3 \cdot \tau_\delta  x$ yields what we needed.
		
		From $- [\cdot,\cdot] : G_{\gamma_2 } \times G_{- \gamma_3} \longrightarrow G_{\gamma_1}$ we obtain a multiplication $\mu : K \times K \longrightarrow K$. We remark that $[\cdot, \cdot] : G_{ -\gamma_2} \times G_{\gamma_1} \longrightarrow G_{ -\gamma_3}$ yields the same multiplication, using the identification $\tau_\gamma$.
		Furthermore, it is easy to check that the short root isomorphism fixes $G_{\gamma_1}$, so that \[ab = \tau_\alpha \mu (a, b) = - [\tau_\alpha a, \tau_\alpha b] = \mu(b,a) = ba\] using $a \in \Lie(G_{\gamma_2})$ and $b \in \Lie(G_{-\gamma_{3}})$. Hence, $K$ is commutative.
		
		Now, we can use the Jordan pair character of $(G_{\gamma_2}, G_{- \gamma_2})$ and the divided power representation on the Lie algebra.
		The divided power representation shows that
		\[ (Q_a b) c = a(b(ac)),\]
		for all $a,b,c \in K$ using the action on $c \in G_{- \gamma_3}$.
		So $(a(ba)) c = (a(b(ac)))$ for all $a,b,c$. In particular, we note that $a^2 c = a(ac)$, i.e., $K$ is left alternative.
		
		Since $K$ is left alternative and commutative, it follows that $K$ is alternative.
		
		If $K$ is associative, we use the typical embeddings of the root groups in $P^\pm(K)$.
		The operators $Q$ and $T$ agree on root groups. We also note that $V_{a,b} c = \pm abc$ whenever $a,b$ have opposite gradings that are both distinct from the grading of $c$. This shows that the $Q$ operators agree everywhere. Hence, all operators agree everywhere.
	\end{proof}
\end{lemma}

\begin{remark}
	Using \cite[Theorem 10]{faulkner1983stable}, one can show that there also exist operator Kantor pairs $P^\pm(K)$ for arbitrary alternative $K$. However, this is not relevant for the argument we want to make.
\end{remark}

Commutative alternative algebras which are non-associative are quite rare.

\begin{lemma}
	\label{lem: com alt alg}
	Suppose that $K$ is a commutative, alternative algebra.
	For each $a,b,c \in K$, we have $3((ab)c - a(bc)) = 0$.
	Hence, $K$ is associative if $K$ has no $3$-torsion.
	\begin{proof}
		This is \cite[Exercise 14.7]{Skip2024}.	
	\end{proof}
\end{lemma}

\begin{remark}
	There exists non-associative, commutative, alternative algebras over fields of characteristic $3$. See, e.g., \cite[Section 2]{GRI11}.
\end{remark}

\begin{proposition}
	Suppose that $(G^+,G^-)$ is a root graded $G_2$-graded operator Kantor pair over $R$ and that $G^+_2$ does not have $3$-torsion, then the operator Kantor pair can be constructed from a well behaved cubic norm pair. Moreover, the cubic norm pair is unital.
	\begin{proof}
		From Lemmas \ref{lem: com alt alg} and \ref{lem: long roots}, we obtain an associative, commutative unital algebra $K$ such that the long roots form correspond to the operator Kantor pair formed constructed in Example \ref{example P}.
		
		So, if we can show that our original operator Kantor pair over $R$ can be seen as an operator Kantor pair over $K$, Theorem \ref{thm: char 1} implies that $(G^+,G^-)$ can be constructed from an cubic norm pair, since the grading automorphism $\Phi(t)$ has an action on the operator Kantor pair for $G_2$-graded operator Kantor pairs.
		
		In order to show this, we first prove that the pair of Example \ref{example P} induces an action of $\text{Mat}_3(K)$ on the subspace of the Lie algebra corresponding to the short roots and thereafter we use this to see that the Lie algebra is $K$-linear. To start, there are two orbits of short roots under $(P^+,P^-)$ and spaces corresponding to short roots in the same orbit are isomorphic. So, calling one such space $J$, we can realize $P^+$ and $P^-$ as matrices of the form of Example \ref{example P} in $\text{Mat}_3(\End_R(J))$. 
		Using the existence of well defined $k E_{ij} \in \text{Mat}_3(\End_R(J))$ for all $k \in K$, and $(k E_{ij})(l E_{jm}) = kl E_{im}$ if $\{i,j,m\} =\{1,2,3\}$, we will construct a morphism $\text{Mat}_3(K) \longrightarrow \text{Mat}_3(\End_R(J))$, using $(a E_{ij})(E_{ji})$ as the definition for $a E_{ii}$. This yields a well defined morphism since $(a E_{ij}) E_{ji} = (a E_{ij}) E_{jk} E_{ki} = (a E_{ik}) E_{ki}$ and that $(a E_{ij})(b E_{ji}) = (a E_{ij})(b E_{jk})E_{kj} = (ab E_{ik}) E_{ki} = ab E_{ii}$ whenever $\{i,j,k\} = \{1,2,3\}$.
		
		We remark that multiplication by $k$ on a given copy of $J$ is given by $k E_{ii} = (k E_{ij})( E_{ji})$ and corresponds, thus, to $(\ad k)(\ad 1)$, for the subpair of $(P^+,P^-)$ corresponding to $(K E_{ij}, K E_{ji})$. 
		
		Hence, we can define a scalar multiplication on the short roots given by the action of $k \cdot I_3$. On the long roots the scalar multiplication is self-evident (and coincides with the scalar multiplication on $\text{Mat}_3(K))$. We can also rewrite the action on a space corresponding to $K E_{ik}$; the scalar multiplication with $l$ is given by $(\ad l E_{ij})(\ad E_{ji})$ whenever $\{i,j,k\} = \{1,2,3\}$.
		
		Lastly, on the trivially graded elements $\delta$, which can be identified with their adjoint action on the other non-trivially graded spaces, we act as $(k \cdot \delta)(u) = \delta( k \cdot u)$ for $u$ non-trivially graded.
		
		This definition shows that brackets involving at least one element of $P^+$ and $P^-$, are $K$-linear. For other brackets $[x,y]$ (with non-trivially graded result), one can express the scalar multiplication as $k \cdot [x,y] = (\ad k)(\ad 1) [x,y]$ with $(k,1)$ in a Jordan subpair of $(P^+,P^-)$ with trivial action on $y$, so that $k \cdot [x,y] = [k\cdot x,y]$ (another subpair shows $k \cdot [x,y] = [x,k \cdot y])$.
		
		Lastly, if $[x,y]$ is trivially graded and $x$ and $y$ are contained in spaces corresponding to short roots, we should check that $[[x,y], k \cdot u] = [[x,k \cdot  y],u]$ for nonzero graded $u$.
		If $u$ has gradings different from the grading of $x$ and $y$, this follows from the Jacobi identity and the previous cases. 
		We also note that $[k \cdot x,y] - [x, k \cdot y] = 0$, since $k \cdot x = [[k,1],x]$ and $- k \cdot y = [[k,1],y]$ using $k$ and $1$ in opposite root groups of $P^\pm$ with $[1,y] = 0$ and $[k,x] = 0$.
		If $u$ has the same grading as $x$, we have $[[x,y], k \cdot u] = [[x,k \cdot u],y] + [x,[y,k \cdot u]] =[x,[y,k \cdot u]] = [x,[k \cdot y, u]] = [[x, k \cdot y], u]$.
		
		Now, we can use that the scalar multiplication $(k,g) \mapsto k \cdot g$ on the vector group $G^+$ coincides with the endomorphism $Q^\text{grp}_{(0,k(t - \bar{t}))}	 Q^\text{grp}_{(0,t - \bar{t})}$ on root groups, and thus everywhere. From this scalar multiplication, it follows that the maps $(k \cdot x)_i$, which are the parts of $\exp(k \cdot x)$ that act as $+i$ on the grading, are $K$-homogeneous of degree $i$, as they are $R$-homogeneous of degree $i$ and can be expressed using $(\ad k(t - \bar{t}))$ (which is $K$-linear) and \[f(k) : G^-_2 \longrightarrow G^+_2 : l(\bar{t} - t) \mapsto k^2 l (t - \bar{t})\] which is $K$-quadratic.
		Namely, one can show that
		\[(k \cdot x)_i = \sum_{a + b = i} (-1)^b k_a \tau(x)_i k_b\]
		with $k_1 = \ad k(t - \bar{t} )$ , $k_2 = f(k)$ and $k_3 = 0 = k_4$ and $\tau$ the grade reversing automorphism associated to $(t - \bar{t})$.
		This implies that the operators $Q$ and $T$ are $K$-homogeneous. Thus, all operators are homogeneous.
		
		Finally, we note that the cubic norm pair is unital by construction, since we used a $\tau_\delta$ for a short root $\delta$ to identify $G_{-\gamma_3}$ and $G_{\gamma_2}$ in Lemma \ref{lem: long roots}.
	\end{proof}
\end{proposition}

\begin{remark}
	If there is $3$-torsion, the long roots still have an action on the short roots.
	This action of the long roots is ``associative", or more precisely, it factors through an associative quotient $R'$.
	So, the short roots form a cubic norm structure over $R'$.
	
	Moreover, by playing around as in the lemma we defined $K$, one can show that $kT(v,w) = T(kv,w) = T(v,kw)$. This shows that $T$ maps to the nucleus. Commuting actions of long and short roots also force $N$ to be in the nucleus. 
\end{remark}

\subsection{Operator Kantor pairs from Hermitian cubic norm structures}
We first describe the operator Kantor pair operators $Q$ and $T$ in terms of the hermitian cubic norm structure. Recall that we defined the groups $G^+$ and $G^-$ in Construction \ref{con: UKJ}.

We will write $[b,c] = b\bar{c} - c \bar{b}$, noting that  $(0,[b,c]) = [(b,b_2),(c,c_2)]$ for elements $(b,b_2)$ and $(c,c_2)$ in $G^\pm$.
For an operator Kantor pair $(G^+,G^-,Q^\text{grp},T,P)$, we often use operators $Q$ and $R$ such that \[Q^\text{grp}_g h = (Q_{(a,b)} (c,d), R_{(a,b)} (c,d)) \in G^+\] for $(a,b) \in G^+$, $(c,d) \in G^-$

\begin{theorem}
	\label{thm: ops}
	Each well behaved hermitian cubic norm structure $J$ over $K/R$ defines a unique operator Kantor pair $(G^+,G^-,Q^\text{grp},T,P)$ with \[G^\pm \cong \{((a,b),(u, ab + b^\sharp )) \in \mathcal{B} \times \mathcal{B} : u + \bar{u} = a\bar{a} + T(b,b)\},\] and operators	
	\begin{itemize}
		\item $Q_{(x,y)} v = (x\bar{v}) x - y v$,
		\item $T_{((0,j),(u,j^\sharp))} (b,k) = [N(j)b + T(j,Q_jk) + uT(j,k), 1]$,
		\item $T_{(a,0),(u,0)} (b,k) = [ub \bar{a},1],$
	\end{itemize}  
	writing $Q_j k$ for $- j^\sharp \times k + T(j,k) j$, for elements $(x,y) \in G^\pm,$ $v \in \mathcal{B}$, $j \in J$, $a,u \in K$.
	\begin{proof}
		First, we try to compute $T_{g}$ for arbitrary $g \in G^+$.
		Write $g \in G^+$ as a product $g_a \cdot g_j = ((a,0),(u_a,0)) \cdot ((0,j),(u_j, j^\sharp)$ with $g_a$ and $g_j$ group elements.
		One has 
		\[ T_g  x= T_{g_a} x + [g_a, Q_{g_j} x ] + [g_j, Q_{g_a} x] + T_{g_j}x\]
		since $g_a$ and $g_j$ commute, using the linearisation $T^{(1,2)}_{(g,h)} k= [g, Q_h k]$ of $T$ \cite[Definition pre-Kantor pair]{OKP}. Hence $T_g$ can be computed from the given formulas using the structure of the group $G^\pm$.
		
		The equations for $T$ match the right hand sides of Equations (\ref{eq: T((0,j))}) and (\ref{eq: T((a,0))}).
		The formula for $Q$ matches with the action $(x,y)_2 : \mathcal{B}_{-1} \longrightarrow \mathcal{B}_1$, as explained in the construction of $U(K,J)$ \ref{con: UKJ}.		
		For split hermitian cubic norm structures, Theorem \ref{thm: cns defines okp} yields an operator Kantor pair satisfying the conditions of this theorem, given that \[o_{i,1}((x,y),(u,v),t) = 1 + t \nu_{i,1}((x,y),(u,v)) \mod t^2\] for $(x,y) \in G^+$ and $(u,v) \in G^-$ with $\nu_{i,1}((x,y),(u,v)) = \ad \left( (x,y)_i \cdot u_{-1} \right)$ (see Lemma \ref{lem: nu well def} for a way to compute the $\nu$).
		
		We remark that all these operators can be expressed in terms of hermitian cubic norm structure operations.
		On the other hand \cite[Lemma 5.2]{OKP} guarantees the uniqueness of the operator Kantor pair.
	\end{proof}
\end{theorem}

\begin{remark}
	\label{rmk: T OKP}
	In the case that we consider a hermitian cubic norm structure which is anisotropic over a field, the operators $T$ and $Q$ were already determined.
	
	A comparison with the formula for $T$ in \cite[subsection 5.4.3]{OKP} shows that we have obtained the same operator $T$.
	Namely, if one puts $a = 0$ in the formula, one obtains that 
	\[T_{((0,j),(u,j^\sharp))} (b,k) = [N(j)b + T(j^\sharp \times k,j) + u T(k,j), 1]\]
	which equals \[ = [N(j)b + T(j,- j^\sharp \times k + T(j,k)j) + (u - T(j,j))T(k,j)].\]
	Using that $[(u - T(j,j))T(k,j), 1] = [- \bar{u} T(k,j), 1] = [ u T(j,k), 1]$ since $u + \bar{u} = T(j,j)$,
	we obtain that both $T$ are equal.
\end{remark}

\subsection{Operator Kantor pairs coming from hermitian cubic norm structures after base change}

For an operator Kantor pair $(G^+,G^-)$ over $R$ and $L \in R\textbf{-alg}$ we use $(G^+_L,G^-_L)$ to denote the operator Kantor pair with underlying (nice) vector groups $G^+(L)$ and $G^-(L)$ and the obvious operators.

Recall that a hermitian cubic norm structure is well behaved if $N(v)^2 = \overline{N(v)}$ or, equivalently, if the associated cubic norm pair is well behaved.

\begin{assumption}
	\label{assumption: final subsection}
	We assume in this subsection that we work with an operator Kantor pair $(G^+,G^-)$ over $R$ such that $(G^+_L,G^-_L)$ can be defined from a (well behaved) hermitian cubic norm structure $J_L$ over $K_L/L$ for a faithfully flat $R$-algebra $L$.
	We also assume that
	\[(1,0) \in \Lie(G^+) \subset \Lie(G^+) \otimes L \cong \mathcal{B}(J_L,K_L) \times \mathcal{S}_L\]
	and 
	\[(1,0) \in \Lie(G^-) \subset \Lie(G^-) \otimes \mathcal{B}(J_L,K_L) \times \mathcal{S}_L.\]
\end{assumption}

	The assumption that $L$ is faithfully flat, guarantees that $G^\pm(K) \longrightarrow G^\pm_L(K \otimes L)$ is injective for all $R$-algebras $K$. This is true by Lemma \ref{lem: ff for vec group}.

We will work towards proving:

\begin{theorem}
	\label{thm: hcns} If Assumption \ref{assumption: final subsection} holds, $(G^+,G^-)$ can be defined from a (well behaved) hermitian cubic norm structure.
\end{theorem}

Denote with $\mathcal{B}_L$ the structurable algebra associated to the operator Kantor pair over $L$ and identify $\mathcal{B}_L$ with $G^+_L/(G^+_2)_L \cong \text{Lie}(G^+_L)/(G^+_2)_L$. Note that $\Lie(G^+_L) \cong \mathcal{B}_L \oplus (G^+_L)_2$.
We use $\mathcal{A}$ to denote the similar direct summand of $\Lie(G^+)$, which we do not assume to be an algebra.

We can and will use the operators $Q$ and $T$ of Theorem \ref{thm: ops} on $G^\pm_L$.

\begin{lemma}
	Define $K$ as the span of all \[Q_{(1,\alpha)} 1 \in \Lie(G^+)/G^+_2 = \mathcal{A}\] for $(1,\alpha) \in G^+$. This $K$ coincides with $K_L \cap \mathcal{A}$.
	This $K$ is a quadratic étale extension of $R$. Moreover, $\mathcal{A}$ is a $K$-module.
	\begin{proof}
			We will make use of the isomorphism $K_L \otimes_L K_L$ identified in Lemma \ref{lem: quadr etale alg ids}.
		
		We want to use that $K_L = \langle \alpha, Q_{G^+_{2,L}} 1 \rangle$ for $\alpha \in K_L$ with trace $1$, i.e., $\alpha$ such that $(1,\alpha) \in G^+$. This space is definitely part of $K_L$, as it is $\langle \alpha, s | s \in K_L : \bar{s} = -s \rangle$. To observe that $K_L$ is not bigger, note that $\langle (\alpha , 1 - \alpha) , (k, -k) | k \in K_L \rangle$ is a spanning set of $K_L \times K_L \cong K_L \otimes_L K_L$, using that $K_L$ is faithfully flat over $L$.
		
		Hence, we define $K$ as the span of all $Q_{(1,\alpha)} 1$ with $(1,\alpha) \in G^+$ and note that $K \otimes L \cong K_L$. 
		Observe that $Q_{(1,\bar{\alpha})} \beta = \beta - \bar{\alpha} \beta = \alpha \beta$ in $K_L$ whenever $\beta \in K_L$, and thus $\alpha \beta \in K = K_L \cap \mathcal{A}$ whenever $\beta \in K$.
		We also remark that $K$ is closed under conjugation, since $Q_{(1,\alpha)} 1= \bar{\alpha}$ and $(-1) \cdot (1,\alpha)^{-1} = (1, \bar{\alpha})$.
		
		This definition immediately establishes that $K$ is projective \cite[Remark 25.5]{Skip2024} since $K_L$ is projective. 
		Now we want to show that $K$ is quadratic étale. We remark that $K(\mathfrak{p}) = K \otimes R(\mathfrak{p})$ embeds into $K_L(\mathfrak{q}) = (K \otimes L) \otimes_L L(\mathfrak{q}) \cong K \otimes_RL(\mathfrak{q})$ whenever $R(\mathfrak{p})$ embeds into $L(\mathfrak{q})$ since $K$ is flat. Since $L$ is faithfully flat, one can find such a $\mathfrak{q}$ for each $\mathfrak{p}$ \cite[Proposition 25.9]{Skip2024} (or the different formulation \cite[Exercise 9.26]{Skip2024} hinted at in \cite[Remark 25.10]{Skip2024}).
		This shows that $K(\mathfrak{p})$ has dimension $2$ as $K(\mathfrak{p}) \otimes_{R(\mathfrak{p})} L(\mathfrak{q})$ has dimension $2$ over $ L(\mathfrak{q})$. It also shows that  $K(\mathfrak{p})$ is either a quadratic field extension of $R(\mathfrak{p})$ or isomorphic to $R(\mathfrak{p})^2$. Hence $K$ is quadratic étale.

		Observe that $\mathcal{A}$ is a $K$-module, since $\mathcal{A}_L$ is a $K$-module and since
		\[ \mathcal{A} \ni Q_{(1,\bar{\alpha})} x = \alpha x\]
		for all $x \in \mathcal{A}$ and generators $\alpha$ of $K$.
	\end{proof}
\end{lemma}

\begin{lemma}
	$\mathcal{A} = K \oplus J$ with $J = J_L \cap \mathcal{A}$.
	\begin{proof}
		
		Note that each $a \in \mathcal{A}_L$ can be written as a direct sum of a $k \in K_L$ and $j \in J_L$.
		Furthermore, we know that $J_L$ is contained in the kernel of $x \mapsto T_{(1,\alpha)} x$ for all $(1,\alpha) \in G^+$ by the formula for $T$ in \ref{thm: ops}. 
		We also remark that $T_{(1,\alpha)} x = \alpha x - \bar{\alpha} \bar{x}$ for $x \in K_L$. Writing $\lambda$ for the trace $x + \bar{x}$ of $x$ when $x \in K_L$, shows us that
		\[ T_{(1,\alpha)} x + \lambda \bar{\alpha}  = \alpha x - \bar{x} \bar{\alpha} + \lambda \bar{\alpha}  = \alpha x - \bar{\alpha} \bar{x} + x \bar{\alpha} + \bar{x} \bar{\alpha}  = x. \] We construct a function $f$ on $\mathcal{A}_L$ which coincides with the trace $x \mapsto x + \bar{x}$ on $K_L$ and is $0$ on $J_L$. Note that the $\tilde{f}(x) \in \End_L(\mathcal{S}_L)$ defined by
		\[ T^{(1,2)}_{(1,\alpha),(0,s)} x = [1, sx] = - \bar{x}s - sx = \tilde{f}(x) s,\]
		acts as left multiplication with $f(x)$, since $s$ and $x$ commute if $x \in K_L$ and since $sx = -\bar{x}s$ if $x \in J_L$.
		This also defines $\tilde{f}(x) \in \End_R(\mathcal{S}) \subset \End_L(\mathcal{S}_L)$ for $x \in \mathcal{A}$.
		Now, $\End_R(\mathcal{S}) \longrightarrow \End_K(\mathcal{S} \otimes K) \cong K$ is injective since $K$ is faithfully flat. Hence, $\tilde{f}(x)$ can be represented as $s \mapsto f(x) \cdot s$ for some $f(x) \in K$ if $x \in \mathcal{A}$. 
		Moreover, this $f(x)$ is fixed under $f(x) \mapsto \overline{f(x)}$. Hence, $f(x) \in R$ by Lemma \ref{lem: quadr etale alg ids}. 
		
		We obtain a projection $\mathcal{A}_L \longrightarrow K_L$ mapping
		\[ k_l + j_l \mapsto T_{(1,\alpha)} (k_l + j_l) + f(k_l + j_l) \alpha = k_l\]
		when $k_l \in K_L$ and $j_l \in J_L$, that restricts to a projection $\mathcal{A} \longrightarrow K$.			
		Hence, each $a \in \mathcal{A}$ can be uniquely written as $k + j$ with $k \in K$ and $j \in J = J_L \cap \mathcal{A}$.
	\end{proof}
\end{lemma}

\begin{lemma}
	The $R$-module $J$ is a $K$-module.
	The restriction of $\sharp$ maps $J$ to $J$. The restrictions of $T$ and $N$ map to $K$.
	\begin{proof}
		That $J$ is a $K$-module follows since $J_L$ is a $K$-module and $\mathcal{A}$ is a $K$-module.
		
		From the embedding $G^+$ into $G^+_L$, we learn that each element in $G^+$ is of the form \[((a,j),(u,aj + j^\sharp))\]
		with at least $(a,j) \in \mathcal{A}$.
		Observe that
		\[ Q_{((0,j),(u,j^\sharp))} 1 = j^\sharp - u\]
		with $j^\sharp \in J = J_L \cap \mathcal{A}$ and $u \in K = K_L \cap \mathcal{A}$.
		By considering $((a,0),(u,0))$, we learn that $G^\pm$ can be defined as a vector group in $\mathcal{A}^2$.
		From the group operation, we can recover $T$ using \[((0,j)(u,j^\sharp))((0,k),(v,k^\sharp)) = ((0,j+k),(u + v + T(j,k), (j + k)^\sharp).\]
		
		We can recover the norm $N(j) = N_1 t + N_2 (1 - t)$ in case $K \cong R[t]/(t^2 - t)$ from 
		\[ T_{((0,j),(u_j,j^\sharp))} t = N_1 t - N_1 (1 - t)\]
		and \[  T_{((0,j),(u_j,j^\sharp))} (1 - t) = N_2 (1 - t) - N_2 t.\]
		Hence $N(j)$ lies in $K \subset K_L$.
		For more general $K$, we can extend scalars to $K$ and obtain an operator Kantor subpair of $(G^+_{L \otimes K}, G^-_{L \otimes K})$ with $L \otimes K$ faithfully flat. We consider $N(j) \in K_L$ and its image $(N(j), \overline{N(j)})$ in $K_L \otimes K \cong L \otimes (K \otimes K) \cong K_L \times K_L$, which lies in the kernel of $(a,b) \mapsto a - \bar{b}$. On the other hand, we can consider $N(j) \in K \otimes K$ and its image in $L \otimes (K \otimes K) \cong K_L \times K_L$. This $N(j)$ should also lie in the kernel of $(a,b) \mapsto a - \bar{b}$, hence $N(j) \in K \subset K \otimes K$.
	\end{proof}
\end{lemma}

\begin{proof}[Theorem \ref{thm: hcns}]
	The previous three lemmas yield a quadratic étale algebra $K$, a $K$-module $J$, and maps $N, \sharp$, and $T$ satisfying the hermitian cubic norm structure axioms. Moreover, the operators $Q$ and $T$ have to match the operators of Theorem \ref{thm: ops}.
\end{proof}

\appendix

\section{Vector groups}
\label{ap: vec groups}

Let $A$ and $B$ be $R$-modules and consider a bilinear map \[\psi : A \times A \longrightarrow B.\]
Then, $A \times B$ forms a group under $(a_1,b_1)(a_2,b_2) = (a_1 + a_2, b_1 + b_2 + \psi(a_1,a_2))$.

\begin{definition}[\cite{OKP} Definition 1.1]
	We call a subgroup $G \le A \times B$ an \textit{almost vector group} if
	\begin{enumerate}
		\item for each $(a,b) \in G$ and $\lambda \in R$, $(\lambda a, \lambda^2b) \in G$,
		\item for each $(0,b) \in G$ and $\lambda \in R$, $(0, \lambda b) \in G$.
	\end{enumerate}
	
	The first two conditions of $G$ define \textit{scalar multiplications} on $G$ and $G \cap B$, namely
	\[ \lambda \cdot_1 (a,b) = (\lambda a, \lambda^2 b)\quad  \text{and} \quad \lambda \cdot_2 (0,b) = (0, \lambda b).\]
\end{definition}

\begin{definition}	
	For arbitrary $K$ we define $\hat{G}(K)$ as the smallest almost vector group in $(A \times B) \otimes K$ containing $G \otimes 1$. 
	This is natural in $K$.	
	In \cite{OKP} an almost vector group is said to be a \textit{vector group} if the natural map \[\{ b \in B : (0,b) \in G\} \otimes K \longrightarrow \{(0,b) \in B \otimes K : (0,b) \in \hat{G}(K)\}\] is surjective. 
	We remark that $\{b \in B :(0,b) \in G\}$ is often denoted as $G_2$.
	In that case, the condition to be a vector group becomes: the natural map $G_2 \otimes K \longrightarrow \hat{G}(K)_2$ is surjective.
\end{definition}

\begin{definition}	
	
	We use $\Lie(G)$ to denote $(\epsilon a, \epsilon b) \in \hat{G}(R[\epsilon]/(\epsilon)^2)$.
	This is a Lie algebra with Lie bracket
	\[ [(\epsilon a, \cdot ), (\epsilon b, \cdot)] = (0, \epsilon (\psi(a,b) - \psi(b,a))).\]
	We note that $\Lie(G) \cong A \oplus G_2$ \cite[Lemma 1.7]{OKP}.
\end{definition}

\begin{definition}
	\label{def: nice cover}
	We call a vector group $G$ \textit{nice} if the natural map \[\Lie(G) \otimes K \longrightarrow \Lie(\hat{G}(K))\] is a bijection.
	A vector group $G \le A \times B$ is \textit{proper} if it is nice or there exists a nice vector group $H \le C \times D$ and linear $(f_1 \times f_2): C \times D \longrightarrow A \times B$ such that $H \cong G$.
	We call such a $H$ a \textit{nice cover}.
	
	For a nice vector group $G$ we write $G(K)$ for $\hat{G}(K)$. For a proper vector group we write $G(K)$ for $\hat{H}(K)$. We will immediately prove that this notion of properness and the notion of \cite[Definition 1.16]{OKP} are equivalent and that $G(K)$ is independent of a choice of $H$.
\end{definition}

\begin{lemma}
	Suppose that $G$ is proper and let $H$ be a nice cover.
	The definition of $G(K)$ is independent of $H$ and $G$ is proper in the sense of \cite{OKP}.
	\begin{proof}
		We recall that $G$ is proper in the sense of \cite{OKP} if the kernel $L$ of 
		\[ F(S_1,S_2) \longrightarrow G,\]
		with $F(S_1,S_2)$ the vector group constructed in \cite[Lemma 1.15]{OKP} using certain sets of generators $S_1$ and $S_2$, is a vector group instead of just an almost vector group.
		We note that $F(S_1,S_2) \le F_1 \times F_2$ is nice, as its Lie algebra is proved to be a direct summand of $F_1 \times F_2$.
		
		Now, suppose that $G$ is a vector group with a nice cover $H$. The only thing we need is that $H \cong G$ as abstract groups, $G_2 \cong H_2$, and that the isomorphism is compatible with scalar multiplications. This guarantees that $F(S_1,S_2) \longrightarrow G \cong H$ is defined from maps $S_1 \longrightarrow H$ and $S_2 \longrightarrow H_2$.
		
		We thus have an exact sequence of groups
		\[ 0 \longrightarrow L \longrightarrow F(S_1,S_2) \longrightarrow H \longrightarrow 0.\]
		Observe that $L$ does depend on $G$ and remains the same for all $H$.
		After extending scalars, we obtain an exact sequence
		\[ 0 \longrightarrow \tilde{L}_K \longrightarrow \widehat{F(S_1,S_2)}(K) \longrightarrow \hat{H}(K) \longrightarrow 0\]
		with $\tilde{L}_K$ an almost vector group containing $\hat{L}(K)$,
		and thus
		\[ 0 \longrightarrow \Lie(\tilde{L}_K) \longrightarrow \Lie(\widehat{F(S_1,S_2)}(K)) \longrightarrow \Lie(\hat{H}(K)) \longrightarrow 0.\]
		Using $\Lie(\widehat{F(S_1,S_2)}(K)) \cong \Lie(F(S_1,S_2)) \otimes K$ and $\Lie(\hat{H}(K)) \cong \Lie(H) \otimes K$, we see that $\Lie(L) \otimes K \longrightarrow \Lie(\tilde{L}_K)$ is surjective. Hence, $\tilde{L}_K \cong \hat{L}(K)$. Now, we immediately see that $L$ is a vector group since $L_2 \otimes K \longrightarrow \tilde{L}_{K,2}$ is surjective.
		We conclude that $G$ is proper in the sense of \cite{OKP}.
		
		We also obtained that $\hat{H}(K) \cong \widehat{F(S_1,S_2)}(K)/\hat{L}(K)$. Since the generating sets $S_1$ and $S_2$ depend on $G$ and not on $H$, and since $L$ is defined from $G$ alone, we see that $G(K) = \hat{H}(K)$ does not depend on $H$ (nor on $S_1$ and $S_2$).
	\end{proof}
\end{lemma}

\begin{lemma}
	\label{lem: ff for vec group}
	Suppose that $L$ is faithfully flat and $G$ is a proper vector group, then $G(K) \longrightarrow G(K \otimes L)$ is injective.
	\begin{proof}
		We may assume that $G \le A \times B$ is nice.
		The result follows directly from the injection $(A \times B) \otimes K \longrightarrow ((A \times B) \otimes K) \otimes L$, using that $L$ is faithfully flat.
	\end{proof}
\end{lemma}

\begin{lemma}
	\label{lem: U nice vector group}
	Suppose that $\mathcal{B} = K \times J$ is a structurable algebra associated to a hermitian cubic norm structure $J$ over $K/R$.
	The group
	\[ U = \{((k,j),(u,kj + j^\sharp)) \in (K \times J)^2 : u + \bar{u}= k\bar{k} + T(j,j)\}\]
	is a nice vector group.
	\begin{proof}
		Note that $\hat{U}(L)$ is $U$ for the hermitian cubic norm structure $J \otimes L$ over $K \otimes L$. 
		This implies $\Lie(\hat{U}(L)) \cong \{ (\epsilon b, (\epsilon u,0)) \in \mathcal{B}_L \times (K \otimes L \times J \otimes L) : u + \bar{u} = 0\}$.
		We conclude that $U$ is nice if and only if
		\[ \{ x \in K : x + \bar{x} = 0\} \otimes L \cong \{ x \in K \otimes L: x + \bar{x} = 0\}.\]
		Taking the tensor product with the faithfully flat $K$, yields that it is an isomorphism if and only if 
		\[ \{ (a,-a) \in K^2\} \otimes L \cong \{ (b,-b) \in (K \otimes L)^2\},\]
		which is the case, as both are isomorphic to $K \otimes L$.
	\end{proof}
\end{lemma}

\section*{Declarations}

\subsection*{Conflict of interest} The author has no conflict of interest to declare that is relevant to this article.

\subsection*{Data availability} Not applicable, as no datasets were generated or analyzed.
	
	\bibliographystyle{plain}
	\bibliography{bib}
	
\end{document}